\documentclass[letterpaper,10pt]{amsart}

\usepackage{epsfig}
\usepackage{amssymb}
\usepackage{amsfonts}
\usepackage{amsmath}
\usepackage{graphicx}
\usepackage{mathrsfs}
\usepackage{stmaryrd}
\usepackage{amsthm}
\usepackage[top=1in, bottom=1in, left=1.25in, right=1.25in]{geometry}

\newtheorem{theorem}{Theorem}[section]
\newtheorem{lemma}{Lemma}[section]
\newtheorem{proposition}{Proposition}[section]
\newtheorem{corallary}{Corollary}[section]

\newtheorem{condition}{Condition}

\title{Dirichlet-to-Neumann map for Poincar\'{e}-Einstein Metrics in Even Dimensions}
\author{Fang Wang}
\date{}
\begin{document}
\begin{abstract}
We study the linearization of the Dirichlet-to-Neumann map for Poincar\'{e}-Einstein metrics in even dimensions on an arbitrary compact manifold with boundary. By fixing a suitable gauge, we make the linearized Einstein equation elliptic. In this gauge the linearization of the Dirichlet-to-Neumann map appears as the scattering matrix for an elliptic operator of 0-type, modified by some differential operators. We study the scattering matrix by using the 0-calculus and generalize a result of Graham for the case of the standard hyperbolic metric on a ball. 
\end{abstract} 
\maketitle

\section[]{Introduction}\label{intro}
Let $M$ be the interior of a compact $(n+1)$-dimensional smooth manifold $\overline{M}$ with non-empty boundary $\partial \overline{M}$. Let $x$ be a fixed smooth defining function for
$\partial \overline{M}$, i.e. $x\in C^{\infty}(\overline{M})$, $x\geq 0$ on $\overline{M}$, $x^{-1}(0)=\partial \overline{M}$ and $dx|_{\partial \overline{M}}\neq 0$. Let $m\in \mathbb{N}_0$ and $0<\alpha<1$. We say a metric $g_+$ on $M$ is $C^{m,\alpha}$ (resp. $C^{m}$, $C^{\infty}$) \textit{conformally compact} if $x^2g_+$ extends continuously to a $C^{m,\alpha}$ (resp. $C^{m}$, $C^{\infty}$) metric on $\overline{M}$, denoted by $\overline{g}$. For a $C^{0}$ metric $\gamma$ on $\partial \overline{M}$, the \textit{conformal class} of $\gamma$ is the set of all $C^{0}$ metircs on $\partial\overline{M}$ which is conformally equivalent to $\gamma$, denoted by $[\gamma]$, i.e.
$$[\gamma]=\{e^{2u}\gamma: u\in C^{0}(\partial\overline{M})\}.$$
If $g_+$ is $c^{m,\alpha}$ conformally compact, $[x^2g_+|_{T\partial\overline{M}}]$ is said to be the \textit{conformal infinity} of $g_+$.

Suppose that $n\geq 3$ and odd. For $0< \alpha<1$, a $C^{2,\alpha}$ \textit{Poincar\'{e}-Einstein
metric} is a $C^{2,\alpha}$ conformally compact metric on $M$ with conformal infinity $[\gamma]$ which satisfies the condition
\begin{eqnarray}\label{eq-nonlinear einstein}
\left\{
\begin{array}{rcll}
Ric(g_+)&=&-ng_+ &\quad\textrm{on $M$},\\
x^2g_+|_{T\partial M}&\in& [\gamma] &\quad \textrm{on $\partial\overline{M}$}.
\end{array}
\right.
\end{eqnarray}
If $g_+$ is a Poincar\'{e}-Einstein metric with conformal infinity
$[\gamma]$, then for a choice of representative $g_0\in [\gamma]$,
there exists a defining function $\rho$ such that for some
$\epsilon>0$ and an identification of a neighborhood of $\partial
\overline{M}$ with $[0,\epsilon)\times \partial \overline{M}$,
$g_+$ takes the \textit{geodesic normal form}
\begin{equation}\label{normal}
g_+=\rho^{-2}(d\rho^2+g(\rho))
\end{equation}
where $g(\rho)$ is a 1-parameter family of metrics on $\partial
\overline{M}$ satisfying $g(0)=g_0$. Here $\rho$ is called the
\textit{geodesic defining function} for $(g_+,g_0)$. If $g_0$ is
$C^{\infty}$, a boundary regularity theorem shows that $g(\rho)\in
C^{\infty}([0,\epsilon)\times\partial \overline{M}; S^2T^*\partial
\overline{M})$ (See [A1][A4][CDLS][H]). Moreover the Taylor
expansion of $g(\rho)$ is even to order $n$, i.e.
\begin{equation}\label{normal series} g(\rho)=
g_0+\rho^2g_2+...+\rho^{n-1}g_{n-1}+\rho^ng_n+O(\rho^{n+1})\end{equation}
where $g_i\in C^{\infty}(\partial \overline{M},
S^2T^{*}\partial \overline{M})$ for $i=2,4,..., n-1,n$, and $g_n$ is transverse-traceless on $(\partial \overline{M}, g_0)$,
i.e.
\begin{equation}\label{transverse traceless cond}
\mathrm{tr}_{g_0}g_n=0,\quad \mathrm{\delta}_{g_0}g_n=0, \end{equation}
where in any local coordinates $(U;y^1,...,y^n)$ of $\partial\overline{M}$,
$$\mathrm{tr}_{g_0}g_n=[g_0]^{kl}[g_n]_{kl},\quad \mathrm{\delta}_{g_0}g_n=[g_0]^{jk}[g_n]_{lj;k}dy^l.$$
Notice that
for $i=2,4,...,n-1$, $g_i$ is locally determined by the
boundary metric $g_0$; however, $g_n$ is locally
undetermined subject only to the two conditions in
(\ref{transverse traceless cond}). The terms $g_0$ and $g_n$ play
the role of Dirichlet data and Neumann data respectively in this
problem.

The \textit{Dirichlet-to-Neumann relation} $\mathscr{N}$ is defined as the set of pairs $(\phi, \psi)$ such that there exists a Poincar\'{e}-Einstein metric $g_+$ on $M$ with geodesic normal
form (\ref{normal}) and Taylor expansion (\ref{normal series})
such that $g_0=\phi$, $g_n=\psi$. $\mathscr{N}$ satisfies the
following equivariance properties with respect to diffeomorphisms
and conformal changes:
\begin{equation}\label{diff inv}
(\phi,\psi)\in \mathscr{N} \Longleftrightarrow
(\Phi^*\phi,\Phi^*\psi)\in \mathscr{N},\quad \Phi=\overline{\Phi}|_{\partial\overline{M}},\ \forall \overline{\Phi}\in
\textrm{Diff}(\overline{M}),
\end{equation}
\begin{equation}\label{conformal inv}
(\phi,\psi)\in \mathscr{N} \Longleftrightarrow
(e^{2u}\phi,e^{(2-n)u}\psi)\in\mathscr{N},\quad \forall u\in
C^{\infty}(\partial \overline{M}).
\end{equation}
If for a fixed boundary metric $\phi$, the Poincar\'{e}-Einstein
metric $g_+$, with $g_0=\phi$ in the geodesic normal
form, is uniquely determined up to diffeomorphism restricting to
identity on the boundary, then the relation $\mathscr{N}$ defines
the \textit{Dirichlet-to-Neumann map} on the boundary denoted by
$\mathcal{N}: \phi\mapsto\psi$. However, the global existence and
uniqueness of Poincar\'{e}-Einstein metrics fails for a general
compact manifold $\overline{M}$.

In this paper, we study the linearization of the
Dirichlet-to-Neumann map under the assumption that it is well
defined. For this goal, we do not need the full strength of uniqueness, the global uniqueness, of Poincar\'{e}-Einstein metric,
but only need to assume the uniqueness of the solution to the
linearized equation at a fixed Poincar\'{e}-Einstein metric $g_+$,
\begin{equation}\label{linear eq}
 Ric'_{g_+}h+nh=0
\end{equation}
up to the deformation generated by a vector field vanishing on the
boundary. See more details in Section \ref{sec-proof}.

For $0<\alpha<1$, let $PE^{m,\alpha}(M)\subset x^{-2}C^{m,\alpha}(\overline{M};S^2T^*\overline{M})$ be the
subspace of $C^{m,\alpha}$ Poincar\'{e}-Einsten metrics on $M$, which can be
given the $C^{m,\alpha}$ topology on $\overline{M}$ via a fixed
compactification as in (\ref{eq-nonlinear einstein}). Obviously,
this topology is independent of the choice of smooth defining
function. Let $Met^{\infty}(\partial \overline{M})$ denote the space of
$C^{\infty}$ metrics on $\partial \overline{M}$. Clearly, the tangent space
of $Met^{\infty}(\partial \overline{M})$ is $C^{\infty}(\partial
\overline{M};S^2T^*\partial \overline{M})$. Hence, our
Dirichlet-to-Neumann map and its linearization, if well defined,
are
$$\mathcal{N}: Met^{\infty}(\partial \overline{M})\longrightarrow
C^{\infty}(\partial \overline{M};S^2T^*\partial \overline{M}),$$
$$d\mathcal{N}: C^{\infty}(\partial \overline{M};S^2T^*\partial \overline{M})\longrightarrow
C^{\infty}(\partial \overline{M};S^2T^*\partial \overline{M}).$$

To set up the problem, let $g_0\in Met^{\infty}(\partial \overline{M})$ be a
representative of the conformal infinity of some $C^{\infty}$
conformal compact Poincar\'{e}-Einstein metric $g_+$. Assume that
there exists a neighborhood $W$ of $g_0$, $W\subset
Met^{\infty}(\partial \overline{M})$, such that for each
$\tilde{g}_0\in W$, there exists a unique Poinca\'{e}-Einstein
metric $\tilde{g}_+\in PE^{2,\alpha}$, up to diffeomorphisms
restricting to the identity on $\partial \overline{M}$, such that
$x^2\tilde{g}_+|_{T\partial \overline{M}}\in [\tilde{g}_0]$. Then
the Dirichlet-to-Neumann map is well defined on $W$. Let
'$\mathrm{tf}_{g_0}h_0$' be the trace-free part of a symmetric 2-tensor $h_0$ with respect to $g_0$.
Our main theorem is the following:
\begin{theorem}\label{thm-main}
The linearization of the Dirichlet-to-Neumann map at $g_0$, $d\mathcal{N}$,
is a pseudo-differential operator of order $n$ with principal symbol
$$\sigma_n(d\mathcal{N})=2^{-n} \frac{\Gamma(-\frac{n}{2})}{\Gamma(\frac{n}{2})}
\sigma_n(\triangle^{\frac{n}{2}}_{g_0} (\mathrm{tf}_{g_0}-\mathrm{tf}_{g_0}\delta^*_{g_0} (\delta_{g_0}\mathrm{tf}_{g_0}\delta^*_{g_0})^{-1} \delta_{g_0}\mathrm{tf}_{g_0}))$$
where $h_0\in C^{\infty}(\partial \overline{M};S^2T^*\partial \overline{M})$.
Moreover, if
$$\Sigma=\{h_0\in C^{\infty}(\partial\overline{M}; \mathscr{S}^2T^*\partial
\overline{M}):\mathrm{tr}_{g_0}h_0=0, \delta_{g_0}h_0=0\}$$
is the subspace of trace-free and divergence-free symmetric 2-tensor fields, then
$d\mathcal{N}|_{\Sigma}$ is elliptic with principal symbol
$$\sigma_n(d\mathcal{N}|_{\Sigma})=2^{-n} \frac{\Gamma(-\frac{n}{2})}{\Gamma(\frac{n}{2})} \sigma_n(\triangle^{\frac{n}{2}}_{g_0}).$$
\end{theorem}

Our work is inspired by the short paper [Gr1] of
Graham, who studied the Dirichlet-to-Neumann map in the case
$M=B^{n+1}=\{y=(y^0,...,y^n)\in
\mathbb{R}^{n+1}:\sum_{i=0}^{n}|y^i|^2<1\}$. Choosing a defining
function $\rho=\frac{1}{2}(1-|y|^2)$, the hyperbolic metric
$g_+=\rho^{-2}\sum_{i=0}^n(dy^i)^2$ is a Poincar\'{e}-Einstein
metric with prescribed infinity $\rho^2g_+|_{T\mathbb{S}^n}=g_0$,
where $g_0=|d\theta|^2$ is the standard sphere metric. Based on his
earlier joint work [GL] with Lee, the Poincar\'{e}-Einstein metric
uniquely exists with prescribed infinity sufficiently closed to $g_0$,
up to the diffeomorphisms restricting to identity on the sphere.
Therefore, the Dirichlet-to-Neumann map is well defined in a
neighborhood of $g_0$ in $C^{k,\alpha}(S^n; S^2T^*S^n)$ for $k>n$.
Graham gave a precise formula for the linearization of the
Dirichlet-to-Neumann map at $g_0=|d\theta|^2$, which is a pseudo-differential operator
and especially, is elliptic on the trace-free and divergence-free sections of (0,2)-tensors with respect to $g_0$. The generalization here of Graham's result does not give an exact formula, but only the principal symbol for this operator. See Section \ref{sec-proof}.

The outline of this paper is as follwos. In Section 2, we set up the problem and reduce it to the study of two 0-type differential operators. In Section 3, we introduce the 0-calculus due to Mazzeo and Melrose, which is used to deal with the geometric operators associated to some asymptotically hyperbolic metric in the interior of a compact manifold and is the main technique used here since a Poincar\'{e}-Einstein metric is clearly asymptotically hyperbolic. In Section 4, we choose a suitable gauge, the Bianchi gauge, to make the linearized Einstein equation elliptic in the $0$-calculus and write out the linearized relation of a family of Poincar\'{e}-Einstein metrics in this gauge and in the geodesice normal form. In Section 5, using the method introduced in [MM], [Ma], [GZ] and [JS], we construct the parametrix, Poisson kernel and scattering matrix for the linearied Einstein equation in the chosen gauge. In Section 6, we prove the main theorem and conclude that the scattering matrix, modified by some differential operators, provides the linearization of the Dirichlet-to-Neumann map. Finally, in Section 7, we also deal with the $N$ point problem for $N\geq 2$, to give a better description of the Dirichlet-to-Neumann map following the constructions in Sections 2-6. 

I would like to thank Richard Melrose for his help on this paper. The many discussions we had on these topics were indispensable and without his patience and guidance, this paper would not exist.

\section[]{Geometric Setting and Linearization}{\label{sec-geom setting}}
Let $x$ be the fixed smooth defining function as in
Section \ref{intro}. Suppose the dimension of $\partial \overline{M}$, n, is $\geq 3$ and odd. For any $q\in \partial \overline{M}$, let $(U;y^1,...,y^n)$ be the local coordinate
patch around $q$ in $\partial \overline{M}$. Then there exists
$\epsilon >0$, such that $([0,\epsilon)\times U;x,y^1,...,y^n)$ is
a local coordinate patch around $q$ in $\overline{M}$. For simplicity, denote
by $y^0=x$ and $\{y^{\nu}\}_{0\leq\nu\leq n}=(x,y)$ with $y=(y^1,...,y^n)$.

Suppose that for some $0<\alpha<1$, $g_{+}$ is a
$C^{2,\alpha}$ conformally compact Poincar\'{e}-Einstein metric
on $M$, for which there is a smooth representative, $g_0$, in the conformal
infinity $[x^2g_+|_{T\partial \overline{M}}]$. Let
$\rho=\rho(x,y)$ be the geodesic defining function for
$(g_+,g_0)$. Let $\overline{g}$ be the $C^{2,\alpha}$ extension of
$x^2g_+$ to $\overline{M}$. According to [Le], $\rho$ has the
following properties.
\begin{lemma}
$\rho\in C^{1,\alpha}(\overline{M})\cap C^{2,\alpha}(M)$ and
has the following properties
\begin{itemize}
\item[(a)] $\rho=\rho(x,y)=e^{v(y)}x+O(x^2)$, where $v\in
C^{2,\alpha}(\partial \overline{M})$ satisfies
$g_0=e^{2v(s,y)}x^2g_+(x)|_{T\partial \overline{M}}$;
\item[(b)]
$\rho^2g_+\in C^{2,\alpha}(M;\mathscr{S}^2T^*M)$ has a
$C^{1,\alpha}$ extension to $\overline{M}$, such that
$\rho^2g_+|_{T\partial \overline{M}}=g_0$;
\item[(c)]
$|d\rho|^2_{\tilde{g}}=1$ in a neighborhood of $\partial
\overline{M}$.
\end{itemize}
\end{lemma}
\begin{proof}
For any $q\in \partial \overline{M}$, choose a local coordinate $([0,\epsilon)\times U;x,y) = ([0,\epsilon)\times U;y^{\nu})$ around $q$ and write $\rho=e^{u(x,y)}x$. Then $u$ satisfies $u(0,y)=v(y)$ and $|d(e^ux)|^2_{e^{2u}\overline{g}}=1$ in $U$. The constant length condition gives
$$2\overline{g}^{0\nu}\partial_{\nu}u+x|du|^2_{\overline{g}} =\frac{1-|dx|^2_{\overline{g}}}{x}.$$ 
This equation is a non-characteristic first order PDE for $u$, which can be solved in a neighborhood of $\partial \overline{M}$ by Hamilton-Jacobi theory. Let $F:T^{*}\overline{M}\longrightarrow \mathbb{R}$ denote the function 
\begin{equation}\label{eq-geodesic}
F(p,\xi)=2\langle dx(p),\xi\rangle_{\overline{g}}+x(p)|\xi|_{\overline{g}}-
\frac{1-|dx(p)|^2_{\overline{g}}}{x(p)}
\end{equation}
for $p\in \overline{M}$, $\xi\in T_p^*\overline{M}$. Solving (\ref{eq-geodesic}) is equivalent to finding a function $u$ on $\overline{M}$ such that $F(p,du)=0$. Since $\overline{g}$ is $C^{2,\alpha}$, the first two terms in $F$ are $C^{2,\alpha}$ functions on $T^*\overline{M}$. Letting $b(p)=\frac{1-|dx(p)|^2_{\overline{g}}}{x(p)}$, then $b\in c^{1,\alpha}(\overline{M})\cap C^{2,\alpha}(M)$ and $b|_{\partial \overline{M}} \in C^{2,\alpha}(\partial \overline{M})$. Thus the Hamilton vector field $X_F$ is in $C^{0,\alpha}(\overline{M})\cap C^{1,\alpha}(M)$ and its flow $\varphi:R\times T^*\overline{M} \longrightarrow T^*\overline{M}$ is a $C^{0,\alpha}$ map whenever it is defined. Let
$$\hat{\omega}(y)=du|_{\partial \overline{M}} =\partial_xu(0,y)dx +\sum_{i=1}^n\partial_{y_i}v(y)dy^i =v_0(y)dx +\sum_{i=1}^nv_i(y)dy^i,$$ 
where $v_i(y)=\partial_{i}v(y)$. By solving (\ref{eq-geodesic}) at the boundary, $v_0$ is uniquely
determined and $v_0\in C^{1,\alpha}(\partial \overline{M})$. Let $\psi$ be the flow-out by $X_F$ from the set $\{(p,\hat{\omega})\in T^*\overline{M}: p=(0,y)\in \partial \overline{M}\}$. Since $F(p,\hat{\omega})=0$ for $p\in \partial \overline{M}$, the image of $\psi$ is a $C^{0,\alpha}$ lagrangian
submanifold with the boundary of $T^*\overline{M}$ contained in $F^{-1}$; moreover, it is $C^{1,\alpha}$ in the interior. Since it is transverse to the fibers of $T^*\overline{M}$ along $\partial \overline{M}$,
so near $\partial \overline{M}$, the flow out is the image of a closed 1-form $\omega \in C^{0,\alpha}(\overline{M}; T^*\overline{M})\cap C^{1,\alpha}(M;T^*M)$. So locally, $\omega=du$ for some function $u\in C^{1,\alpha}(\overline{M})\cap C^{2,\alpha}(M)$. With $u(0,y)=v(y)$ on the boundary, $u$ is uniquely determined and clearly $\rho=e^{u(x,y)}x$ satisfies $(a)$, $(b)$, $(c)$.
\end{proof}

\begin{proposition}\label{prop-smoothness}
There exists a $C^{1,\alpha}$ diffeomorphism $\Phi:\overline{M}\longrightarrow \overline{M}$ satisfying
$\Phi|_{\partial\overline{M}}=Id$, with a $C^{2,\alpha}$ restriction $\Phi|_{M}: M\longrightarrow M$, such that in a neighborhood of the boundary,
\begin{equation}
\tilde{g}_+=\Phi^{*}g_+=x^{-2}(dx^2+g(x))
\end{equation}
where $g(x)$ is smooth family of metrics on $\partial \overline{M}$ parametrized by $x$ with Taylor expansion even up to order $n$, i.e.
\begin{equation}
g(x)=g_0+x^2g_2+\cdots+x^{n-1}g_{n-1}+x^ng_n+O(x^{n+1}).
\end{equation}
Here $g_0$ is the smooth representative of $x^{2}g_+|_{T\partial \overline{M}}$ as above and $g_i\in C^{\infty}(\partial \overline{M};S^2T^*\partial M)$ for $i=2,4,...,n-1$ and $n$.
\end{proposition}

\begin{proof}
According to [CDLS], there exist a $C^{1,\alpha}$ collar
diffeomorphism $\Psi$ from $[0,\epsilon)\times \partial \overline{M}$ to a
neighborhood of $\partial\overline{M}$, denote as $\widetilde{W}$ such that
$\Psi^{-1}(x,y)=(\rho(x,y),y)$ and
$\Psi^{*}g_+=\rho^{-2}(d\rho^2+G(\rho))$, where $G\in
C^{\infty}([0,\epsilon),Met^{\infty}(\partial \overline{M}))$ and $G(0)=g_0$.
Then $\Psi$ is $C^{2,\alpha}$ away from the boundary since $\rho\in
C^{1,\alpha}(\overline{M})\cap C^{2,\alpha}(M)$. On the other hand,
using local coordinates $(x,y)$, we can identify
$[0,\epsilon)\times
\partial M$ with a neighborhood of $\partial\overline{M}$ denoted as $W$. Hence
$\Psi$ can be viewed as a $C^{1,\alpha}$ diffeomorphism from $W$
to $\widetilde{W}$, such that $\Psi|_{\partial \overline{M}}=Id$. Let $\mu$
be a $C^{\infty}$ function on $\mathbb{R}$ such that $\mu(x)=1$ is
$x\leq \tfrac{1}{4}$, $\mu(x)=0$ if $x\geq \frac{1}{2}$ and
$-8\leq\mu'\leq 0$. For $0<\epsilon '<\epsilon$ and $\epsilon$
small enough, $\widetilde{\Psi}_{\epsilon '}:
\widetilde{W}\supset [0,\epsilon ']\times\partial \overline{M}\longrightarrow
\subset W$  is well defined by
$$\widetilde{\Psi}_{\epsilon '}(x,y)=(\mu(\frac{x}{16\epsilon '
e^{v(y)}})(\rho(x,y)-x)+x,y)=(\rho_{\epsilon'}(x,y),y).$$
Then for $0\leq x\leq\epsilon '$ and $\epsilon '$ small, 
$\partial_x (\rho_{\epsilon'})>0$. Hence
$\widetilde{\Psi}_{\epsilon'}:[0,\epsilon']\times\partial
\overline{M}\longrightarrow [0,\epsilon']\times\partial \overline{M}$ is a
$C^{1,\alpha}$ diffeomorphism, which is $C^{2,\alpha}$ on
$(0,\epsilon']\times\partial \overline{M}$ and satisfies
$$\widetilde{\Psi}_{\epsilon'}|_{[0,\tfrac{\epsilon'}{4}]\times\partial \overline{M}}=\Psi^{-1},\quad
\widetilde{\Psi}_{\epsilon'}|_{[\tfrac{\epsilon'}{2},\epsilon']\times\partial
\overline{M}}=Id.$$
Therefor $\widetilde{\Psi}_{\epsilon'}$ can be extend to
a $C^{1,\alpha}$ diffeomorphism on $\overline{M}$. Let
$\overline{\Phi}=\widetilde{\Psi}_{\epsilon'}^{-1}$. It is easy to
check that such $\Phi$ has all the above properties.
\end{proof}

Now let's make clear the assumptions and notations which are fixed from Section \ref{sec-geom setting} to Section \ref{sec-proof}. Assume:
\begin{itemize}
\item[(A1)] For $s\in (-\theta,\theta)$, $\theta>0$, $g_+(s)$ is a $C^{1}$ family of $C^{2,\alpha}$ conformally compact Poincar\'{e}-Einstein metrics on $M$ with a $C^1$ family of conformal infinities, i.e, $\exists$ $g_0(s)\in C^{1}((-\theta,\theta), Met^{\infty}(\partial M))$ such that $g_0(s)\in [x^2g_+(s)|_{T\partial M}]$ for each s;
\item[(A2)] $g_+(0)$ is $C^{\infty}$ conformally compact.
\end{itemize}

Under this two assumptions, using Proposition \ref{prop-smoothness}, there exists a $C^{1}$ family of $C^{1,\alpha}$ diffeomorphisms $\Phi_s$ over $\overline{M}$, which are $C^{2,\alpha}$ on $M$, such that $\tilde{g}_+(s)=\Phi_s^{*}g_+(s)$ is also a $C^1$ family of $C^{2,\alpha}$ conformally compact Poincar\'{e}-Einstein metrics, which are smooth near the bondary, i.e. for $0<x<\epsilon$,
$$\tilde{g}_+(s)=x^{-2}(dx^2+g_0(s)+x^2g_2(x) +\cdots+x^{n-1}g_{n-1}(s)+x^ng_n(x)+O(x^{n+1})).$$
In particular,$\Phi_0$ is smooth. Hence $\tilde{g}_+(0)$ is smooth. For simplicity, let
$$g=\tilde{g}_+(0)=x^{-2}(dx^2+g_0+x^2g_2+\cdots+x^{n-1}g_{n-1}+x^ng_n+O(x^{n+1})),$$
$$\tilde{h}=\frac{d\tilde{g}_+(s)}{ds}\vline_{s=0}=x^{-2}(g'_0+x^2g'_2+ \cdots +x^{n-1}g'_{n-1} +x^ng'_n +O(x^{n+1})).$$ 
where $g'_i=\frac{dg_i(s)}{ds}|_{s=0}$ for $i=0,2,...,n$. Then the linearized Einstein equation
\begin{equation}\label{eq-linear einstein}
\frac{1}{2}\triangle_{g}\tilde{h}-\delta_{g}^{*}\beta_{g}\tilde{h}
+R\tilde{h}+n\tilde{h}=0
\end{equation}
is obtained by calculating $\frac{d}{ds}|_{s=0}(Ric(g(s))+ng(s))=0$. See [GL] for details.
Here $\triangle_{g}$, $\delta_{g}$ and $\delta^{*}_g$ are the standard connection Laplacian,
divergence operator and its adjoint with respect to the metric $g$ and
$$\beta_g\tilde{h}=\delta_g\tilde{h}+\frac{1}{2}d\mathrm{tr}_g\tilde{h}$$
is the Bianchi operator; $R$ denotes the the $0$ order operator induced by the curvature of $g$.
For $h\in C^{2}(M;S^2T^{*}M)$ and $\omega\in C^1(M;T^{*}M)$, in any local coordinate patch $(W\subset M;\{y^{\nu}\}_{0\leq \nu\leq n})$, these operators can be written as
$$[\triangle_gh]_{\mu\nu}=-h_{\mu\nu;\alpha}^{\ \ \ \ \alpha},$$
$$[\delta_gh]_{\nu}=-h_{\mu\nu;}^{\ \ \ \mu },$$
$$[\delta^{*}_g\omega]_{\mu\nu}=\tfrac{1}{2}(\omega_{\mu;\nu}+\omega_{\nu;\mu}),$$
$$[Rh]_{\mu\nu}=R_{\alpha\mu\nu\beta}h^{\alpha\beta}+\tfrac{1}{2}(R_{\ \mu}^{\alpha}h_{\alpha\nu}+R_{\ \nu}^{\alpha}h_{\mu\alpha}),$$
$$[\triangle_g\omega]_{\nu}=-\omega_{\nu;\mu}^{\ \ \ \mu},$$
$$[R\omega]_{\nu}=R_{\nu\mu}\omega^{\mu}.$$

The linearized Dirichlet-to-Neumann map, if well-defined at $g_0$,
is the map $g_0'\mapsto g_n'$. We want to obtain this map by
studying the properties of the solution to (\ref{eq-linear
einstein}). The first well-known difficulty is that (\ref{eq-linear einstein}) is not
elliptic, but only transverse elliptic, since Einstein equation is
invariant under diffeomorphisms. To use elliptic theory, we
first need to choose a gauge to make (\ref{eq-linear einstein})
elliptic. The condition we impose is $\beta_gh=0$. We wish to find a $C^1$ family of
diffeomorphisms $\widetilde{\Phi}_s\in
\mathrm{Diff}^2(\overline{M})$, such that
$h=\frac{d}{ds}|_{s=0}\widetilde{\Phi}_s^{*}\tilde{g}_+(s)$
satisfies $\beta_gh=0$. Then $h$ will satisfy the elliptic
equation
\begin{equation}\label{eq-elliptic linear einstein}
\triangle_{g}h+2Rh+2nh=0.
\end{equation}
Since we are dealing with the linearized gauge condition, instead of finding exact $\widetilde{\Phi}_s$ for $s\in (-\theta,\theta)$, we only need to look for a generating vector field $X$, with dual 1-form $\omega$, such that
$$h=L_Xg+\tilde{h}=\delta_g^{*}\omega+\tilde{h}$$
satisfies $\beta_gh=0$. Then
$\omega$ should satisfy the equation
\begin{equation*}
2\beta_g\delta_g^{*}\omega+2\beta_g\tilde{h}=0
\Longleftrightarrow
\triangle_g\omega-R\omega+2\beta_g\tilde{h}=0.
\end{equation*}
Since $g$ is a Poincar\'{e}-Einstein metric implies that $R\omega=-n\omega$, the linearized gauge fixing problem reduces to finding a solution to
\begin{equation}\label{eq-gauge fix}
\triangle_gw+nw+2\beta_g\tilde{h}=0.
\end{equation}

Set $J_g=\triangle_g+n\in \textrm{Diff}^2(M;T^{*}M)$ and $L_g=\triangle_{g}+2R+2n\in \textrm{Diff}^2(M;S^2T^{*}M)$. They are both uniformly degenerate elliptic operators. Next, we
will carefully study these two operators, as well as the solutions to (\ref{eq-gauge fix}) and (\ref{eq-elliptic linear einstein}). The main technique is to use the $0$-calculus due to Mazzeo and
Melrose, which will be introduced in the following section.

\section[]{0-Calculus}\label{sec-0calculus}
In this section, we review the construction of a parametrix for an elliptic differential operator on an asymptocially hyperbolic space due to Mazzeo and Melrose ([MM]). As in Section \ref{sec-geom setting}, $(M,g)$ is a Riemannian manifold of dimension $n+1$ with smooth boundary defining function $x$. Near the boundary,
$$g=x^{-2}(dx^2+g_0+x^2g_2+\cdots+x^{n-1}g_{n-1}+x^ng_n+O(x^{n+1})).$$
Hence all sectional curvatures have limit $-|dx|^2_{x^2g}|_{\partial\overline{M}}=-1$, i.e., $(M,g)$ is asymptotically hyperbolic. 
 
Suppose $V$ is a smooth vector bundle over $\overline{M}$ with fibre dimension $N$.

\subsection{Vector Fields and Differential Operators}
For each $q\in\partial\overline{M}$, let $(U;y)$ be a local coordinate patch in $\partial \overline{M}$ around $q$. Then $([0,\epsilon)\times U;x,y)$ is a local coordinate patch in $\overline{M}$ around $q$. 

Let $\mathscr{V}_0$ be the Lie algebra consisting of vector fields 
which are locally sums of $C^{\infty}$ function multiples of the
vector field $x\partial_x$ and $x\partial_{y_i}$. $\mathscr{V}_0$
can be viewed as the set of $C^{\infty}$ sections of a
natural vector bundle over $M$, denoted by $^0TM$. Obviously,
$\{x\partial_x,x\partial_{y_1},...,x\partial_{y_n}\}$ gives a
local basis of $^0TM$. The dual bundle, naturally denoted
by $^0T^{*}M$, has a local basis
$\{\frac{dx}{x},\frac{dy_1}{x},...,\frac{dy_n}{x}\}$. Moreover,
$^0TM$ and $^0T^{*}M$ can be viewed as vector bundles over
$\overline{M}$, which are isomorphic to $T\overline{M}$ and $T^{*}\overline{M}$ over $\overline{M}$. However, the isomorphisms are not natural at $\partial\overline{M}$.
Similarly, we can define $\mathscr{V}_b$ be the space of
$C^{\infty}$ vector fields generated by $\{x\partial_x,
\partial_{y^1},...,\partial_{y^n}\}$. Notice that the Lie algebras
$\mathscr{V}_0$ and $\mathscr{V}_b$ satisfy
$$[\mathscr{V}_0,\mathscr{V}_0]\subset \mathscr{V}_0,\quad
[\mathscr{V}_0,\mathscr{V}_b]\subset \mathscr{V}_0,$$ where
$[\cdot,\cdot]$ is the Lie bracket.

For any constant $m\in \mathbb{R}$, denote by $S^m(^0T^*M)$ the space of $u\in C^{\infty}(^0T^*M)$ such that in any local coordinates $(W\subset M;p)$ away from the boundary, i.e. $\overline{W}\cap \partial\overline{M} =\emptyset$, denoting by $^0T^*M|_W = W_{p}\times \mathbb{R}^{n+1}_{\xi}$ canonically, then for any $(n+1)$-multiple indices $\alpha, \beta$,
$$|\partial_p^{\alpha}\partial_{\xi}^{\beta}u(p,\xi)|\leq C_{\alpha\beta}(1+|\xi|)^{m-|\beta|}$$
for some constant $C_{\alpha\beta}$; and in any local coordinates $([0,\epsilon)\times U;x,y);x,y)$ near the boundary with $U\subset \partial\overline{M}$, denoting by $^0T^*M|_{[0,\epsilon)\times U} = [0,\epsilon)\times U \times \mathbb{R}^{n+1}_{\tau,\eta}$ canonically, then for any $n$-multiple indices $\alpha,\beta$ and $k,l\in\mathbb{N}_0$,
$$|\partial_x^k\partial_{y}^{\alpha}\partial_{\tau}^l\partial_{\eta}^{\beta}u(x,y,\tau,\eta)|\leq C_{\alpha\beta kl}(1+|\tau|+|\eta|)^{m-l-|\beta|}$$
for some constant $C_{\alpha\beta kl}$. Similarly, we can define $S^{m}(^0T^*M;End(V))$. 

For $m\in \mathbb{N}_0$, let $\textrm{Diff}_0^m(M;V)$ be the space of differential
operators of order $m$ which can be written as a finite sum of at
most $m$-fold products of vector fields in $\mathscr{V}_0$ with
values in $End(V)$. Near the boundary, in a
coordinate patch $([0,\epsilon)\times U;x,y)$, choose a
trivialization of $V$, such that $V_{[0,\epsilon)\times U}\cong
([0,\epsilon)\times U)\times \mathbb{C}^N$. Then any $P\in \textrm{Diff}_0^m(M;V)$ is of the form
\begin{equation}\label{diff op}
P_{ij}=\sum_{|\alpha|+m\leq k}C_{ij,m\alpha}(x,y)(x\partial
x)^m(x\partial_y)^{\alpha}.
\end{equation}
The
$0$-\textit{princinpal symbol} of $P$ is defined by
$$[^0\sigma_m(P)]_{ij}=\sum_{k+|\alpha|=m}C_{ij;k\alpha}(x,y)\tau^k\eta^{\alpha}.$$
The coordinate invariant property shows that $^0\sigma_m(P)$ is
globally well defined, i.e.
$$^0\sigma_m(P)\in S^m(^0T^{*}M;End(V))/S^{m-1}(^0T^{*}M;End(V)).$$

Using the symbol language, we can define the \textit{'small'
0-pseudo-differential operators} $\Psi_0^m(M;V)$ naturally as the
usual smooth pseudodifferential operators. The symbol map induces
a short exact sequence
\begin{equation}\label{sss-symbol}
0\longrightarrow \Psi_0^{m-1}(M;V) \longrightarrow\Psi_0^m(M;V) \mathrel{\mathop{\longrightarrow}^{^0\sigma_m}} S^m(^0T^{*}M;End(V))/S^{m-1}(^0T^{*}M;End(V))
\longrightarrow 0.
\end{equation}

\subsection{Stretched Product Space}
To investigate the properties of Schwartz kernel for operators in $\Psi_0^m(M)$, we recall the \textit{stretched product} $\overline{M}\times_0\overline{M}$ obtained by blowing up the
boundary diagonal $\textit{diag}(\partial \overline{M})$ in $\overline{M}\times\overline{M}$. 
Let $\beta$ be the blow down map, i.e.
$$\beta: \overline{M}\times_0\overline{M}\longrightarrow \overline{M}\times \overline{M}.$$
If we denote by $(x',y')$ the corresponding local coordinate on the second copy of
$M$, then in the neighborhood around $(q,q)$ in $diag(\partial\overline{M})$, we can introduce the polar coordinates
$$R=|x^2+x'^2+|y-y'|^2|^{\frac{1}{2}}, \ r=|y-y'|, \ s=\frac{x}{x'},\ t=\frac{x'}{x},$$
$$z=\frac{y-y'}{x'},\ z'=\frac{y-y'}{x},\ \varrho=\frac{x}{R},\ \varrho'=\frac{x'}{R},\ \omega=\frac{y-y'}{r},$$
which give the local coordinates in $\overline{M}\times_0\overline{M}$ around $\beta^{-1}(q,q)$. 
$\overline{M}\times_0\overline{M}$ is a compact manifold with corners up to dimension $3$. It has three boundary hypersurfaces, top face $T$, bottom face $B$ and the front face $F$. Then $R$ (resp. $\rho$, $\rho'$) is the defining function for $F$ (resp. $T$, $B$). The front face at $q$ is denoted by $F_q=\beta^{-1}(q,q)$, which is a smooth compact manifold with corners of codimension $2$ and has two boundary hypersurfaces $F_q\cap T$ and $F_q\cap B$ with corresponding defining functions $\rho$ and $\rho'$. Moreover, the interior of $F_q$ is isomorphism to hyperbolic space $\mathbb{H}^{n+1}$, or
the inward half tangent space $T_q^+M$, carrying on the nature semi-product group structure as
$\mathbb{R}_+\ltimes\mathbb{R}^n$. Denote by $\textit{diag}_0(M)$ the closure of $\beta^{-1} (\textit{diag}(M))$ in $\overline{M}\times_0\overline{M}$. Clearly,
$$\textit{diag}_0(M)\cap F=\{s=1,z=0,R=0\}.$$

Let $[\overline{M}\times_0\overline{M}]^2$ be the $\mathscr{V}_0$ stretched product doubled across the front face. It has the structure of a smooth manifold with corners up to codimension $2$.
Furthermore, the doubled diagonal $[\textit{diag}_0(M)]^2$ lies in the interior. Note that any vector bundle over $\overline{M}\times_0\overline{M}$ can be extended to $[\overline{M}\times_0\overline{M}]^2$.

Similarly, $\partial \overline{M}\times_0\overline{M}$ is the blow-up of $\partial \overline{M} \times \overline{M}$ along $\textit{diag}(\partial \overline{M})$. It can be embedded in $\overline{M} \times_0 \overline{M}$. Actually, the identification of $\partial \overline{M}\times_0\overline{M}$ with $\{\rho=0\}\subset \overline{M}\times_0\overline{M}$ provides the embedding map. Denote by $\tilde{\beta}$ the blown down map
$$\tilde{\beta}:\partial \overline{M}\times_0\overline{M}\longrightarrow
\partial \overline{M}\times\overline{M}.$$
The front face of $\partial \overline{M}\times_0\overline{M}$ at any $q\in \partial \overline{M}$ is just $F_q\cap T$.

Finally, denote by $\partial \overline{M}\times_0\partial\overline{M}$ the  blow up of $\partial \overline{M}\times\partial \overline{M}$ along $\textit{diag}(\partial \overline{M})$, which
is naturally embedded into $\partial
\overline{M}\times_0\overline{M}$ by identification with $\{\rho'=0\}\subset \partial \overline{M}\times_0\overline{M}$, with the blow down map
$$\hat{\beta}:\partial \overline{M}\times_0\partial\overline{M}\longrightarrow
\partial \overline{M}\times\partial\overline{M}.$$
Note that $\partial \overline{M}\times_0\partial\overline{M}$ is a
smooth manifold with boundary containing the lift of
$\textit{diag}(\partial \overline{M})$ via $\hat{\beta}$, which
can be identified with the spherical normal bundle of
$\textit{diag}(\partial \overline{M})$, i.e.
$SN(\textit{diag}(\partial \overline{M}))$.

\subsection{Half Densities}
Let $\Gamma^{\frac{1}{2}}_0(\overline{M})$ be the line bundle of
singular half densities on $\overline{M}$, trivialized by
$\nu:=|dvol_g|^{\frac{1}{2}}$, and $\Gamma^{\frac{1}{2}}(\partial\overline{M})$ be the bundle of half density on $\partial \overline{M}$, trivialized by
$\nu_0:=|dvol_{g_0}|^{\frac{1}{2}}$. Naturally we can define
$\Gamma_0^{\frac{1}{2}}(\overline{M}\times\overline{M})$ (resp.
$\Gamma_0^{\frac{1}{2}}(\partial \overline{M}\times\overline{M})$
and
$\Gamma^{\frac{1}{2}}(\partial\overline{M}\times\partial\overline{M})$)
the half density bundle on $\overline{M}\times\overline{M})$
(resp. $\partial \overline{M}\times\overline{M}$ and
$\partial\overline{M}\times\partial\overline{M}$) trivialized by
$\nu\times\nu$ (resp. $\nu_0\times\nu$ and $\nu_0\times\nu_0$).
Hence finally we can define the half density bundles over the
blown up spaces via the blow down map, i.e
$$\Gamma_0^{\frac{1}{2}}(\overline{M}\times_0\overline{M})
=\beta^*\Gamma_0^{\frac{1}{2}}(\overline{M}\times\overline{M}),$$
$$\Gamma_0^{\frac{1}{2}}(\partial \overline{M}\times_0\overline{M})
=\tilde{\beta}^*\Gamma_0^{\frac{1}{2}}(\partial\overline{M}\times_0\overline{M}),$$
$$\Gamma_0^{\frac{1}{2}}(\partial\overline{M}\times_0\partial\overline{M})
=\hat{\beta}^*\Gamma^{\frac{1}{2}}(\partial\overline{M}\times\partial\overline{M}).$$
Note that $\Gamma_0^{\frac{1}{2}}(\overline{M}\times_0\overline{M})$ is
trivial on $F_q$. For example, in coordinates $(x',y',s,z)$,
$\Gamma_0^{\frac{1}{2}}(\overline{M}\times_0\overline{M})|_{F_q}$
is spanned by
$|\frac{dx'dy'dsdz}{x'^{n+1}s^{n+1}}|^{\frac{1}{2}}$; in
coordinates $(x,y,t,z')$, it is spanned by
$|\frac{dxdydtdz'}{x^{n+1}t^{n+1}}|^{\frac{1}{2}}$; in coordinates
$(\rho,\rho',r,\omega,y)$, it is spanned by $|\frac{d\rho d\rho'
drd\omega dy}{\rho^{n+1}\rho'^{n+1}r^{n+1}}|^{\frac{1}{2}}$.
Similarly, $\Gamma_0^{\frac{1}{2}}(\partial
\overline{M}\times_0\overline{M})$ is trivial on $F_q\cap T$.
Hence, $\Gamma_0^{\frac{1}{2}}(\overline{M}\times_0\overline{M})$
extends to half density bundle
$\Gamma_0^{\frac{1}{2}}([\overline{M}\times_0\overline{M}]^2)$
over $[\overline{M}\times_0\overline{M}]^2$.

With such half density bundles, we can define the Sobolev spaces
by
$$L_0^2(M;\Gamma_0^{\frac{1}{2}}(\overline{M}))=\{f\in C^{-\infty}(M;\Gamma_0^{\frac{1}{2}}(\overline{M})):
\int_M|f|^2<\infty\}$$ which is a Hilbert space with inner
product $$\langle f_1,f_2\rangle_{L_0^2(M)} = \int_M
f_1\bar{f}_2$$ for any $f_1,f_2\in
L_0^2(M;\Gamma_0^{\frac{1}{2}}(\overline{M}))$. Moreover, for $m\in\mathbb{N}$,
$$H^m_0(M;\Gamma_0^{\frac{1}{2}}(\overline{M}))=\{f\in L_0^2(M;\Gamma_0^{\frac{1}{2}}(\overline{M})):
\mathscr{V}_0^kf\in L_0^2(M;\Gamma_0^{\frac{1}{2}}(\overline{M})),
 0\leq k\leq m\},$$
$$H^{-m}_0(M;\Gamma_0^{\frac{1}{2}}(\overline{M})=(H^m_0(M;\Gamma_0^{\frac{1}{2}}(\overline{M}))'.$$

\subsection{Boundary Cornormal Distributions}
Let $X$ be any compact manifold with $k$ boundary hypersurfaces $\partial_i X$ for $1\leq i\leq k$, with corresponding defining function $\rho_i$. Let $K\in \mathbb{N}_0$ be a fixed constant. Then for any $c_1,...,c_k\in \mathbb{C}$, we can define a functions space $\mathscr{A}^{c_1,...,c_k}(X)$ associated to $K$ with certain conormal sigularities at the boundary surfaces, i.e. 
\begin{equation}\label{defn-A}
\begin{split}
\mathscr{A}^{c_1,...,c_k}(X) = \{u\in
\mathscr{D}'(X):\mathscr{V}_b^l u\in
\rho_1^{c_1}\cdots\rho_k^{c_k} (\sum_{i=1}^N(\log\rho_1)^{d_{i1}}\cdots(\log\rho_k)^{d_{ik}}
L^{\infty}(X)),\\
\textrm{for some $N\in\mathbb{N}_0$ and for $1\leq i\leq N$, $1\leq j\leq k$, $d_{ij}\in\mathbb{N}_0\cap [0,K]$}
\}
\end{split}
\end{equation}
And if some $e_i=\infty$, define
$\mathscr{A}^{c_1,...\infty,...,c_k}(X)=\bigcap_{l=1}^{\infty}
\mathscr{A}^{c_1,...l,...,c_k}(X)$.

Moreover, we can define
$\mathscr{A}_{phg}^{c_1,...,c_k}(X)$ to be the natural subspace of
polyhomogeneous elements in $\mathscr{A}^{c_1,...,c_k}(X)$ by replacing $L^{\infty}(X)$ by $C^{\infty}(X)$ in the definition, i.e.
\begin{equation}\label{defn-Aphg}
\begin{split}
\mathscr{A}_{phg}^{c_1,...,c_k}(X) = \{u\in \mathscr{D}'(X): \mathscr{V}_b^l u \in
\rho_1^{c_1}\cdots\rho_k^{c_k} (\sum_{i=1}^N(\log\rho_1)^{d_{i1}}\cdots(\log\rho_k)^{d_{ik}}
C^{\infty}(X)),\\
\textrm{for some $N\in\mathbb{N}_0$ and for $1\leq i\leq N$, $1\leq j\leq k$, $d_{ij}\in\mathbb{N}_0\cap [0,K]$} \}
\end{split}
\end{equation}

These definitions can be easily extended to boundary conormal distributional sections of a vector bundle over $X$.

\subsection{0-Pseudodifferential Operators}
Let $\mathscr{K}_0^m(\overline{M}\times_0\overline{M};diag_0(M); End(V)\otimes\Gamma_0^{\frac{1}{2}})$ be the space of distributional sections of $End(V)\otimes\Gamma_0^{\frac{1}{2}} (\overline{M}\times_0\overline{M})$ on $\overline{M}\times_0\overline{M}$, which can be extended to $[\overline{M}\times_0\overline{M}]^2$ as distributional sections of $End(V)\otimes\Gamma_0^{\frac{1}{2}} ([\overline{M}\times_0\overline{M}]^2)$, conormal along the doubled diagonal $[diag_0(M)]^2$ of order $m$ and vanishing to infinite order on the doubled boundary faces $T$ and $B$. Then
$$P\in \Psi_0^m(M;V\otimes\Gamma_0^{\frac{1}{2}})\Longleftrightarrow \kappa(P)\in \mathscr{K}_0^m(\overline{M}\times_0\overline{M};
diag_0(M);End(V)\otimes\Gamma_0^{\frac{1}{2}}),$$
where $\kappa(P)$ denotes the Schwartz kernel of $P$.

Suppose $V|_{\partial\overline{M}}$ is decomposed into $V^1\oplus\cdots\oplus V^r$ where $V^i$ is a vector bundle over $\partial \overline{M}$ for $1\leq i\leq r\leq N$. Given any constant vector
$a=(a_1,...,a_r)^{t}\in \mathbb{C}^r$. Then for some constant $K\in \mathbb{N}_0$, we can define
\begin{equation}
\begin{split}
\mathscr{A}_{phg}^{a}(\overline{M};V\otimes\Gamma_0^{\frac{1}{2}})=\{v\in \mathscr{D}'(\overline{M};V\otimes\Gamma_0^{\frac{1}{2}})\cap C^{\infty}(M;V\otimes\Gamma_0^{\frac{1}{2}}): v=\sum_{i=1}^r\sum_{j=0}^K x^{a_i}(\log x)^ju_{ij},\\
u_{ij}\in C^{\infty}(\overline{M};V\otimes\Gamma_0^{\frac{1}{2}}),\ u_{ij}|_{\partial\overline{M}}\in C^{\infty}(\partial\overline{M};V^i\otimes\Gamma_0^{\frac{1}{2}}), \textrm{for $1\leq i\leq r$ and $0\leq j\leq K$}\}
\end{split}
\end{equation}

Given constant matrices $E_T,E_B\in
M(r,\mathbb{C})$. Let $[E_T]_{i\bullet}$ denote the the $i$-th row
of $E_T$ and $[E_B]_{\bullet j}$ denote the $j$-th column of
$E_B$. Then we can define $\mathscr{A}_{phg}^{E_T,E_B}([\overline{M}\times_0\overline{M}]^2;End(V)\otimes\Gamma_0^{\frac{1}{2}})$
to be a subspace of $\mathscr{D}'([\overline{M}\times_0\overline{M}]^2;End(V)\otimes\Gamma_0^{\frac{1}{2}})$ such that any element $u\in \mathscr{A}_{phg}^{E_T,E_B}([\overline{M}\times_0\overline{M}]^2;End(V)\otimes\Gamma_0^{\frac{1}{2}})$ has the form $[u]_{\bullet j}\sim u_B\rho'^{[E_B]_{\bullet j}}$ near boundary face $B$ and
$ [u]_{i\bullet}\sim\rho^{[E_T]_{i\bullet}}u_T$ near boundary face $T$, where with some cut off function $\chi$ satisfying $\chi(\rho)=1$ if
$|\rho|<\varepsilon$ and $\chi(\rho)=0$ if $|\rho|>2\varepsilon$ for some $\varepsilon>0$, $\chi(\rho')u_B,\chi(\rho)u_T\in
\mathscr{A}^{0,0}([\overline{M}\times_0\overline{M}]^2);End(V)\otimes\Gamma_0^{\frac{1}{2}})$.
Note, for any constant $c$, and any
defining function $x$, we always denote
$$[E_T+c]_{ij}=[E_T]_{ij}+c,\quad [x^{E_T}]_{ij}=x^{[E_T]_{ij}}.$$

Now we can define the space of
distributions cornormal to boundaries $T$ and $B$ and smooth up to the
front face $F$ and the corresponding space of operators:
\begin{equation}\begin{split}
\mathscr{K}_0^{-\infty,E_T,E_B}(\overline{M}\times_0\overline{M};End(V)\otimes\Gamma_0^{\frac{1}{2}})
=\mathscr{A}^{E_T,E_B}([\overline{M}\times_0\overline{M}]^2;End(V)\otimes\Gamma_0^{\frac{1}{2}})
|_{\overline{M}\times_0\overline{M}}
\end{split}\end{equation}
\begin{equation}\begin{split}
\Psi_0^{-\infty,E_T,E_B}(M;V\otimes\Gamma_0^{\frac{1}{2}})=\{A \in C^{-\infty} (M;End(V)\otimes\Gamma_0^{\frac{1}{2}}):\quad\quad\quad\quad\quad\quad\quad
\\ \kappa(A) \in \mathscr{K}^{-\infty,E_T,E_B} (\overline{M}\times_0\overline{M};End(V)\otimes\Gamma_0^{\frac{1}{2}}).
\end{split}\end{equation}
So we have the \textit{'large' class of 0-pseudo-differential
operators} defined by
\begin{equation}\begin{split}
\Psi_0^{m,E_T,E_B}(M;V\otimes\Gamma_0^{\frac{1}{2}})=\Psi_0^{m}(M;V
\otimes\Gamma_0^{\frac{1}{2}})+\Psi_0^{-\infty,E_T,E_B}(M;V\otimes
\Gamma_0^{\frac{1}{2}})
\end{split}\end{equation} with corresponding kernel space
\begin{equation}\begin{split}
\mathscr{K}_0^{m,E_T,E_B}(\overline{M}\times_0\overline{M};diag_0(M);End(V)\otimes\Gamma_0^{\frac{1}{2}})=
\mathscr{K}_0^m(\overline{M}\times_0\overline{M};diag_0(M);End(V)\otimes\Gamma_0^{\frac{1}{2}})\\
+\mathscr{K}_0^{\infty,E_T,E_B}(\overline{M}\times_0\overline{M};End(V)\otimes\Gamma_0^{\frac{1}{2}}).
\end{split}\end{equation}

Note that the decomposition of $V$, which has only been considered at the boundary so far, usually can be extended to a neighborhood of $\partial\overline{M}$ up to $O(x^{\Lambda})$ for some certain constant $\Lambda\in\mathbb{N}_0\cup \{\infty\}$ in practice. Suppose this is the case. Consider $u$ as a section of $V$.  Near the boundary, if writing
$$u=\sum_{i=1}^r\sum_{j=0}^K x^{a_i}(\log x)^ju_{ij},\textrm{with $u_{ij}\in C^{\infty}(\overline{M};V\otimes \Gamma_0^{\frac{1}{2}}))$, $u_{ij}|_{\partial\overline{M}}\in C^{\infty}(\partial \overline{M};V^{i}\otimes \Gamma_0^{\frac{1}{2}}$)},$$
then the expression makes sense if and if only
$$|\Re a_i-\Re a_j|\leq \Lambda,\quad \forall\ 1\leq i,j\leq r.$$
Usually, $\Lambda$ is determined by $V$, $\overline{M}$ and the metric.
Hence $\Psi_0^{m,E_T,E_B}(M;V\otimes \Gamma_0^{\frac{1}{2}})$
makes sense if and only if $E_T, E_B$ satisfy the following condition: 
\begin{condition}\label{cond-ET}
$$|min_{\ 1\leq j\leq
r}\{\Re [E_T]_{ij}\}-min_{\ 1\leq j\leq
r}\{\Re [E_T]_{kj}\}|< \Lambda, \quad \forall\ 1\leq i,k\leq r,$$
$$|min_{\ 1\leq j\leq
r}\{\Re [E_B]_{ji}\}-min_{\ 1\leq j\leq
n}\{\Re [E_B]_{jk}\}|< \Lambda, \quad \forall\ 1\leq i,k\leq r.$$
\end{condition}

According to [Ma], we have the mapping properties for such large
class of 0-pseudodifferential operators.
\begin{proposition}\label{prop-map property of 0pseudo} For each
$A\in \Psi_0^{m,E_T,E_B}(M;V\otimes \Gamma_0^{\frac{1}{2}})$ and constants $m,m'\in\mathbb{N}_0$ and $b,b'\in\mathbb{R}$,
$$A: x^bH_0^{m'+m'}(M;V\otimes \Gamma_0^{\frac{1}{2}})\longrightarrow
x^{b'}H_0^{m'}(M;V\otimes \Gamma_0^{\frac{1}{2}})$$
is a
bounded operator provided that 
$$\min_{1\leq i,j\leq r}\{[E_T]_{ij}\}-b'>\frac{n}{2}, \quad \min_{1\leq i,j\leq r}\{[E_B]_{ij}\}+b>\frac{n}{2},$$
$$\min_{1\leq i,j\leq r}\{[E_T]_{ij}\}+\min_{1\leq i,j\leq r}\{[E_B]_{ij}\}>n, \quad b\geq b'.$$
\end{proposition}
\begin{proposition}\label{prop-mapproperty}
For any $A\in \Psi_0^{m,E_T,E_B}(M;V\otimes
\Gamma_0^{\frac{1}{2}})$, let $e_j=[E_T]_{\bullet j}$ for $1\leq j\leq r$. Then
$$A:\dot{C}^{\infty}(M;V\otimes\Gamma_0^{\frac{1}{2}})\longrightarrow \sum_{j=1}^r
\mathscr{A}^{e_j}(M;V\otimes\Gamma_0^{\frac{1}{2}})$$ 
is continuous.
\end{proposition}

\subsection{Normal Operator}
Suppose $A\in \Psi_0^{m,E_T,E_B}(M;V\otimes\Gamma_0^{\frac{1}{2}})$. The
\textit{normal operator} of $A$ at $q$ is defined by
$$N(A)f=\int \kappa(A)|_{F_q}f(\frac{x}{s},y-\frac{x}{s}z)\frac{ds}{s}dz\cdot d\mu,\quad
\forall f\in C_c^{\infty}(\mathbb{H}^{n+1},\mathbb{C}^N),$$ 
where $d\mu=|\frac{dxdy}{x^{n+1}}|^{\frac{1}{2}}$. Note that $V|_{F_q}=\pi^*(V|_q)$ is trivial, where $\pi$ is the projection map $T_q^+M\rightarrow \{q\}$. In particular, if $P\in \textrm{Diff}_0^m(M;V \otimes \Gamma_0^{\frac{1}{2}})$ is of the form (\ref{diff op}), then its normal operator is
$$[N(P)]_{ij}=\sum_{|\alpha|+m\leq k}C_{ij,m\alpha}(0,y)(x\partial x)^m(x\partial_y)^{\alpha}.$$ 
The normal operator characterizes the behavior of the original operator near the boundary, which can
be easily understood via an equivalent definition
$$N_q(P)u=\lim_{\tau\rightarrow 0}R^*_{\tau}\Phi^*P(\Phi^{-1})^*R_{\frac{1}{\tau}}^*u,$$
where $\Phi$ is a local diffeomorphism carrying a neighborhood of $0\in\mathbb{H}^{n+1}$ to a neighborhood of $q\in \partial \overline{M}$ and $R_{\tau}$ denotes the scaling. According to [MM], we have a short
exact sequence
\begin{equation}\label{sss-normal}
\begin{split}
0\longrightarrow
R\Psi_0^{m,E_T,E_B}(M;End(V)\otimes\Gamma_0^{\frac{1}{2}})\longrightarrow
\Psi_0^{m,E_T,E_B}(M;End(V)\otimes\Gamma_0^{\frac{1}{2}})\quad\quad\quad\quad \\
\mathrel{\mathop{\longrightarrow}^{N_q}} \mathscr{K}_0^{m,E_{F_q\cap
T},E_{F_q\cap B}}(F_q;F_q\cap diag_0(M);M(N,\mathbb{C})\otimes
\Gamma_0^{\frac{1}{2}})\longrightarrow 0.
\end{split}
\end{equation}
Here $R$ is the boundary defining function of $F_{q}$ in $\overline{M}\times_0\overline{M}$ and $R\Psi_0^{m,E_T,E_B}(M;End(V)\otimes\Gamma_0^{\frac{1}{2}})$ consists of operators with kernels in $R\mathscr{K}_0^{m,E_T,E_B} (\overline{M}\times_0\overline{M}; diag_0(M);End(V)\otimes\Gamma_0^{\frac{1}{2}})$. 

\subsection{Indicial Operator}
Also for $A\in \Psi_0^{m,E_T,E_B}(M;V\otimes\Gamma_0^{\frac{1}{2}})$ we can
define the \textit{indicial operator} by
$$I(A)(s)f=x^{-s}(A(x^sf))|_{\partial \overline{M}}, \quad \forall f\in C^{\infty}(\partial \overline{M};V\otimes\Gamma_0^{\frac{1}{2}}).$$
Similarly, if $P\in \textrm{Diff}_0^m(M;V\otimes\Gamma_0^{\frac{1}{2}})$ with components as in
(\ref{diff op}), then we can write
$$[I(P)(s)]_{ij}=\sum_{m\leq k}C_{ij,m0}(0,y)s^m,\ \forall s\in\mathbb{C}.$$
The indicial operator characterizes the leading order of the operator near the boundary.
According to [MM], there is a short exact sequence
\begin{equation}\label{sss-indicial}
\begin{split}
0\longrightarrow
R^{\infty}\Psi_0^{-\infty,E_T+1,E_B}(M;End(V)\otimes\Gamma_0^{\frac{1}{2}})\longrightarrow
R^{\infty}\Psi_0^{-\infty,E_T,E_B}(M;End(V)\otimes\Gamma_0^{\frac{1}{2}})\quad\quad\quad\quad
\\ \mathrel{\mathop{\longrightarrow}^{I}}
R^{\infty}\mathscr{K}_0^{-\infty,E_T,E_B}(\overline{M}\times_0\overline{M};End(V)\otimes\Gamma_0^{\frac{1}{2}})
/R^{\infty}\mathscr{K}_0^{-\infty,E_T+1,E_B}(\overline{M}\times_0\overline{M};End(V)\otimes\Gamma_0^{\frac{1}{2}})
\longrightarrow 0,
\end{split}
\end{equation}
where $R$ is the boundary defining function of the front face $F$ in $\overline{M}\times_0\overline{M}$. 

\subsection{Parametrix}
The construction of a parametrix for an elliptic differential operator of 0-type is based on the three short exact sequences (\ref{sss-symbol}), (\ref{sss-normal}) and (\ref{sss-indicial}) generated by symbol map, normal operator and indicial operator respectively. Here we only do this for a small class of elliptic differential operators, which includes $L_g$ and $J_g$ introduced in Section \ref{sec-geom setting}. First, notice that
\begin{lemma} 
If $A\in \Psi_0^{m,E_T,E_B}(M;End(V)\otimes\Gamma_0^{\frac{1}{2}})$, $B\in R^{\infty} \Psi_0^{-\infty,E_T,E_B} (M;End(V)\otimes\Gamma_0^{\frac{1}{2}})$ and $P\in \mathrm{Diff}_0^k (M;End(V)\otimes\Gamma_0^{\frac{1}{2}})$, then
$$N(PA)=N(P)N(A),$$
$$I(PB)=I(P)I(B).$$
\end{lemma}

\begin{theorem}\label{thm-construct parametrix}
Let $V$ be a vector bundle over $\overline{M}$ with fibre dimension $N$ and a decomposition $V|_{\partial\overline{M}}=V^1\otimes\cdots\oplus V^r$ at the boundary, for some $1\leq r\leq N$. Suppose that the decomposition can be extended to a neiborhood of $\partial\overline{M}$ up to $O(x^{\Lambda})$ for some $\Lambda\in\mathbb{N}\cup\{\infty\}$ and suppose that $P\in \mathrm{Diff}_0^2 (M;V\otimes\Gamma_0^{\frac{1}{2}})$ is elliptic, i.e., $^0\sigma_2(P)\in S^{2}(^0T^{*}M;End(V)\otimes\Gamma_0^{\frac{1}{2}})$ is invertible. Assume $P$ also has the following properties:
\begin{itemize}
\item[(1)] The indicial operator $I(P)$ is diagonal. For $1\leq k\leq r$, $I(P)(s)|_k=f_k(s)$, where
$f_k(s)=-s^2+ns+c_k$ for some $c_k\geq 0$. Let $s_k\leq s^k$ be
the two indicial roots of $f_k(s)=0$, then $|s_i-s_j|\leq \Lambda$ for $1\leq i,j\leq r$. Moreover,
$$\frac{n}{2}\in(\underline{s},\overline{s})
=\bigcap_{k=1}^r(s_k,s^k)\neq \emptyset.$$

\item[(2)] $\forall
q\in\partial\overline{M}$, the normal operator $N_q(P)$ is invertible, i.e.
 $$N_q(P)^{-1}:\dot{C}^{\infty}(F_q;M(N,\mathbb{C})\otimes\Gamma_0^{\frac{1}{2}})\longrightarrow
\mathscr{A}_{phg}^{E_T,E_B}(F_q;M(N,\mathbb{C})\otimes\Gamma_0^{\frac{1}{2}})$$
 is bounded, where
 $E_T$ and $E_B$ are $r\times r$ matrices satisfying $\mathrm{Condition\ \ref{cond-ET}}$ and
$$[E_T]_{ij}=s^i+1-\delta_{ij}, \quad E_B=E_T^t.$$
\end{itemize}
Then there exists a right parametrix $Q\in \Psi_0^{-m,E_T,E_B}(M;End(V)\otimes\Gamma_0^{\frac{1}{2}})$ such that
$$PQ-Id=R_1\in R^{\infty}\Psi_0^{-\infty,\infty,E_B}(M;End(V)\otimes \Gamma_0^{\frac{1}{2}}).$$
\end{theorem}
\begin{proof}
Use the symbol calculus and (\ref{sss-symbol}), it is easy to find a first approximation to the parametrix, $Q_1\in \Psi_0^{-m}(M;End(V)\otimes\Gamma_0^{\frac{1}{2}})$ such that
$$PQ_1-Id=E_1\in \Psi_0^{-\infty}(M;End(V)\otimes\Gamma_0^{\frac{1}{2}}).$$
Notice that the Schwartz kernel $\kappa(E_1)$ vanishes to infinite order at the boundary faces $T$ and $B$. Hence
$$\kappa(E_1)\sim\sum_jR^j\kappa_j, \quad \kappa_j\in \dot{C}^{\infty} (B^{n+1};M(N,\mathbb{C}) \otimes\Gamma_0^{\frac{1}{2}}) \in L^2_0(B^{n+1};M(N,\mathbb{C})\otimes\Gamma_0^{\frac{1}{2}}).$$
Since $N(P)$ is invertible, by (\ref{sss-normal}) there exists $Q_2\in\Psi_0^{-\infty,E_T,E_B}$ such that
$$PQ_2-E_1=E_2\in R^{\infty}\Psi_0^{-\infty,E_T,E_B} (M;End(V)\otimes\Gamma_0^{\frac{1}{2}}).$$ 
Notice that $P$ actually kills the leading term of $Q_2$ at boundary $T$, by a modification as in [MM], we can choose $\widetilde{Q}_2\in \Psi_0^{-\infty,E_T+Id,E_B}$. Hence
$$P\widetilde{Q}_2-E_1=\widetilde{E}_2\in R^{\infty}\Psi_0^{-\infty,E_T+Id,E_B} (M;End(V)\otimes\Gamma_0^{\frac{1}{2}}).$$ 
So $\kappa(\widetilde{E}_2)\in \mathscr{A}^{E_T+Id,E_B}$. At last, since $I(P)(E_T+Id+c)$ is invertible for all $c\in \mathbb{N}_0$, we can find $Q_3\in R^{\infty}\Psi_0^{-\infty,E_T+Id,E_B}
(M;End(V)\otimes\Gamma_0^{\frac{1}{2}})$, such that
$$PQ_3-\widetilde{E}_2=E_3\in R^{\infty}\Psi_0^{-\infty,\infty,E_B} (M;End(V)\otimes\Gamma_0^{\frac{1}{2}}).$$
Then
$Q=Q_1+\widetilde{Q}_2+Q_3$.
\end{proof}

\section[]{Gauge Fixing} \label{sec-gauge fix}
In this section, let's study the operator $J_g$ introduced in Section \ref{sec-geom setting} carefully. Recall near the boundary
$$g=x^{-2}(dx^2+g_0+x^2g_2+\cdots+x^{n-1}g_{n-1}+x^ng_n+O(x^{n+1})),$$
and $J_g=\triangle_g+n$ acting on sections of $T^*M$. More naturally, we can view $J_g$ as an operator acting on sections of ${}^0T^*M$.

\subsection{Indicial Operator}
Clearly, ${}^0T^*M|_{\partial \overline{M}}$ has a decomposition with repect to metric $g$, i.e.
$${}^0T^*M|_{\partial \overline{M}}=\mathscr{U}^1\oplus\mathscr{U}^2=span\{\frac{dx}{x}\}\oplus {}^0T^*\partial \overline{M}.$$
Such decomposition can be extended to a neighborhood of the boundary by identifying it with $[0,\epsilon)\times \partial \overline{M}$ in the geodesic normal way, i.e.
$${}^0T^*M|_{[0,\epsilon)\times \partial \overline{M}}=({}^0T^*\partial \overline{M})^{\perp} \oplus{}^0T^*\partial \overline{M},$$
where the $'\perp '$ is defined by the induced metric on ${}^0T^*M$ from $g$. Hence in this case, we can choose $\Lambda=\infty$ in Condition \ref{cond-ET}. The indicial operator corresponding to this
decomposition is
$$I(P)(s)=-s^2+ns+\left[\begin{array}{cc}2n&0\\0&(n+1)\end{array}\right],\quad\forall\ s\in\mathbb{C}.$$
So the two indicial functions corresponding to this decomposition are
$$f_1(s)=-x^2+ns+2n\quad\mathrm{and} \quad f_2(s)=-x^2+ns+n+1$$
with corresponding
indicial roots
$$s_1=\frac{n-\sqrt{n^2+8n}}{2},\ s^1=\frac{n+\sqrt{n^2+8n}}{2}\quad\mathrm{and}\quad \ s_2=-1,\
s^2=n+1.$$ Hence $(\underline{s},\overline{s})=(-1,n+1)$.

\subsection{Normal Operator}
To compute the normal operator, consider the half plane model of hyperbolic space
$$\mathbb{H}^{n+1}=\mathbb{R}^+_x\times\mathbb{R}^n_y\quad \text{with\ metric}\quad r=\frac{dx^2+dy^2}{x^2},\quad dvol_r=\frac{dxdy}{x^{n+1}}.$$
For any $q\in \partial \overline{M}$, $^0T_q^*M$ has a basis $\{\frac{dx}{x},\frac{dy^i}{x},..., \frac{dy^n}{x}\}$, which corresponds to the basis of $^0T^*\mathbb{H}^{n+1}$. Set $e^0=\frac{dx}{x}$ and $e^i=\frac{dy^i}{x}$ for $1\leq i\leq n$. Then
$^0T^*\mathbb{H}^{n+1}=\mathbb{H}^{n+1}\times (\mathbb{C}\oplus\mathbb{C}^n)$ corresponds to the decomposition of $^0T^*M_{\partial \overline{M}}$. The normal operator of $J_g$ can be expressed as
$$N_q(J_g)=-(x\partial_x)^2+nx\partial_x-x^2\sum_{i=1}^{n}\partial_{y^i}^2+
\left[\begin{array}{cc}2n&0\\0&(n+1)Id_n\end{array}\right]
+2x\left[\begin{array}{cccc}0&-\partial_{y^1}&\cdots&
-\partial_{y^n}\\
\partial_{y^1}&0&\cdots
&0\\ \vdots
&\vdots&\cdots&\vdots\\\partial_{y^n}&0&\cdots&0\end{array}\right].$$
On the other hand $N_q(J_g)=J_r=\triangle_r+n$ acting on $^0T^*\mathbb{H}^{n+1}$, so it is self-adjoint.

\begin{lemma}
For any $a\in
(\tfrac{n-\sqrt{n^2+4n}}{2},\tfrac{n+\sqrt{n^2+4n}}{2})$,
$$N_q(J_g):x^{a}H^2_0(\mathbb{H}^{n+1};\mathbb{C}^{n+1})\longrightarrow
x^{a}L^2_0(\mathbb{H}^{n+1};\mathbb{C}^{n+1})$$ is an isomorphism.
\end{lemma}
\begin{proof}
For each $u\in H^2_0(\mathbb{H}^{n+1};\mathbb{C}^{n+1})$,
$N_q(J_g)(x^{a}u)=a(n-a)x^{a}u+x^{a}N_q(J_g)u$. Hence,
$$\langle N_q(J_g)(x^{a}u),x^{a}u\rangle_{x^{a}L_0^2}
=\langle a(n-a)u+N_q(J_g)u,u\rangle_{L_0^2}\geq \|\nabla
u\|^2_{L_0^2}+(-a^2+na+n)|u|^2_{L_0^2}\geq 0,$$ and equality holds if
and only if $u=0$.
\end{proof}

For any $u\in L_0^2(\mathbb{H}^{n+1},\mathbb{C}^{n+1})$, let $\hat{u}(x,\xi)=\mathscr{F}_{y\rightarrow \xi}(u)(x,\xi)$ be the Fourier Transform with respect to the variable $y$. Then $N_q(J_g)u(x,y)=0$ is reduced to the ordinary differential equation $\widehat{N}_q(J_g)\hat{u}(x,\xi)=0$, where
$$\widehat{N}_q(J_g)=-(x\partial_x)^2+nx\partial_x+x^2|\xi|^2+
\left[\begin{array}{cc}2n&0\\0&(n+1)Id_n\end{array}\right]
+2ix\left[\begin{array}{cccc}0&-\xi^1&\cdots&
-\xi^n\\
\xi^1&0&\cdots &0\\ \vdots &\vdots&\cdots&\vdots\\
\xi^n&0&\cdots&0\end{array}\right].$$
Letting $\bar{x}=x|\xi|$ for $\xi\neq 0$ and $\hat{\xi}=\xi/|\xi|\in S^{n-1}$, then the
equation $\widehat{N}_q(J_g)v(x|\xi|,\hat{\xi})=0$ is equivalent to $\widetilde{N}_q(J_g)v(\bar{x},\xi)=0$, where
$$\widetilde{N}_q(J_g)=-(\bar{x}\partial_{\bar{x}})^2+n\bar{x}\partial_{\bar{x}}
+\bar{x}^2+
\left[\begin{array}{cc}2n&0\\0&(n+1)Id_n\end{array}\right]
+2i\bar{x}\left[\begin{array}{cccc}0&-\hat{\xi}^1&\cdots&
-\hat{\xi}^n\\
\hat{\xi}^1&0&\cdots &0\\ \vdots &\vdots&\cdots&\vdots\\
\hat{\xi}^n&0&\cdots&0\end{array}\right].$$

Let $\Gamma_0^{\frac{1}{2}}(\mathbb{R}^+)$ be the half density bundle over $\mathbb{R}_x^+$ trivialized by $|\frac{dx}{x^{n+1}}|^{\frac{1}{2}}$.
\begin{lemma}\label{lem-no L2 ef for tilde Ng}
For any $\hat{\xi}\in S^{n-1}$, the map
$$\widetilde{N}_q(J_g): H^2_0(\mathbb{R}_{\bar{x}}^+;\mathbb{C}^{n+1}\otimes \Gamma_0^{\frac{1}{2}}(\mathbb{R}^+))
\longrightarrow L^2_0(\mathbb{R}_{\bar{x}}^+;\mathbb{C}^{n+1}\otimes \Gamma_0^{\frac{1}{2}}(\mathbb{R}^+))$$ 
is continuous and positive.
\end{lemma}
\begin{proof}
For any $w(\bar{x})=\sum_{i=0}^nw_ie^i \in H^2_0(\mathbb{R}_{\bar{x}}^+;\mathbb{C}^{n+1}\otimes \Gamma_0^{\frac{1}{2}}(\mathbb{R}^+))$,
\begin{eqnarray*}
\langle \widetilde{N}_q(J_g)w,w\rangle_{L^2_0} &=&
\|\bar{x}\partial_{\bar{x}}w\|^2_{L^2_0}+\|\bar{x}w\|_{L^2_0}
+2n\|w_0\|_{L^2_0}+(n+1)\sum_{i=1}^n\|w_i\|_{L^2_0}\\&&
+\int_{\mathbb{R}^+}2ix\sum_{i=1}^n
\hat{\xi}^i(w_0\bar{w}_i-w_i\bar{w}_0)\bar{x}^{-n-1}d\bar{x}\\
&&\geq
\|\bar{x}\partial_{\bar{x}}w\|^2_{L^2_0}+\|\bar{x}w\|_{L^2_0}
+2n\|w_0\|_{L^2_0}+(n+1)\sum_{i=1}^n\|w_i\|_{L^2_0}-\|w\|_{L^2_0}-\|\bar{x}w\|_{L^2_0}\\&&
\geq \|\bar{x}\partial_{\bar{x}}w\|^2_{L^2_0}+n\|w\|_{L^2_0}\geq 0
\end{eqnarray*}
And equality holds if and only if $w=0$.
\end{proof}

\begin{lemma}\label{lem-no L2 ef for hat Ng} 
The map
$$\widehat{N}_q(J_g): H^2_0(\mathbb{R}_x^+;\langle\xi\rangle^{-2}
L^2(\mathbb{R}_{\xi}^n;\mathbb{C}^{n+1})\otimes \Gamma_0^{\frac{1}{2}}(\mathbb{R}^+))
\longrightarrow 
L^2_0(\mathbb{R}_x^+;L^2(\mathbb{R}_{\xi}^n;\mathbb{C}^{n+1})\otimes \Gamma_0^{\frac{1}{2}}(\mathbb{R}^+))$$ is continuous and positive.
\end{lemma}
\begin{proof}
For any $w(x,\xi)\in H^2_0(\mathbb{R}^+;\langle\xi\rangle^{-2} L^2(\mathbb{R}^n;\mathbb{C}^{n+1})\otimes \Gamma_0^{\frac{1}{2}}(\mathbb{R}^+))$, by changing variables and last Lemma, we have
$$\langle \widehat{N}_q(J_g)w,w\rangle_{L^2_0L^2}\geq
\|x\partial_{x}w\|^2_{L^2_0L^2}+n\|w\|_{L^2_0L^2}\geq 0$$
and $\langle \widehat{N}_q(J_g)w,w\rangle_{L^2_0L^2}=0$ holds if and only if $w=0$.
\end{proof}
 
According to [CL], the linear ODE $\widetilde{N}_q(J_g)\hat{u}=0$ has a regular singular point at $\bar{x}=0$ and irregular singular point at $\bar{x}=\infty$ and no other singular points. So a formal solution at $\bar{x}=0$ is a true solution and for each formal solution at $\bar{x}=\infty$, there is an actual solution which has asymptotic expansions at $\bar{x}=\infty$ corresponding to it. Hence, to find
all $2N$ linear independent solutions of $\widetilde{N}_q(J_g)\hat{u}=0$, we only need to find out all the
linear independent formal solutions at $\bar{x}=0$. Since such solutions can not blow up at any point between $0$ and $\infty$, we can extend them to the whole half real line $(0,\infty)$. By using the indicial functions $f_i$ and indicial roots $s_i$, $s^i$ defined in Subsection 4.1 for $i=1,2$, the construction of formal solutions at $\bar{x}=0$ is simple as follows.
\begin{itemize}
\item[(1)] Start with $\bar{x}^{s^1}e^0$. Since $f_i(s^1+k)\neq 0$
for all $k\in \mathbb{N}$ and $i=1,2$, so assume
$v^0=\bar{x}^{s^1}e^0+\sum_{k=1}^{\infty}\bar{x}^{s^1+k}w_k$ is a
solution to $\widetilde{N}_q(J_g)\hat{u}=0$ we can solve each function
$w_k$ by induction.
\item[(2)] Start with $\bar{x}^{s_1}e^0$.
Since $f_i(s_1+k)\neq 0$ for all $k\in \mathbb{N}$ and $i=1,2$, as
in $(1)$, there is a solution $v_0=\bar{x}^{s_1}e^0+\sum_{k=1}^{\infty}\bar{x}^{s_1+k}w_k$.
\item[(3)] For each $1\leq i\leq n$, start with
$\bar{x}^{s^2}e^i$. Since $f_i(s^2+k)\neq 0$ for all $k\in
\mathbb{N}$ and $i=1,2$, as in $(1)$, we have a solution $v^i$ of
the form $\bar{x}^{s^2}e^i+\sum_{k=1}^{\infty}\bar{x}^{s^2+k}w_k$.
\item[(4)] For each $1\leq i\leq n$, start with
$\bar{x}^{s_2}e^i$. Note in this case, $f_1(s_2+k)\neq 0$ for all
$k\in \mathbb{N}$ and $f_2(s_2+k)\neq 0$ only if $k\neq n+2$.
Write
$v_i=\bar{x}^{s_2}e^i+\sum_{k=1}^{\infty}\bar{x}^{s_2+k}w_k+\bar{x}^{n+1}\log\bar{x}w$,
we still can solve all $w_k$ and $w$ by induction to makt $v_i$ a solution.
\end{itemize}
On the other hand, the formal solution at $\bar{x}=\infty$ has leading
term either $\bar{x}^{\frac{n}{2}}e^{-\bar{x}}$ or
$\bar{x}^{\frac{n}{2}}e^{\bar{x}}$. By Lemma \ref{lem-no L2 ef for
tilde Ng}, as $\bar{x}\rightarrow \infty$, for $0\leq i\leq n$,
the behavior of solutions constructed above is
$$v_i\sim \bar{x}^{\frac{n}{2}}e^{-\bar{x}},\quad v^i\sim \bar{x}^{\frac{n}{2}}e^{\bar{x}}.$$

Since $\widetilde{N}_q(J_g)$ is self djoint with respect to
$\bar{x}^{-(n+1)}d\bar{x}$, the standard method in [CL] yields the
Green function for $\widetilde{N}_q(J_g)$, i.e.
$$G_1(\bar{x},\bar{x}',\hat{\xi})=U(\bar{x},\hat{\xi})V^*(\bar{x}',\hat{\xi})H(\bar{x}'-\bar{x})
+V(\bar{x},\hat{\xi})U^*(\bar{x}',\hat{\xi})H(\bar{x}-\bar{x}'),$$
where $U(\bar{x},\hat{\xi})$ and $V(\bar{x},\hat{\xi})$ are
$(n+1)\times (n+1)$ matrices whose columns are linear combination
of $v^i$ and $v_i$.
$$\widetilde{N}_q(J_g)G_1(\bar{x},\bar{x}',\hat{\xi}) = \bar{x}^{n+1}\delta(\bar{x}-\bar{x}').$$ 
Moreover,
$$U'(\bar{x},\hat{\xi})V^*(\bar{x}',\hat{\xi})-V'(\bar{x},\hat{\xi})U^*(\bar{x}',\hat{\xi})
=-\frac{1}{2}\bar{x}^{n-1}Id.$$ 
This implies that
$$U(\bar{x},\hat{\xi})=(a_0v^0,a_1v^1,...,a_nv^n),\quad
V(\bar{x},\hat{\xi})=(b_0v_0,b_1v_1,...,b_nv_n),$$ for some
$a_i\neq 0$, $b_i\neq 0$ for $0\leq i\leq n$.

\begin{lemma}
The Green function $G_1(\bar{x},\bar{x}',\hat{\xi})$ induces a bounded integral operator on $L^2(\mathbb{R}_{\bar{x}}^+;\Gamma_0^{\frac{1}{2}}(\mathbb{R}^+))$.
\end{lemma}

\begin{lemma} \label{lem-L2 boundness of G2}
There exists a unique kernel $G_2(x,x',\xi)$ which is bounded on
$L^2_0(\mathbb{R}_x^+;L^2(\mathbb{R}_{\xi}^n;\mathbb{C}^{n+1})\otimes \Gamma_0^{\frac{1}{2}}(\mathbb{R}^+))$ satisfying
$$\widehat{N}_q(J_g)G_2(x,x',\xi)=x^{n+1}\delta(x-x')\otimes Id.$$
\end{lemma}
\begin{proof}
Let $$G_2(x,x',\xi)=|\xi|^{-n}G_1(x|\xi|,x'|\xi|,\hat{\xi})$$
Since the 1-dimensional $\delta$ function is of homogeneity $-1$,
it is easy to check that
$$\widehat{N}_q(J_g)G_2(x,x',\hat{\xi})
=x^{n+1}\delta(x-x')\otimes Id.$$ Note for any $a>0$, the map
$$f(x)\longrightarrow f_a(x)=a^{\frac{n}{2}}f(a^{-1}x)$$
is an isometry on $L^2_0(\mathbb{R}^+;\Gamma_0^{\frac{1}{2}}(\mathbb{R}^+))$. Therefore, for
any $u(x,\xi)\in L^2_0(\mathbb{R}^+;L^2(\mathbb{R}^n;\mathbb{C}^{n+1})\otimes \Gamma_0^{\frac{1}{2}}(\mathbb{R}^+))$,
\begin{eqnarray*}
&&\int_{\mathbb{R}^+}\left|\int_{\mathbb{R}^+}G_2(x,x',\xi)
u(x',\xi))\frac{dx'}{x'^{n+1}}\right|^2 \frac{dx}{x^{n+1}} \\&=&
\int_{\mathbb{R}^+}\left|\int_{\mathbb{R}^+}G_1(\bar{x},\bar{x}',\xi)
|\xi|^{\frac{n}{2}}u(|\xi|^{-1}\bar{x}',\xi))\frac{d\bar{x}'}{\bar{x}'^{n+1}}\right|^2
\frac{d\bar{x}}{\bar{x}^{n+1}}\\ & \leq & c\|
|\xi|^{\frac{n}{2}}u(|\xi|^{-1}x,\xi)\|^2_{L_0^2}=c\|u\|^2_{L_0^2}.
\end{eqnarray*}
Integrating the above inequality over $\mathbb{R}^n_{\xi}$, gives the $L_0^2(\mathbb{H}^{n+1};\mathbb{C}^{n+1}\otimes \Gamma_0^{\frac{1}{2}})$-boundness of $G_2(x,x',\xi)$. The uniqueness follows directly from Lemma \ref{lem-no L2 ef for
hat Ng}, i.e. there is no $L_0^2$ null space.
\end{proof}

\begin{proposition}
The operator $N_q(J_g)=J_r:H_0^2(\mathbb{H}^{n+1};\mathbb{C}^{n+1}\otimes \Gamma_0^{\frac{1}{2}})\longrightarrow
L_0^2(\mathbb{H}^{n+1};\mathbb{C}^{n+1}\otimes \Gamma_0^{\frac{1}{2}})$ is invertible with Green
function 
\begin{equation}\label{eq-green fuction for NJg}
G(x,x',y,y')=\frac{1}{(2\pi)^n}\int_{\mathbb{R}^n}
|\xi|^{-n}G_1(x|\xi|,x'|\xi|,\hat{\xi})e^{i\xi(y-y')}d\xi.
\end{equation}
Moreover, $G\in
\Psi_0^{-2,E_T,E_B}(\mathbb{H}^{n+1};M(n+1,\mathbb{C})\otimes
\Gamma_0^{\frac{1}{2}})$, where
$$E_T=\left[\begin{array}{cc}\frac{n+\sqrt{n^2+8n}}{2}& n+2\\
\frac{n+\sqrt{n^2+8n}}{2}+1&n+1\end{array}\right],\quad
E_B=E_T^{\ast}=\left[\begin{array}{cc}\frac{n+\sqrt{n^2+8n}}{2}&
\frac{n+\sqrt{n^2+8n}}{2}+1\\n+2&n+1\end{array}\right].$$
\end{proposition}
\begin{proof}
The green function
constructed as (\ref{eq-green fuction for NJg}) satisfies
$$N_q(J_g)G(x,x',y,y')=x^{n+1}\delta(x-x')\delta(y-y').$$
The $L^2_0(\mathbb{H}^{n+1};{}^0T^*\mathbb{H}^{n+1}\otimes \Gamma_0^{\frac{1}{2}})$-boundness of
$G$ follows directly from Lemma \ref{lem-L2 boundness of G2}. The
behavior of $G$ at the boundary $T$ and $B$ is obtained by computing
the Fourier Transform (\ref{eq-green fuction for NJg}) based on
asymptotic analysis of $G_2(x,x',\xi)$. We can follow the proof
given by Mazzeo in [Ma] step by step.
\end{proof}

Notice that the front face can be thought of as the ball $B^{n+1}$
blown up a point on the boundary. We can change all the discussion
above for $\mathbb{H}^{n+1}$ to the ball model $B^{n+1}$. Let
$e_1=(s^1,s^1+1)^t$ and $e_2=(s^2+1,s^2)^t$ corresponding to the
$i$-th column of $E_T$ for $i=1,2$.

\begin{corallary}
The map $J_r^{-1}:\dot{C}^{\infty}(\overline{B}^{n+1};\mathbb{C}^{n+1}\otimes
\Gamma_0^{\frac{1}{2}})\longrightarrow \sum_{i=1}^2
\mathscr{A}_{phg}^{e_i}(\overline{B}^{n+1};\mathbb{C}^{n+1}\otimes
\Gamma_0^{\frac{1}{2}})$ is continuous.
\end{corallary}

Identifying $F_q$ with $B^{n+1}$ blown up $p\in
\partial B^{n+1}$, where $\rho$ corresponds to the original
defining function of $B^{n+1}$ and $\rho'$ is corresponds to the
new defining function generated by the blow-up. A similar analysis
as in the proof of Proposition $6.19$ in [MM] shows that

\begin{corallary}
The map $N_q(J_g)^{-1}:\dot{C}^{\infty}(\overline{F_q};M(n+1,\mathbb{C})\otimes
\Gamma_0^{\frac{1}{2}})\longrightarrow
\mathscr{A}_{phg}^{E_T,E_B}(\overline{F_q};M(n+1,\mathbb{C})\otimes
\Gamma_0^{\frac{1}{2}})$ is continuous.
\end{corallary}

Notice that the foundamental solution $U(\bar{x},\hat{\xi})$ does not involve any logrithm terms. Hence in this case, we can take the constant $K=0$ in the definition of (\ref{defn-Aphg}), i.e. all the spaces $\mathscr{A}_{phg}^{0}(\overline{M})$ used in this section can actually be replaced by $C^{\infty}(\overline{M})$.

\subsection{Parametrix}
\begin{proposition}
The operator $J_g: H^2_0(M;{}^0T^*M\otimes \Gamma_0^{\frac{1}{2}})\longrightarrow
L^2_0(M;{}^0T^*M\otimes \Gamma_0^{\frac{1}{2}})$ is self adjoint, positive and thus invertible.
\end{proposition}
\begin{proof}
Clearly, $J_g=\nabla^*\nabla+n$ is self-adjoint. For any $\omega\in H^2_0(M;{}^0T^*M\otimes \Gamma_0^{\frac{1}{2}})$,
$$\langle J_g\omega, \omega\rangle_{x^rL^2_0} =\|\nabla\omega\|^2_{L^2_0}+n\|\omega\|_{L^2_0}\geq 0,$$
with equality holds if and only if $\omega=0$. So $J_g$ is positive and thus injective. Note $J_g$ satisfies all the assumption in Theorem \ref{thm-construct parametrix}, so there exists $Q\in
\Psi_0^{-2,E_T,E_B}(M;{}^0T^*M\otimes\Gamma_0^{\frac{1}{2}})$, such that 
$$J_gQ-Id=R_1\in R^{\infty}\Psi_0^{-\infty,\infty,E_B}(M;{}^0T^*M\otimes\Gamma_0^{\frac{1}{2}}),$$
$$ Q^*J_g-Id=R_1^*\in R^{\infty}\Psi_0^{-\infty,E_T,\infty}(M;{}^0T^*M\otimes\Gamma_0^{\frac{1}{2}}),$$
where
$$E_T=\left[\begin{array}{cc}\frac{n+\sqrt{n^2+8n}}{2}& n+2\\
\frac{n+\sqrt{n^2+8n}}{2}+1&n+1\end{array}\right],\quad
E_B=E_T^*=\left[\begin{array}{cc}\frac{n+\sqrt{n^2+8n}}{2}&
\frac{n+\sqrt{n^2+8n}}{2}+1\\n+2&n+1\end{array}\right].$$ 
By Proposition \ref{prop-map property of 0pseudo}, the map
$$Q: H^2_0(M;{}^0T^*M\otimes \Gamma_0^{\frac{1}{2}})\longrightarrow
L^2_0(M;{}^0T^*M\otimes \Gamma_0^{\frac{1}{2}})$$ 
is bounded. It is also clear that the two maps
$$R_1:L_0^2(M;{}^0T^*M\otimes\Gamma_0^{\frac{1}{2}})\longrightarrow
\dot{C}^{\infty}(M;{}^0T^*M\otimes\Gamma_0^{\frac{1}{2}})$$
$$R_1^*:L_0^2(M;{}^0T^*M\otimes\Gamma_0^{\frac{1}{2}})\longrightarrow
\sum_{i=1}^2\mathscr{A}_{phg}^{e_i}(M;{}^0T^*M\otimes\Gamma_0^{\frac{1}{2}})$$
are bounded. Therefore, as a map from
$L^2_0(M;{}^0T^*M\otimes\Gamma_0^{\frac{1}{2}})$ to itself, $R_1$
and $R_1^*$ are both compact, which implies that 
$$Id+R_1: L^2_0(M;{}^0T^*M\otimes\Gamma_0^{\frac{1}{2}})\longrightarrow 
L^2_0(M;{}^0T^*M\otimes\Gamma_0^{\frac{1}{2}})$$ 
is Fredholm, so $Ran(J_g)$ is closed in $L^2_0(M;{}^0T^*M)$ with $Coker(J_g)$ of finite dimension. Since $J_g$ is self-adjoint, $Coker(J_g)=Ker(J_g)=\{0\}$. Hence, $J_g$ is surjective and thus invertible.
\end{proof}

\begin{proposition}\label{prop-map property of Jg}
For $N>2n$,
$J_g^{-1}:x^N\mathscr{A}^0(\overline{M};{}^0T^*M\otimes
\Gamma_0^{\frac{1}{2}})\longrightarrow
\sum_{i=1}^2\mathscr{A}^{e_i}(\overline{M};{}^0T^*M\otimes
\Gamma_0^{\frac{1}{2}})$ is continuous.
\end{proposition}
\begin{proof}
Since $J_g^{-1}=Q^*+R_1^*J_g^{-1}$, the continuity of $J_g^{-1}$
is directly from the continuity of the following map:
$$Q^*:x^N\mathscr{A}^0(\overline{M};{}^0T^*M\otimes
\Gamma_0^{\frac{1}{2}})\longrightarrow
\sum_{i=1}^2\mathscr{A}^{e_i}(\overline{M};{}^0T^*M\otimes\Gamma_0^{\frac{1}{2}}),$$
$$J_g^{-1}:x^N\mathscr{A}^0(\overline{M};{}^0T^*M\otimes
\Gamma_0^{\frac{1}{2}})\longrightarrow L_0^2(M;{}^0T^*M\otimes
\Gamma_0^{\frac{1}{2}}),$$
$$R_1^*:L^{2}_0(M,{}^0T^*M\otimes
\Gamma_0^{\frac{1}{2}})\longrightarrow
\sum_{i=1}^2\mathscr{A}^{e_i}(M;{}^0T^*M\otimes\Gamma_0^{\frac{1}{2}}).$$
\end{proof}

\begin{proposition}\label{prop-property of inverse}
$J_g^{-1}\in
\Psi_0^{-2,E_T,E_B}(M;{}^0T^*M\otimes\Gamma_0^{\frac{1}{2}})$.
\end{proposition}
\begin{proof}
Use the notation as above, we have
$$J_g^{-1}=Q+J_g^{-1}R_1=Q^*+R_1^*J_g^{-1}=Q^*+R_1^*Q+R_1^*J_g^{-1}R_1
=Q+Q^*R_1+R_1^*J_g^{-1}R_1.$$ 
Notice that $\kappa(R_1)\in R^{\infty} \mathscr{K}^{-\infty,\infty,E_B} (\overline{M}\times_0\overline{M};End(v)\otimes\Gamma_0^{\frac{1}{2}})$
and $\kappa(R_1)\in R^{\infty} \mathscr{K}^{-\infty,E_T,\infty} 
(\overline{M}\times_0\overline{M};End(v)\otimes\Gamma_0^{\frac{1}{2}})$
and both vanish to infinite order at the front face. A theory of
composition of $0$-pseudodifferential operators shows that
$$\kappa(Q^*R_1), \kappa(R_1^*J_g^{-1}R_1)\in R^{\infty} \mathscr{K}^{-\infty,E_T,E_B} (\overline{M}\times_0\overline{M}; End(V)\otimes \Gamma_0^{\frac{1}{2}}).$$ 
The conclusion follows directly.
\end{proof}

\subsection{Preliminary Computation}

Let $\overline{g}=x^2g$. Then $\overline{g}|_{T\partial \overline{M}}=g_0$.
Let $\Gamma_{\ ij}^k$ (resp. $\overline{\Gamma}_{\ ij}^k$ and
$\hat{\Gamma}_{\ ij}^k$ ) be the Christoffel symbols for $g$,
(resp. $\overline{g}$ and $g_0$), and let $\nabla_i$, (resp.
$\overline{\nabla}_i$ and $\hat{\nabla}_i$) be the Levi-Civita connection with respect to $g$ (resp. $\overline{g}$ and $g_0$). Throughout this section, let $E(z;c_1,..c_l)$ be a polynomial
in $z$ without constant term and with coefficients being functions of $(c_1,...,c_l)$. Let $T_l$ be a differential operator of order
$\leq l$ on $\partial\overline{M}$ with coefficients being functions of the
components of $g_0$ and their tangential derivatives; Let $\widetilde{T}_l$ be a differential operator of order
$\leq l$ on $\partial\overline{M}$ with coefficients being functions of the
components of $g_0, g_n$ and their tangential derivatives.
Let $T_l=0$ and $\widetilde{T}_l=0$ if $l<0$.
Recall that
$$\tilde{h}=x^{-2}(g_0'+x^2g_2'+
\cdots+x^{n-1}g'_{n-1}+x^ng_n'+O(x^{n+1})).$$
\begin{lemma} \label{lemma-transevertracefree}
There exist differential operators $\widetilde{T}_1, \widetilde{T}_0$ such that
$tr_{g_0}g_n'=\widetilde{T}_0g_0'$, $\delta_{g_0}g_n'=\widetilde{T}_1g_0'$.
\end{lemma}
\begin{proof}
According to [FG], $tr_{g_0(s)}g_n(s)=0$ and $\delta_{g_0(s)}g_n(s)=0$ for all $s\in(-\theta,\theta)$, which obviously implies that $tr_{g_0}g_n'=-tr_{g_n}g_0'=\widetilde{T}_0g_0'$ and $\delta_{g_0}g_n'=\widetilde{T}_1g_0'$.
\end{proof}
According to [FG], the one to one correspondence of normal forms for Poincar\'{e}-Einstein metrics and normal forms for ambient metrics implies that in local coordinates $(U;y^1,...,y^n)$ on the boundary, 
$$[g_2(s)]_{ij}=-P_{ij}(s)=-\tfrac{1}{(n-2)}(\hat{R}_{ij}(s)
-\tfrac{\hat{S}(s)}{2(n-1)}[g_0(s)]_{ij}),$$
and for $2\leq l\leq \frac{n-1}{2}$,
$$[g_{2l}(s)]_{ij}=\tfrac{-2^{1-l}}{(n-4)(n-6)\cdots(n-2l)}(\triangle_{g_0(s)}^{m-1}P_{ij}
-\triangle_{g_0(s)}^{m-2}{}^s\hat{\nabla}_i{}^s\hat{\nabla}_jtr_{g_0(s)}P(s))+lots,$$
where $\hat{R}_{ij}(s)$ and $\hat{S}(s)$ are the components of the Ricci curvature and scalar curvature of $g_0(s)$ and ${}^s\hat{\nabla}$ is the covariant derivative with respect to $g_0(s)$. Here $lots$ denotes quadratic and higher terms involving fewer derivatives of curvature.

\begin{lemma}By direct computation, we have
\begin{itemize}
\item[(a)] For $2\leq l\leq \tfrac{n-1}{2}$,
\begin{eqnarray*}
g_2' &=&-\tfrac{1}{n-2}(\tfrac{1}{2}\triangle_{g_0}g_0'-\delta_{g_0}\beta_{g_0}g_0')+
\tfrac{1}{2(n-1)(n-2)}(\triangle_{g_0}tr_{g_0}g_0'-tr_{g_0}\delta_{g_0}\delta^*_{g_0}g_0')g_0
+T_1g_0'
\\
g_{2l}' &=& -\tfrac{1}{2^{l-1}(n-4)\cdots (n-2l)}
\{\tfrac{1}{n-2}(\tfrac{1}{2}\triangle^l_{g_0}g_0'-\triangle^{l-1}_{g_0}\delta^*_{g_0}\beta_{g_0}g_0')
-\tfrac{1}{2(n-2)(n-1)}(\triangle_{g_0}^l tr_{g_0}g_0' \\&&
-\triangle_{g_0}^{l-1}tr_{g_0}\delta^*_{g_0}\delta_{g_0}g_0')g
-\tfrac{1}{2(n-1)}\hat{\nabla}\hat{\nabla}(\triangle^{l-1}_{g_0}tr_{g_0}g_0'
-\triangle^{l-2}_{g_0}tr_{g_0}\delta^*_{g_0}\delta_{g_0}g_0')\}
+T_{2l-1}g_0'
\end{eqnarray*}

\item[(b)] If $tr_{g_0}g_0'=0$, for $2\leq l\leq \tfrac{n-1}{2}$
\begin{eqnarray*}
g_2' &=&
-\tfrac{1}{n-2}(\tfrac{1}{2}\triangle_{g_0}g_0'-\delta^*_{g_0}\delta_{g_0}g_0')-
\tfrac{1}{2(n-1)(n-2)}(tr_{g_0}\delta_{g_0}\delta^*_{g_0}g_0')g_0
+T_1g_0'
\\
g_{2l}' &=& -\tfrac{1}{2^{l-1}(n-4)\cdots (n-2l)}
\{\tfrac{1}{n-2}(\tfrac{1}{2}\triangle^l_{g_0}g_0'-\triangle^{l-1}_{g_0}\delta^*_{g_0}\delta_{g_0}g_0')
+\tfrac{1}{2(n-2)(n-1)}(\triangle_{g_0}^{l-1}tr_{g_0}\delta^*_{g_0}\delta_{g_0}g_0')g\\&&
+\tfrac{1}{2(n-1)}\hat{\nabla}\hat{\nabla}(\triangle^{l-2}_{g_0}tr_{g_0}\delta^*_{g_0}\delta_{g_0}g_0')\}
+T_{2l-1}g_0'
\end{eqnarray*}

\item[(c)]If $tr_{g_0}g_0'=\delta_{g_0}g_0'=0$, for $1\leq l\leq
\tfrac{n-1}{2}$.
$$g_{2l}'=d_{2l}\triangle_{g_0}^lg_0'+T_{2l-1}g_0'$$
where $d_{2l}=-\frac{1}{2^l(n-2)(n-4)\cdots (n-2l)}$.
\end{itemize}
Moreover, for $1\leq l\leq \frac{n-1}{2}$, each differential operator $T_{2l-1}$ appearing above has coefficients depending on the components of $g_0$ and their derivatives up to order $2l$.
\end{lemma}
For any $H\in C^{\infty}((0,\epsilon)\times
U;S^2(T^*M))$,
denote $\Gamma_{\alpha\beta\mu}(H)=\frac{1}{2}(\partial_{\alpha}H_{\beta\mu}+\partial_{\beta}H_{\mu\beta}-
\partial_{\mu}H_{\alpha\beta})$ for simplicity.
\begin{lemma}
Suppose that $((0,\epsilon)\times U;x,y^1,...,y^n)$ is a local coordinate system near the boundary. Set $x=y^0$. Then the Christofel symbols of $g$ are
\begin{eqnarray*}
\Gamma_{\ 00}^0&=&-x^{-1},\quad \Gamma_{\ 0i}^0=\Gamma_{\
i0}^0=\Gamma_{\ 00}^i=0 \quad\textrm{for}\ i\geq 1,\\
\Gamma_{\
0j}^k&=&x^{-1}(-\delta_j^k+E(x^2;g_0,...,g_{n-1})+\tfrac{n}{2}x^n[g_0]^{kt}[g_n]_{tj}+O(x^{n+1}))\\
&&\quad \textrm{for}\ j,k\geq 1,\\
\Gamma_{\ ij}^0&=&x^{-1}([g_0]_{ij}+E(x^2;g_0,...,g_{n-1})
-\tfrac{n-2}{2}x^n[g_n]_{ij}+O(x^{n+1}))\\
&&\quad \textrm{for}\ i,j\geq 1\\
\Gamma_{\ ij}^k&=& \hat{\Gamma}_{\
ij}^k+E(x^2;g_0,...,g_{n-1})+x^n\Lambda^{k}_{\ ij}(g_0,g_n)+O(x^{n+1})\\&& \ \textrm{for}\ i,j,k\geq 1.
\end{eqnarray*}
where $\Lambda^{k}_{\ ij}(g_0,g_n)=[g_0]^{kt}\Gamma_{ijt}(g_n)-[g_n]^{kt}\Gamma_{ijt}(g_0)
=\tfrac{1}{2}[g_0]^{kt}(\hat{\nabla}_i[g_n]_{tj}+
\hat{\nabla}_j[g_n]_{it}-\hat{\nabla}_t[g_n]_{ij})$.
\end{lemma}

\begin{lemma}\label{lem-beta h} In local coordinates
$((0,\epsilon)\times U;x,y)$, suppose
$\beta_g\tilde{h}=[\beta_g\tilde{h}]_0dx+\sum_{i=1}^n[\beta_g\tilde{h}]_idy^i$.
Then
$$[\beta_g\tilde{h}]_0 = \sum_{l=0}^{\tfrac{n-1}{2}}x^{2l-1}((l-1)tr_{g_0}g_{2l}'
+T_{2l-1}g_0')+x^{n-1}\widetilde{T}_0g_0'+O(x^{n}),$$
$$[\beta_g\tilde{h}]_i = \sum_{l=0}^{\tfrac{n-1}{2}}x^{2l}([\delta_{g_0}g_{2l}']_i
+\hat{\nabla}_i(tr_{g_0}g_{2l}')+T_{2l}g_0')+x^{n}\widetilde{T}_1g_0'+O(x^{n+1}),$$
for $1\leq i\leq n$. In particular, if $\mathrm{tr}_{g_0}g_0'=\mathrm{\delta}_{g_0}g_0'=0$, then the two formulae reduce to
$$ [\beta_g\tilde{h}]_0 = \sum_{l=0}^{\tfrac{n-1}{2}}x^{2l-1}(T_{2l-1}g_0')+x^{n-1}\widetilde{T}_0g_0'+O(x^{n}),$$
$$[\beta_g\tilde{h}]_{i} =\sum_{l=0}^{\tfrac{n-1}{2}}x^{2l}(T_{2l}g_0')+x^{n}\widetilde{T}_1g_0'+O(x^{n+1}).$$
Note for $0\leq k\leq n-1$, the coefficients of each differential operator $T_k$ appearing above only depend on the components of $g_0$ and their derivatives up to order $k+1$; and for $k=0,1$, the coefficients of each differential operator $\widetilde{T}_k$ only depend on the compeoents of $g_0, g_n$ and their derivatives up to order $k$.
\end{lemma}
\begin{proof} Note that $\tilde{h}_{0\mu}=\tilde{h}_{\mu 0}=0$ for $0\leq \mu\leq n$. By direct computation,
$$tr_{g_0}\tilde{h} =
\sum_{l=0}^{\frac{n-1}{2}}x^{2l}(\mathrm{tr}_{g_0}g_{2l}'+T_{2l-1}g_0')+
x^n(\mathrm{tr}_{g_0}g_n'-\mathrm{tr}_{g_n}g_0')+O(x^{n+1}),$$
$$[\delta_g \tilde{h}]_0=
\sum_{l=0}^{\frac{n-1}{2}}x^{2l-1}(-\mathrm{tr}_{g_0}g_{2l}'+T_{2l-1}g_0')+x^{n-1}(\tfrac{n+4}{2}
\mathrm{tr}_{g_n}g_0')+O(x^n),$$
andfor $1\leq i\leq n$,
\begin{eqnarray*}
[\delta_g \tilde{h}]_i &=& \sum_{l=0}^{\frac{n-1}{2}} x^{2l}(\delta_{g_0}g_{2l}'+T_{2l}g_0')+x^n
([\delta_{g_0}g_n']_i-[\delta_{g_n}g_0']_i \\&&
+[g_0]^{jk}(\Lambda^s_{\ ij}(g_0,g_n)[g'_0]_{sk}+\Lambda^s_{\ jk}(g_0,g_n)[g'_0]_{is}))+O(x^{n+1}).\\
\end{eqnarray*}
Using $\beta_g\tilde{h}=\delta_g\tilde{h}+\tfrac{1}{2}d\mathrm{tr}_g\tilde{h}$ and by Lemma \ref{lemma-transevertracefree} we can get the formulae easily.
\end{proof}

\subsection{Choice of Gauge}
\begin{proposition}\label{prop-gauge choice}
There exist a $C^1$ family of $C^{n-1}$ diffeomorphisms $\overline{\Psi}_s: \overline{M} \longrightarrow \overline{M}$, such that $h=\frac{d}{ds}|_{s=0}(\Psi^{*}_sG_+(s))$ satisfies $\beta_gh=0$, where $\Psi_s$ is the restriction of $\overline{\Psi}_s$ to $M$. Moreover, $h\in C^{n,\alpha} (\overline{M};S^2(^0T^{*}M))$ for any $0<\alpha<1$, with asymptotic expansion
$$h=x^{-2}(h_0+xh_1+...+x^{n}h_n+x^{n+1}\log xH+O(x^{n+2}\log x)$$
 at the boundary, where $h_0=g_0'-\frac{1}{n}tr_{g_0}g_0'$, $h_i=T_ig_0'$ for $1\leq i\leq n-1$ and $h_n=g_n'-T_ng_0'$. Moreover, if $tr_{g_0}g_0'=\delta_{g_0}g_0'=0$, then $h_0=g_0'$ and $h_n=g_n'-T_{n-1}g_0'$.
\end{proposition}

\begin{proof} 
As the discussion in Section \ref{sec-geom setting}, we can assume that $X$ is the infinitesimal generator of the curve of diffeomorphisims $\overline{\Psi}_s$. Then
$$h=\frac{d}{ds}|_{s=0}(\overline{\Phi}^*_sG_+(s)) =L_{X}g+\tilde{h}.$$ 
So we only need to show that there exists a vector field $X$ on $\overline{M}$ such that
$$\beta_g(L_Xg+\tilde{h})=0.$$
Denoting by $\omega$ the dual form of $X$, then $L_X g=\delta_g^{*}\omega$, so $\omega$ should satisfy the equaiton
$$\beta_g(\delta_g^{*}\omega)+\beta_g\tilde{h}=0.$$
This is equivalent to the equation
\begin{equation}\label{eq-gauge}
\triangle_g\omega-Ric_g(\omega)+2\beta_g\tilde{h}=0.
\end{equation}
where $J_g=\triangle_g-Ric_g=\triangle_g+n\in
\textrm{Diff}_0^2(M;{}^0T^*M)$ studied in the previous subsections. To simplify the following computation, for any $q\in \partial \overline{M}$, choose normal coordinates $(U;y)$ around $p$ with respect to $g_0$. So at $p$,
$$ J_g=-(x\partial_x)^2+nx\partial_x-x^2\sum_{i=1}^{n}\partial_{y^i}^2+
\left[\begin{array}{cc}2n&0\\0&(n+1)Id_n\end{array}\right]
+2x\left[\begin{array}{cccc}0&-\partial_{y^1}&\cdots&
-\partial_{y^n}\\
\partial_{y^1}&0&\cdots
&0\\ \vdots
&\vdots&\cdots&\vdots\\\partial_{y^n}&0&\cdots&0\end{array}\right]+xD,$$
for some $D\in \textrm{Diff}_0^2(M;{}^0T^*M)$. For $\omega\in C^{\infty}(\partial\overline{M};{}^0T^*M)$, write $\omega=w_0\frac{dx}{x}+\sum_{i=1}^nw_i\frac{dy^i}{x}$. Then
\begin{eqnarray*}
J_g(x^k\omega) &=&
x^{k}(I(J_g)(k)w)+x^{k+1}(-\sum_{i=1}^n2\partial_iw_i\frac{dx}{x}
+2\partial_iw_0\frac{dy^i}{x}+T_0w)\\&&
+x^{k+2}(\triangle_{g_0}w+T_1w)+\sum_{l=3}^{\infty}x^{k+l}T_2w
\\&=& x^{k}(I(J_g)(k)w) +O(x^{k+1}),
\end{eqnarray*}
where $T_i\in\mathrm{Diff}^i(\partial\overline{M};{}^0T^*M)$. Given $\tilde{h}=x^{-2}(g'_0+x^2g'_2+\cdots+x^ng'_n+O(x^{n+1}))$ as above, according to Lemma \ref{lem-beta h},
\begin{eqnarray*}
\beta_g\tilde{h} &=&
[\sum_{l=0}^{\tfrac{n-1}{2}}x^{2l}((l-1)tr_{g_0}g_{2l}'
+T_{2l-1}g_0')+x^{n}\widetilde{T}_0g_0'+O(x^{n+1})]\frac{dx}{x}
\\&&+
[\sum_{l=0}^{\tfrac{n-1}{2}}x^{2l+1}([\delta_{g_0}g_{2l}']_i
+\hat{\nabla}_i(tr_{g_0}g_{2l}')+T_{2l}g_0')+x^{n+1}\widetilde{T}_1g_0'+O(x^{n+2})]\frac{dy^i}{x}.
\end{eqnarray*}
Hence $\beta_g\tilde{h}\in C^{\infty}(\overline{M};{}^0T^*\overline{M})$. Moreover, setting
$v^{0}=(-2tr_{g_0}g_0')\frac{dx}{x}$, 
$$2\beta_g\tilde{h}-v^{0}\in x^{1}C^{\infty}(\overline{M};{}^0T^*M)).$$
Let $\omega^{(0)}=\omega^0=-(I(J_g)(0))^{-1}v^{0}=(\frac{1}{n}tr_{g_0}g_0')\frac{dx}{x}$, we have  
$$J_g\omega^{(0)}+2\beta_g\tilde{h}-xv^{1}\in x^2C^{\infty}(\overline{M};{}^0T^*\overline{M}).$$
where $v^{1}=\sum_{i=1}^n([\delta_{g_0}g_0]_i+\frac{n+1}{n}\partial_itr_{g_0}g_0')\frac{dy^i}{dx}$.
Choose $\omega^1=-(I(J_g)(1))^{-1}v^{1}$, then $w^{(1)}=w^{(0)}+xw^1$ satisfies
$$J_gw^{(1)}+2\beta_g\tilde{h}\in x^2C^{\infty}(\overline{M};{}^0T^*\overline{M}).$$
Since for $1\leq k\leq n$, $$I(J_g)(k)=\left[\begin{array}{cc}-k^2+nk+2n&0\\0&-k^2+nk+n+1\end{array}\right]$$
is invertible, we can repeate this process up to $k=n$. In other words we can construct a sequence of approximate solutions
$\omega^{(k)}=\omega^0+x\omega^1+...+x^{k}\omega^{k}$ to equation
(\ref{eq-gauge}) in the sense that
$$J_g\omega^{(k)}+2\beta_g\tilde{h}=x^{k+1}v^{k+1}+O(x^{k+2}),$$
where $\omega^k=\omega^k_0\frac{dx}{x}+\omega^k_i\frac{dy^i}{x}$ and $v^{k}=v^k_0\frac{dx}{x}+v^k_i\frac{dy^i}{x}$.
Note that if $k\leq n-1$, for $0\leq \mu\leq n$
$\omega^k_{\mu}=T_kg_0'\in C^{\infty}(\partial
\overline{M})$ and $v^k_{\mu}=T_kg_0'\in C^{\infty}(\partial
\overline{M})$ for some differential operators $T_k$ with coefficients depending only on the components of $g_0$ and their derivatives up to order $k$. However for $k=n$ then
$$v_0^n=T_ng_0'+\widetilde{T}_0g_0',\quad v_i^n=T_ng_0'.$$
Hence $w^n=-(I(J_g)(n))^{-1}v^n$ satisfies the same property, i.e.
$$\omega_0^n=T_ng_0'+\widetilde{T}_0g_0',\quad \omega_i^n=T_ng_0'.$$
Now $\omega^{(n)}=\omega^0+x\omega^1+...+x^{n}\omega^{n}$ satisfies
$J_g\omega^{(n)}+2\beta_g\tilde{h}=x^{n+1}v^{n+1}+O(x^{n+2})$.
However,
$$I(n+1)=\left[\begin{array}{cc}2n&0\\0&0\end{array}\right]$$ is degenerate. Notice that
$(-(x\partial_x)^2+nx\partial_x+n+1)(x^{n+1}\log x)=-x^{n+1}$, so we
can choose $$\omega^{n+1}=-\tfrac{1}{2n}v^{n+1}_0\frac{dx}{x},\quad
\overline{\omega}^{n+1}=\sum_{i=1}^nv^{n+1}_i\frac{dy^i}{x}.$$
Then $\omega^{(n+1)}=\omega^{(n)}+x^{n+1}\omega^{n+1}+(x^{n+1}\log
x)\overline{\omega}$ satisfies
$$J_g(\omega^{(n+1)})+2\beta_g\tilde{h}=r
\in (x^{n+2}\log x) C^{\infty}(\overline{M};{}^0T^*\overline{M})
+x^{n+2}C^{\infty}(\overline{M};{}^0T^*\overline{M}).$$
Note that for all $k>n+1$, $I(J_g)(k)$ is invertible. Repeating
this process again, we can find a refined sequence of approximate solutions
$$\omega^{(N)}=\sum_{k=0}^{N}x^k\omega^k+\log
x\sum_{k=n+1}^{N}x^k\overline{\omega}^k$$
for $N$ arbitrarily large with
$w^k,\overline{\omega^k}\in C^{\infty}(\partial\overline{M};{}^0T^*\overline{M})$ and satisfying $$J_g\omega^{(N)}+2\beta_g\tilde{h}=\gamma^{N+1}
\in x^{N+1}\mathscr{A}^0(\overline{M};{}^0T^*\overline{M}).$$
By Proposition
\ref{prop-map property of Jg}, choosing $N$ large enough, we have
$$J_g^{-1}\gamma^{N+1}\in
x^{n+1}C^{\infty}(\overline{M};{}^0T^*\overline{M})
+x^{\frac{n+\sqrt{n^2+8n}}{2}}C^{\infty}(\overline{M};{}^0T^*\overline{M}).$$
Then $\widetilde{\omega}=\omega^{(N)}-J_g^{-1}r^{N+1}$ is a true solution to equation
(\ref{eq-gauge}). Moreover, letting $\widetilde{\omega}=\widetilde{\omega}_0dx+\sum_{i=1}^n\widetilde{\omega}_idy^i$, then near the boundary, 
$$\widetilde{\omega}_0=x^{-1}(\omega_0^0+x\omega_0^1+\cdots +x^{n+1}\omega^{n+1}_0+O(x^{n+\epsilon})),$$
$$\widetilde{\omega}_i=x^{-1}(x\omega_i^1+\cdots +x^{n}\omega^{n}_i+x^{n+1}\log x \bar{\omega}^{n+1}_i+O(x^{n+\epsilon})),$$
where $\omega_0^0=\tfrac{1}{n}\mathrm{tr}_{g_0}g_0'$.
Hence, the dual vector filed $X\in
C^{n-1}(\overline{M};{}^0TM)=xC^{n-1}(\overline{M};T\overline{M})$, which implies that $X$ vinishes at the boundary. Hence the 1-parameter family of diffeomorphisms
generated by $X$ is $C^{n-1}$ over $\overline{M}$ which restricts
to the identity on the boundary.
By definition,
$$h=\delta^{*}_g\widetilde{\omega}+\tilde{h}\in C^{n-1}(\overline{M};S^2(^0T^{*}M))=
x^{-2}C^{n-1}(\overline{M};S^2T^{*}\overline{M})$$
which satisfies the gauge condition $\beta_gh=0$. More explicitly, using the formula
$$[\delta^{*}_g\widetilde{\omega}]_{\alpha\beta}=\frac{1}{2}(\nabla_{\alpha}\widetilde{\omega}_{\beta}
+\nabla_{\beta}\widetilde{\omega}_{\alpha})=\frac{1}{2}(\partial_{\alpha}\widetilde{\omega}_{\beta}
+\partial_{\beta}\widetilde{\omega}_{\alpha})-\Gamma_{\alpha\beta}^{\mu}\widetilde{\omega}_{\mu}$$
and lemmas in the previous subsection, we can compute the components of $\delta^{*}_g\widetilde{\omega}$ as
$$[\delta^{*}_g\widetilde{\omega}]_{00} =  x^{-2}(\sum_{k=1}^{n-1}kx^k\omega_0^k +nx^n\omega_0^n+(n+1)x^{n+1}\omega_0^{n+1}+O(x^{n+1+\epsilon})),$$
$$[\delta^{*}_g \widetilde{\omega}]_{0i} = x^{-2}(\sum_{k=1}^{n-1}x^kF_k
+x^n(\tfrac{n+1}{2}\omega_i^n+\tfrac{1}{2}\partial_i\omega_0^0)+x^{n+1}\log x (\tfrac{n+2}{2}\bar{\omega}_i^{n+1})+O(x^{n+1})),$$
$$[\delta^{*}_g\widetilde{\omega}]_{ij} = x^{-2}(-\omega_0^0[g_0]_{ij}+\sum_{k=1}^{n-1}x^kF_k
+x^{n}(F_{n-1}-\omega_0^n[g_0]_{ij}+\tfrac{n-2}{2}\omega_0^0[g_n]_{ij})+O(x^{n+1}))$$
where $F_k$ is a function of $\omega_{\mu}^0,\omega_{\mu}^1,...,\omega_{\mu}^{k}$ for $0\leq\mu\leq n$ and the components of $g_0$ and their derivatives. Clearly, $F_k$ is linear in $\omega_{\mu}^i$ for $1\leq i\leq k$ and $0\leq \mu\leq n$. Therefore, writing
$$h=x^{-2}(h_0+xh_1+...+x^{n}h_n+x^{n+1}\log xH+O(x^{n+1})),$$
we established:
\begin{itemize}
\item[(a)]
$h_0=g_0'-\tfrac{1}{n}(tr_{g_0}g_0')g_0$, which implies that $tr_{g_0}h_0=0$;
\item[(b)]
For $1\leq k\leq n-1$, $h_k=g'_k+T_kg_0'$ for some differential opeator $T_k$ of order $k$ with coefficients only depending on $g_0$ and their derivatives ($g'_k=0$ if $k$ is odd.);
\item[(c)]
$h_n=g_n'+T_ng_0'+\widetilde{T}_0g_0'$. $\tilde{T}_0$ is of order $0$ with coefficients depending on the components of $g_0$ and $g_n$;
\item[(d)]
$H=T_{n+1}g_0'+\widetilde{T}_1g'_0$. The coefficients of $\widetilde{T}_1$ depends on the components of $g_0, g_n$ and their first order derivatives.
\end{itemize}
Furthermore, if $tr_{g_0}g_0'=\delta_{g_0}g_0'=0$, then in the construction of $\widetilde{\omega}$, we have $w_{\mu}^{k}=T_{k-1}g_0'$ for $0\leq k\leq n-1$ and $w_{\mu}^{n}=T_{n-1}g_0'+\widetilde{T}_0g'_0$. Hence in this case, (a),(c) and (d) are replaced by
\begin{itemize}
\item[(a')] $h_0=g_0'$;
\item[(c')] $h_n=g_n'+T_{n-1}g_0'+\widetilde{T}_0g_0'$;
\item[(d')] $H=T_{n}g_0'+\widetilde{T}_1g'_0$.
\end{itemize}
\end{proof}

\section[]{Linearized Einstein Equation}\label{sec-linear einstein op}
After fixing the gauge, the linearized Einstein equation becomes
elliptic, i.e.
$$L_gh=0,\ \mathrm{where}\ L_g=\triangle_{g}h+2R+2n.$$
In this section, we want to study
$L_g$  as was done for $J_g$ and with
further analysis of the Poisson operator and scattering matrix
to characterize the solution to the linearized Einstein equation.
For this goal, we can study a family of operators instead. Denote
$$L_g(\lambda)=\triangle_{g}h+2R+2n-\lambda(n-\lambda),\quad\mathrm{for}\ \lambda\in
D_{\varepsilon},$$ where $D_{\varepsilon}$ is a simply connected neighborhood of $[n,\infty)$ in $\mathbb{C}$ defined by
\begin{equation}\label{holoregion}
D_{\varepsilon}=\{\lambda\in \mathbb{C}:
\Re\lambda-n>-\varepsilon,\ |\Im\lambda|< \varepsilon\}
\end{equation}
for some $\varepsilon>0$ and small. Throughout this section, the parameter $\lambda$ will be in such a domain $D_{\varepsilon}$ for some $\varepsilon >0$ and small enough. 

\subsection[]{Indicial Operator}
For any $q\in \partial \overline{M}$, let $([0,\epsilon)\times U; x,y)$ be the
local coordinates. At the boundary, $S^2({}^0T^*M)|_{\partial\overline{M}}$ has a decomposition w.r.t. the metrics $g$ and $g_0$, i.e.
$S^2({}^0T^*M)|_{\partial\overline{M}}=\mathscr{V}^0\oplus\mathscr{V}^1\oplus\mathscr{V}^2\oplus\mathscr{V}^3$,
where
\begin{eqnarray*}
\mathscr{V}^0 &=& span\{\tfrac{dx}{x}\otimes \tfrac{dx}{x}\},\\
\mathscr{V}^1 &=& \{h\in S^2({}^0T^*\partial \overline{M});tr_{g_0}h=0\}, \\
\mathscr{V}^2 &=& span\{n\tfrac{dx}{x}\otimes
\tfrac{dx}{x}-x^{-2}g_0\},\\
\mathscr{V}^3 &=& span \{\tfrac{dx}{x}\otimes
\tfrac{dy^i}{x}+\tfrac{dy^i}{x}\otimes \tfrac{dx}{x}:1\leq i\leq n\}.
\end{eqnarray*}
Note that $\mathscr{V}^3$ doesn't depend on the choice of local
boundary coordinates $(U;y)$. Hence such decomposition is well
defined on $\partial \overline{M}$. Moreover, it can be extended to a neiborhood of the boundary up to $O(x^2)$, i.e. we can choose $\Lambda=2$ in Condition \ref{cond-ET}. That is because
$$g=x^{-2}(dx^{2}+g(x))=x^{-2}(dx^2+g_0+O(x^{2}))$$
in a neighborhood of $\partial \overline{M}$. If we replace $g_0$ by
$g(x)=g_0+O(x^2)$ in the above definition of $\mathscr{V}^i$, the
decomposition can be extended to whole neighborhood of
$\partial \overline{M}$. The indicial operator corresponding to such a
decomposition is
$$I(L_g(\lambda))(s)=-s^2+ns+\left [\begin{array}{cccc}
2n & 0&0 &0\\
0&0 & 0&0\\
0&0 & 2n&0 \\
0&0 &0 & n+1
\end{array}\right]-\lambda(n-\lambda), \quad \forall\ s\in\mathbb{C}.$$
Hence for $\lambda\in D_{\varepsilon}$, the corresponding indicial functions and indicial roots are:
\begin{eqnarray*}
f_0(s,\lambda)&=&-s^2+ns+2n-\lambda(n-\lambda)=0
\\&\Rightarrow& \
s_0(\lambda)=\tfrac{n-\sqrt{(2\lambda-n)^2+8n}}{2},
\ s^0(\lambda)=\tfrac{n+\sqrt{(2\lambda-n)^2+8n}}{2},\\
f_1(s,\lambda)&=& -s^2+ns+2n-\lambda(n-\lambda)=0 \\
&\Rightarrow& \ s_1(\lambda)=n-\lambda,\ s^1(\lambda)=\lambda,\\
f_2(s,\lambda)&=& -s^2+ns+2n-\lambda(n-\lambda)=0 \\
&\Rightarrow& \ s_2(\lambda)=s_0(\lambda), s^2(\lambda)=s^0(\lambda),\\
f_3(s,\lambda)&=& -s^2+ns+n+1-\lambda(n-\lambda)=0 \\
&\Rightarrow& \
s_3(\lambda)=\tfrac{n-\sqrt{(2\lambda-n)^2+4n+4}}{2},
s^3(\lambda)=\tfrac{n+\sqrt{(2\lambda-n)^2+4n+4}}{2}.
\end{eqnarray*}
Then if $\varepsilon>0$ is small enough, we can fix a choice of
square roots such that $s_i(\lambda)$ and $s^i(\lambda)$ are
holomorphic in $D_{\varepsilon}$. Note that $\Re\lambda=\Re s^1(\lambda)<
\Re s^3(\lambda)<\Re s^0(\lambda)=\Re s^2(\lambda)<\Re\lambda+2$ for all
$\lambda\in D_{\varepsilon}$.

\subsection[]{Normal Operator}
Consider the half-plane model of hyperbolic space
$(\mathbb{H}^{n+1},r)$, where
$$\mathbb{H}^{n+1}=\mathbb{R}^+_x\times \mathbb{R}^n_y,\quad
r=x^{-2}(dx^2+dy^2).$$
At each $q\in \partial\overline{M}$, let
$([0,\epsilon)\times U;x=y^0,y)$ be a coordinate system.
Then $S^2({}^0T_q^*M)$ lifts to a trivial bundle over $T^+_qM\cong
\mathbb{H}^{n+1}$, i.e. we can use $\{
\tfrac{dy^i}{x}\otimes\tfrac{dy^j}{x}+\tfrac{dy^j}{x}\otimes\tfrac{dy^i}{x}:0\leq
i,j\leq n \}$ as a basis of $S^2({}^0T^*\mathbb{H}^{n+1})$. And the
decomposition corresponding to 
$\mathscr{V}^0\oplus\mathscr{V}^1\oplus\mathscr{V}^2\oplus\mathscr{V}^3$
is
\begin{eqnarray*}
\mathscr{W}^0&=&\{h:h_{00}=h_{ii}, h_{0j}=h_{ij}=0,\ \mbox{for}\ i,j\geq 1\},\\
\mathscr{W}^1&=&\{h:\sum_{i\geq1}h_{ii}=0, h_{00}=h_{0j}=0,\ \mbox{for}\ j\geq 1\},\\
\mathscr{W}^2&=&\{h:h_{00}=-nh_{ii},h_{0j}=h_{ij}=0,\ \mbox{for}\ i,j\geq 1\},\\
\mathscr{W}^3&=&\{h:h_{00}=h_{ij}=0,\ \mbox{for}\ i,j\geq 1\}.
\end{eqnarray*}
The fibre dimensions $d_i$ of $\mathscr{W}^{i}$ are
$$d_0=d_2=1,\ d_1=\tfrac{n(n-1)}{2}+n-1,\ d_3=n, \ N=d_0+d_1+d_2+d_3=\tfrac{(n+1)(n+2)}{2}.$$
Then the normal operator can be expressed as
\begin{eqnarray*}
N_q(L_g(\lambda))&=&-(x\partial_x)^2+nx\partial_x-x^2\sum_{i=1}^n\partial_{y^i}^2
+2x\left [\begin{array}{cccc}
0 & 0&0 &0\\
0&0 & 0& E_{13}\\
0&0 & 0& E_{23} \\
0&E_{31} & E_{32}& 0
\end{array}\right] \\ &&
+\left [\begin{array}{cccc}
2n & 0&0 &0\\
0&0 & 0&0\\
0&0 & 2n&0 \\
0&0 &0 & n+1
\end{array}\right]-\lambda(n-\lambda),
\end{eqnarray*}
where $E_{ij} \in \textrm{Diff}^1(\mathbb{R}^n) \otimes
Hom(\mathscr{W}^i;\mathscr{W}^j)$. More explicitly, for
$h=\sum_{i,j=0}^nh_{ij}\tfrac{dy^i}{x}\otimes\tfrac{dy^j}{x}$,
$h_{ij}=h_{ji}$,
\begin{eqnarray*}
E(h) &=&
(-2\sum_{k=1}^n\partial_kh_{k0})\tfrac{dx}{x}\otimes\tfrac{dx}{x}
+\sum_{i,j=1}^n(\partial_ih_{0j}+\partial_jh_{i0})\tfrac{dy^i}{x}\otimes\tfrac{dy^j}{x}\\&&
+\sum_{j=1}^n(-\sum_{k=1}^n\partial_kh_{kj}+\partial_jh_{00})
(\tfrac{dx}{x}\otimes\tfrac{dy^j}{x}+\tfrac{dy^j}{x}\otimes\tfrac{dx}{x}).
\end{eqnarray*}
In particular, $E(h)=0$ if $h\in
C^{\infty}(\mathbb{H}^{n+1};\mathscr{W}^1)$ satisfies
$\delta_rh=0$. In this case, the normal operator becomes diagonal
and so effectively reduces to $\triangle_r$ acting on functions.

Notice that $N_q(L_g)=L_r=\triangle_r+2R_r+2n$. We deal with the
normal operator $N_q(L_g(\lambda))$ almost in the same way as was done for $J_g$. So we omit the proofs of some statements if they are essentially the same as in the last section.

For any $u\in x^{a}H_0^2(\mathbb{H}^{n+1};\mathbb{C}^N)$ for some
$a\in \mathbb{R}$, let $\hat{u}(x,\xi)=\mathscr{F}_{y\rightarrow
\xi}(u)(x,\xi)$ be the Fourier Transform w.r.t the variable $y$. Then $u(x,t)$ is a solution of
$(L_r-\lambda(n-\lambda))u=0$ if and only if
$\hat{u}(x,\xi)$ is a solution of the ordinary equation
$(\widehat{L}_r-\lambda(n-\lambda))\hat{u}=0$, where
$$\widehat{L}_r=-(x\partial_x)^2+nx\partial_x+x^2|\xi|^2+\left[\begin{array}{cccc}
2n&0&0&0\\0&0&0&0\\0&0&2n&0\\0&0&0&n+1\end{array}\right]+2ix\tilde{E}(\xi)$$
where $\tilde{E}$ is homogeneous of order 1 in $\xi$. Moreover,for
any $\xi\neq 0$, by changing variable to
$\bar{x}=x|\xi|$ and $\hat{\xi}=\xi/|\xi|$,
$(\widehat{L}_r-\lambda(n-\lambda))v(x|\xi|,\hat{\xi})=0$ is
equivalent to
$(\widetilde{L}_r-\lambda(n-\lambda))v(\bar{x},\hat{\xi})=0$,
where
$$\widetilde{L}_r=-(\bar{x}\partial_{\bar{x}})^2+n\bar{x}\partial_{\bar{x}}
+\bar{x}^2+\left[\begin{array}{cccc}
2n&0&0&0\\0&0&0&0\\0&0&2n&0\\0&0&0&n+1\end{array}\right]+2i\bar{x}\tilde{E}(\hat{\xi}).$$

\begin{lemma} For $a\in (\tfrac{n-\sqrt{n^2-8}}{2},\tfrac{n+\sqrt{n^2-8}}{2})$, the map
$$N_q(L_g)=L_r: x^aH^2_0(\mathbb{H}^{n+1};\mathbb{C}^N \otimes\Gamma_0^{\frac{1}{2}})\longrightarrow
x^aL^2_0(\mathbb{H}^{n+1};\mathbb{C}^N \otimes\Gamma_0^{\frac{1}{2}})$$ is continuous and positive.
\end{lemma}
\begin{proof}
For any $h\in H^2_0(\mathbb{H}^{n+1}$,
$N_q(L_g)(x^ah)=a(n-a)x^ah+x^aN_q(L_g)(h)$. Let $h=h^1+h^2$, where
$h^1=ur$ for some $u\in C^{\infty}(\mathbb{H}^{n+1})$ and
$tr_{r}h^2=0$, then
$L_r(h)=(\triangle_r+2n)h^1+(\triangle_r-2)h^2$. Hence
\begin{eqnarray*}
\langle N_q(L_g)(x^ah),x^ah\rangle_{x^aL^2_0}&=& \langle
a(n-a)h+L_rh,h\rangle_{L_0^2} \\
&=& \|\nabla h^1\|^2_{L_0^2}+\|\nabla
h^2\|^2_{L_0^2}+(-a^2+na+2n)\|h^1\|^2_{L_0^2}\\
&&+(-a^2+na-2)\|h^2\|_{L^2_0}\\
&\geq & 0
\end{eqnarray*}
with equality if and only if $h=0$. The continuity is obvious.
\end{proof}
Note that since $n\geq 3$, $[1,n-1]
\subset (\tfrac{n-\sqrt{n^2-8}}{2},\tfrac{n+\sqrt{n^2-8}}{2})
\subset[\tfrac{n-\sqrt{n^2-8}}{2},\tfrac{n+\sqrt{n^2-8}}{2}] \subset (0,n)$

\begin{lemma}\label{lem-no L2 ef for Ng}
There exist  $\delta>0$ and $\varepsilon>0$, such that for $a\in
(\tfrac{n}{2}-\delta,\tfrac{n}{2}+\delta)$ and
$\lambda\in D_{\varepsilon}$, where $D_{\varepsilon}$ is defined by (\ref{holoregion}), the map
$$N_q(L_g(\lambda))=L_r(\lambda): x^{a}H^2_0(\mathbb{H}^{n+1};\mathbb{C}^N  \otimes\Gamma_0^{\frac{1}{2}}) \longrightarrow x^{a}L^2_0(\mathbb{H}^{n+1};\mathbb{C}^N \otimes\Gamma_0^{\frac{1}{2}})$$ is continuous and injective.
\end{lemma}
\begin{proof}
According to the proof of last Lemma,
$$\langle N_q(L_g(\lambda))(x^ah),x^ah\rangle_{x^aL^2_0}=\|\nabla h\|^2_{L^2_0}+(-a^2+na-2-\lambda(n-\lambda))\|h\|^2_{L^2_0}
+2(n+1)\|h^1\|^2_{L^2_0}.$$
If $\delta>0$ and $\varepsilon>0$ are small enough, then for $\lambda\in D_{\varepsilon}$ and $a\in
(\tfrac{n}{2}-\delta,\tfrac{n}{2}+\delta)$, either
$\Re(\tfrac{n^2}{4}-\delta^2-2-\lambda(n-\lambda))>0$ or $\Im(\tfrac{n^2}{4}-\delta^2-2-\lambda(n-\lambda))\neq 0$ holds. Therefore the map is injective.
\end{proof}
Let $\Gamma_0^{\frac{1}{2}}(\mathbb{R}^+)$ be the half density over $\mathbb{R}_x^+$ trivialized by $|\frac{dx}{x^{n+1}}|^{\frac{1}{2}}$. 
\begin{corallary} \label{cor-no ef for hat and tilde}
Let $\varepsilon$ and $\delta$ be the same as in $\mathrm{Lemma\ \ref{lem-no L2
ef for Ng}}$, then for $a\in
(\tfrac{n}{2}-\delta,\tfrac{n}{2}+\delta)$ and
$\lambda\in D_{\varepsilon}$, the two maps
$$\widehat{L}_r(\lambda):x^{a}H^2_0(\mathbb{R}^+;\langle\xi\rangle^{-2}
L^2(\mathbb{R}^n;\mathbb{C}^N)\otimes \Gamma_0^{\frac{1}{2}}(\mathbb{R}^+))\longrightarrow
x^{a}L^2_0(\mathbb{R}^+;
L^2(\mathbb{R}^n;\mathbb{C}^N)\otimes \Gamma_0^{\frac{1}{2}}(\mathbb{R}^+)),$$
$$\widetilde{L}_r(\lambda):\bar{x}^{a} H^2_0(\mathbb{R}^+;\mathbb{C}^N \otimes\Gamma_0^{\frac{1}{2}}(\mathbb{R}^+))
\longrightarrow
\bar{x}^{a}L^2_0(\mathbb{R}^+;\mathbb{C}^N \otimes \Gamma_0^{\frac{1}{2}}(\mathbb{R}^+))$$ 
are continuous and
injective.
\end{corallary}

The linear ODE $\widetilde{L}_r(\lambda)v=0$ has a regular
singular point at $\bar{x}=0$ and an irregular singular point at
$\bar{x}=\infty$ and no other singular points. So a formal
solution at $\bar{x}=0$ is an actual solution and for each formal
solution at $\bar{x}=\infty$, there is an actual solution which
has asymptotic expansions at $\bar{x}=\infty$ corresponding to it.
So to find out all the $2N$ linear independent solutions for
$\widetilde{L}_r(\lambda)=0$, we only need to find out all
linear independent formal solutions at $\bar{x}=0$. Since such
solutions can't blow up at any point in $(0,\infty)$, we can
extend them to the half real line $(0,\infty)$.

Notice that for $\lambda=n$, the difference between two of the
indicial roots may be an integer. Then for $\lambda\in D_{\varepsilon}$, where $\varepsilon>0$ is small
enough, the $2N$ linear independent formal solutions, which are true
solutions, are constructed as follows:
\begin{itemize}
\item [(1)] For the basis $e^0\in \mathscr{U}^0$,
$v^0=\bar{x}^{\frac{n}{2}}I_{\nu(\lambda)}(\bar{x})e^0$ and
$v_0=\bar{x}^{\frac{n}{2}}K_{\nu(\lambda)}(\bar{x})e^0$ are two
independent solutions with leading term $\bar{x}^{s^0(\lambda)}e$
and $\bar{x}^{s_0(\lambda)}e$ respectively, where
$I_{\nu(\lambda)}(\bar{x})$ and $K_{\nu(\lambda)}(\bar{x})$ are
the modified Bessel function of the first kind and Macdonald's
function, respectively, of order
$\nu(\lambda)=\tfrac{\sqrt{(2\lambda-n)^2+8n}}{2}$.

\item [(2)] For any basis $e^1\in \mathscr{U}^1$, for $\lambda\in D_{\varepsilon}$, note that $f_i(s^1(\lambda)+k)=0$ if and only if $i=3$,
$\lambda=n$ and $k=1$. So for $\lambda\neq n$, we can construct a
solution
$u^1=\bar{x}^{s^1(\lambda)}e^1+O(\bar{x}^{s^1(\lambda)+1})$.
However, $u^1$ blows up at $\lambda=n$, i.e, $u^1$ is meromorphic
in $D_{\varepsilon}$, with a pole of first order at
$\lambda=n$. Define $v^1=\partial_{\lambda}((\lambda-n)u^1)$. Then
$v^1$ is a solution with leading term $\bar{x}^{s^1(\lambda)}e^1$.

Similarly for $\lambda\in D_{\varepsilon}$, $f_i(s_1(\lambda)+k)=0$ if and
only if $\lambda=n$, $i=1,k=n$ or $\lambda=n$, $i=3, k=n+1$. For
$\lambda\neq n$, we want to construct a formal solution
$u_1=\bar{x}^{s_1(\lambda)}e^1+O(\bar{x}^{s_1(\lambda)+1})$. $u_1$
is meromorphic in $D_{\varepsilon}$ with a pole of second
order at $\lambda=n$. Define
$v_1=\partial_{\lambda}^2((\lambda-n)^2u_1)$. Then $v_1$ is a
solution with leading term $\bar{x}^{s_1(\lambda)}e^1$.

\item [(3)] For the basis $e^2\in \mathscr{U}^2$, since
$f_i(s_2(\lambda)+k)\neq 0$ for all positive integer $k$ and all
$i=0,1,2,3$, we can construct two solutions with leading term
$v_2=\bar{x}^{s_2(\lambda)}e^2+O(\bar{x}^{s_2(\lambda)+1})$ and
$v^2=\bar{x}^{s^2(\lambda)}e^2+O(\bar{x}^{s^2(\lambda)+1})$
respectively.

\item [(4)] For any basis $e^3\in \mathscr{U}^3$, first note that
$f_i(s^3(\lambda)+k)\neq 0$ for all $k\in \mathbb{N}$ and
$i=0,1,2,3$. Hence we can construct a solution
$v^3=\bar{x}^{s^3(\lambda)}+O(\bar{x})$. \\
For $\lambda\in D_{\varepsilon}$,
$f_i(s_3(\lambda)+k)=0$ if and only if $\lambda=n, i=1, k=1$ or
$\lambda=n, i=1, k=n+1$ or $\lambda=n, i=3, k=n+2$. For
$\lambda\neq n$, we can construct a solution
$u_3=\bar{x}^{s_3(\lambda)}e^3+O(\bar{x}^{s_3(\lambda)+1})$. $u_3$
is meromorphic in $D_{\varepsilon}$ with a pole of third
order at $\lambda=n$. Define $v_3=\partial_{\lambda}^3
((\lambda-n)^3u_3)$. Then $v_3$ is a solution with leading term
$\bar{x}^{s_3(\lambda)}e^3$.
\end{itemize}
All the solutions $v_i$ and $v^i$ are holomorphic in
$D_{\varepsilon}$ for $i=0,1,2,3$. On the other hand, at $\bar{x}=\infty$, the
solutions constructed above have leading term either
$\bar{x}^{\frac{n}{2}}e^{-\bar{x}}$ or
$\bar{x}^{\frac{n}{2}}e^{+\bar{x}}$. By Corallary \ref{cor-no ef
for hat and tilde},  the behavior of solutions constructed above is
$$v_i\sim \bar{x}^{\frac{n}{2}}e^{-\bar{x}},\quad v^i\sim \bar{x}^{\frac{n}{2}}e^{\bar{x}}$$
as $\bar{x}\rightarrow \infty$ for $i=0,1,2,3$.

For $\lambda\in \mathbb{R}\cap D_{\varepsilon}$, since
$\widetilde{L}_r(\lambda)$ is self-adjoint w.r.t
$\bar{x}^{-n-1}d\bar{x}$, using the standard method again in [CL]
we get the Green function for $\widetilde{L}_r(\lambda)$, which is
$$G_1(\lambda,\bar{x},\bar{x}',\hat{\xi})=U(\lambda,\bar{x},\hat{\xi})V^*(\lambda,\bar{x}',\hat{\xi})H(\bar{x}'-\bar{x})
+V(\lambda,\bar{x},\hat{\xi})U^*(\lambda,
\bar{x}',\hat{\xi})H(\bar{x}-\bar{x}'),$$ where $U(\lambda,
\bar{x},\hat{\xi})$ and $V(\lambda, \bar{x},\hat{\xi})$ are
$N\times N$ matrices whose columns are linear combination of $v^i$
and $v_i$, namely,
\begin{equation}\label{eq-green id}
\widetilde{L}_r(\lambda)G_1(\lambda, \bar{x},\bar{x}',\hat{\xi})
=\bar{x}^{n+1}\delta(\bar{x}-\bar{x}').
\end{equation} Moreover,
\begin{equation}\label{eq-wronsik id}
U'(\lambda, \bar{x},\hat{\xi})V^*(\lambda,
\bar{x}',\hat{\xi})-V'(\lambda, \bar{x},\hat{\xi})U^*(\lambda,
\bar{x}',\hat{\xi}) =-\frac{1}{2}x^{n-1}Id. \end{equation} This
implies that
$$U(\lambda, \bar{x},\hat{\xi})=(a_0v^0,a_1v^1,a_2v^2,,a_3v^3),\quad
V(\lambda, \bar{x},\hat{\xi})=(b_0v_0,b_1v_1,b_2v_2,b_3v_3),$$ for
some $a_i\neq 0$, $b_i\neq 0$ for $0\leq i\leq 3$. Note $U$, $V$,
$a_i$, $b_i$ can be extended holomorphicly to $D_{\varepsilon}$
for $\varepsilon>0$ small and so are $U^*$, $V^*$. Denote
$\widetilde{U},\widetilde{V}$ be the extension of $U^*$, $V^*$.
Remark that the conjugate relation doesn't holds any more, but
only $\widetilde{U}(\lambda)=(U(\bar{\lambda}))^*$,$\widetilde{V}(\lambda)=(V(\bar{\lambda}))^*$.
Hence $G_1$ can be extended to whole $D_{\varepsilon}$ such that
(\ref{eq-green id}) and (\ref{eq-wronsik id}) still hold in such
domain by analytic theorem.

\begin{lemma} If $\varepsilon>0$ and $\delta>0$ are small
enough, then for all $\lambda\in D_{\varepsilon}$ and $a\in
(\tfrac{n}{2}-\delta,\tfrac{n}{2}-\frac{\delta}{2})$,
$G_1(\lambda, \bar{x},\bar{x}',\hat{\xi})$ induces a bounded
integral transformation on
$x^aL^2(\mathbb{R}^+;\mathbb{C}^N \otimes\Gamma_0^{\frac{1}{2}}(\mathbb{R}^+))$.
\end{lemma}

\begin{lemma}If $\varepsilon>0$ and $\delta>0$ are small
enough, then for all $\lambda\in D_{\varepsilon}$ and $a\in
(\tfrac{n}{2}-\delta,\tfrac{n}{2}-\frac{\delta}{2})$, there exists
a unique kernel $G_2(\lambda, x,x',\xi)$ which is bounded on
$x^aL^2_0(\mathbb{R}_x^+;L^2(\mathbb{R}_{\xi}^n;\mathbb{C}^{N})\otimes\Gamma_0^{\frac{1}{2}}(\mathbb{R}^+))$
satisfying
$$\widehat{L}_r(\lambda)G_2(\lambda, x,x',\xi)=x^{n+1}\delta(x-x')\otimes Id.$$
\end{lemma}
\begin{proof}
Let $G_2(\lambda,x,x',\xi)=|\xi|^{-n}G_1(\lambda,
x|\xi|,x'|\xi|,\hat{\xi})$. The remainder of the proof exactly follows that of Lemma \ref{lem-L2 boundness of G2}.
\end{proof}

\begin{proposition}If $\varepsilon>0$ and $\delta>0$ are small
enough, then for all $\lambda\in D_{\varepsilon}$ and $a\in
(\tfrac{n}{2}-\delta,\tfrac{n}{2}-\frac{\delta}{2})$,
$$L_r(\lambda):x^aH_0^2(\mathbb{H}^{n+1};\mathbb{C}^{N}\otimes \Gamma_0^{\frac{1}{2}}(\mathbb{R}^+))\longrightarrow
x^aL_0^2(\mathbb{H}^{n+1};\mathbb{C}^{N}\otimes\Gamma_0^{\frac{1}{2}}(\mathbb{R}^+))$$ is invertible with
Green function
\begin{equation}\label{eq-green fuction for NLg}
G(\lambda, x,x',y,y')=\frac{1}{(2\pi)^n}\int_{\mathbb{R}^n}
|\xi|^{-n}G_1(\lambda,
x|\xi|,x'|\xi|,\hat{\xi})e^{i\xi(y-y')}d\xi.
\end{equation}
Moreover, $G(\lambda)\in
\Psi_0^{-2,E_T(\lambda),E_B(\lambda)}(\mathbb{H}^{n+1};M(N,\mathbb{C})\otimes
\Gamma_0^{\frac{1}{2}})$, where
$$E_T(\lambda)=\left[\begin{array}{cccc}
 s^0(\lambda)&  s^1(\lambda)+2& s^2(\lambda)+2& s^3(\lambda)+2\\
 s^0(\lambda)+2& s^1(\lambda)& s^2(\lambda)+1& s^3(\lambda)+1\\
 s^0(\lambda)+2& s^1(\lambda)+1& s^2(\lambda)& s^3(\lambda)+1\\
 s^0(\lambda)+2& s^1(\lambda)+1& s^2(\lambda)+1&
s^3(\lambda)\end{array}\right]$$ and
$E_B(\lambda)=E_T(\lambda)^t$.
\end{proposition}

As in the last section, let $x$ be the boundary defining function of the
hyperbolic ball $B^{n+1}$ and $\rho$ and $\rho'$ be boundary
defining functions for $F_q$. Let $e_i(\lambda)=(s^i(\lambda)+1,...,s^{i}(\lambda),...,s^i(\lambda)+1)$
be the constant vector corresponding to the $i$-th column of $E_T(\lambda)$.

\begin{corallary} \label{ei-lambda}
For $\lambda\in D_{\varepsilon}$, where $\varepsilon>0$ is small enough, the map $$G(\lambda):\dot{C}^{\infty}(\overline{B}^{n+1};\mathbb{C}^{N}\otimes
\Gamma_0^{\frac{1}{2}})\longrightarrow
\sum_{i=0}^{3}\mathscr{A}_{phg}^{e_i(\lambda)}(\overline{B}^{n+1};\mathbb{C}^{N}\otimes
\Gamma_0^{\frac{1}{2}})$$
is continuous.
\end{corallary}

\begin{corallary}
For $\lambda\in D_{\varepsilon}$, where $\varepsilon>0$ is small enough, the map  $$G(\lambda):\dot{C}^{\infty}(\overline{F_q};M(N,\mathbb{C})\otimes
\Gamma_0^{\frac{1}{2}})\longrightarrow
\mathscr{A}_{phg}^{E_T(\lambda),E_B(\lambda)}(\overline{F_q};M(N,\mathbb{C})\otimes
\Gamma_0^{\frac{1}{2}})$$
is continuous.
\end{corallary}

\subsection[]{Parametrix and Resolvent}

\begin{proposition}
If $\varepsilon>0$ and $\delta>0$  are small
enough, then for all $\lambda\in D_{\varepsilon}$ and $a\in
(\tfrac{n}{2}-\delta,\tfrac{n}{2}-\frac{\delta}{2})$,
\begin{eqnarray*}
L_g(\lambda)=L_g-\lambda(n-\lambda):x^aH_0^2(M;S^2({}^0T^*M)\otimes\Gamma_0^{\frac{1}{2}})
\longrightarrow
x^aL_0^2(M;S^2({}^0T^*M)\otimes\Gamma_0^{\frac{1}{2}})
\end{eqnarray*} is Fredholm. There exist right and left
parametrices
$$Q_r(\lambda), Q_l(\lambda)\in
\Psi_0^{-2,E_T(\lambda),E_B(\lambda)}(M;S^2({}^0T^*M)\otimes\Gamma_0^{\frac{1}{2}}),
$$ such that
$$L_g(\lambda)Q_r(\lambda)-Id=R_r\in
\Psi_0^{-\infty,\infty,E_B(\lambda)}(M;S^2({}^0T^*M)\otimes\Gamma_0^{\frac{1}{2}}),$$
$$Q_l(\lambda)L_g(\lambda)-Id=R_l\in
\Psi_0^{-\infty,E_T(\lambda),\infty}(M;S^2({}^0T^*M)\otimes\Gamma_0^{\frac{1}{2}}).$$
$Q_l(\lambda), Q_r(\lambda), R_l(\lambda), R_r(\lambda)$ are
holomorphic in $D_{\varepsilon}$.
\end{proposition}
\begin{proof}
Based on the analysis for the normal operator in the last subsection the
existence of the right parametrix $Q_r(\lambda)$ follows directly by Theorem
\ref{thm-construct parametrix}. By Proposition \ref{prop-map
property of 0pseudo},
$$Q_r(\lambda):x^aL_0^2(M;S^2({}^0T^*M)\otimes\Gamma_0^{\frac{1}{2}})
\longrightarrow
x^aH_0^2(M;S^2({}^0T^*M)\otimes\Gamma_0^{\frac{1}{2}})$$ is
bounded and
$$R_r(\lambda)=L_g(\lambda)Q_r(\lambda)-Id\in \Psi_0^{-\infty,\infty,E_B(\lambda),}(M;S^2({}^0T^*M)\otimes\Gamma_0^{\frac{1}{2}}).$$
Moreover, since $L_g(\bar{\lambda})^*=L_g(\lambda)$, we can define $Q_l(\lambda)=Q_r(\bar{\lambda})^*$. Then
$$R_l(\lambda)=Q_l(\lambda)L_g(\lambda)-Id=(L_g(\bar{\lambda})Q_r(\bar{\lambda}))^*-Id\in
\Psi_0^{-\infty,E_T(\lambda),\infty}(M;S^2({}^0T^*M)\otimes\Gamma_0^{\frac{1}{2}}).$$
Note that $R_r(\lambda)$ and $R_l(\lambda)$ are two compact operators on $x^aL_0^2(M;S^2({}^0T^*M)\otimes\Gamma_0^{\frac{1}{2}})$. That is because
$$\mathrm{Ran}(R_r(\lambda)|_{x^aL_0^2})\subset
\mathscr{A}^{\infty}(\overline{M};S^2({}^0T^*M)\otimes\Gamma_0^{\frac{1}{2}}),$$
$$\mathrm{Ran}(R_l(\lambda)|_{x^aL_0^2})\subset \sum_{i=0}^{i=3}\mathscr{A}^{e_i(\lambda)}(\overline{M};S^2({}^0T^*M)\otimes\Gamma_0^{\frac{1}{2}}),$$
where $e_i(\lambda)$ is the $i$-th column of $E_T(\lambda)$, and the inclusion maps
$$\mathscr{A}^{\infty}(\overline{M};S^2({}^0T^*M)\otimes\Gamma_0^{\frac{1}{2}})\hookrightarrow x^aL_0^2(M;S^2({}^0T^*M)\otimes\Gamma_0^{\frac{1}{2}}),$$
$$\mathscr{A}^{e_i(\lambda)}(\overline{M};S^2({}^0T^*M)\otimes\Gamma_0^{\frac{1}{2}})\hookrightarrow x^aL_0^2(M;S^2({}^0T^*M)\otimes\Gamma_0^{\frac{1}{2}})$$
are compact for $\lambda\in D_{\varepsilon}$ and $a\in
(\tfrac{n}{2}-\delta,\tfrac{n}{2}-\frac{\delta}{2})$, where $\varepsilon>0$ and $\delta>0$ are small enough.
Hence $Id+R_r$ and $Id+R_l$ are Fredhom operators on $x^aL_0^{2}(M;S^2({}^0T^*M)\otimes\Gamma_0^{\frac{1}{2}})$, which implies that $L_g(\lambda)$ is Fredholm.
\end{proof}

\begin{lemma}
If $\varepsilon>0$ and $\delta>0$ are small enough, then there exists
$\lambda_0\in\mathbb{R}\cap D_{\varepsilon}$ such that for all
$\lambda\in\mathbb{R}$, $\lambda\geq\lambda_0$ or $\lambda\in
D_{\varepsilon}$, $\Im\lambda\neq 0$, and for all $a\in
(\tfrac{n}{2}-\delta,\tfrac{n}{2}-\frac{\delta}{2})$,
$$L_g(\lambda): x^aH_0^2(M;S^2({}^0T^*M)\otimes\Gamma_0^{\frac{1}{2}})
\longrightarrow
x^aL_0^{2}(M;S^2({}^0T^*M)\otimes\Gamma_0^{\frac{1}{2}})$$ is
invertible. Moreover, $L_g(\lambda)^{-1}\in
\Psi_0^{-2,E_T(\lambda),E_B(\lambda)}(M;S^2({}^0T^*M)\otimes\Gamma_0^{\frac{1}{2}})$.
\end{lemma}
\begin{proof}
Since $\partial\overline{M}$ is compact, we can choose a finite coordinate covering $\{U_i:i=1,...,p\}$ for $\partial \overline{M}$. Hence we get
a finite number of coordinate patches $\{[0,\epsilon)\times
U_i:i=1,...,p\}$ which cover $[0,\epsilon)\times\partial\overline{M}$ for some $\epsilon>0$ and small.
Since $M-[0,\epsilon/2)\times \partial\overline{M}$ is compact, it also can be
covered by finite coordinate patches. Thus we have fixed a finite
coordinate covering for $\overline{M}$. By directly computation,
there exist some constant $C$ such that any sectional curvature
$|R_{ijkl}|<Cx^4$ in any coordinate patch. So for any $h\in
H_0^2(M;S^2({}^0T^*M)\otimes\Gamma_0^{\frac{1}{2}})$, $$\|Rh\|\leq
C_1 \|h\|_{L^2_0}$$ for some constant $C_1$. Choose
$\lambda_0=\frac{n+\sqrt{n^2+8C_1}}{2}$. Then for
any $h\in H_0^2(M;S^2({}^0T^*M)\otimes\Gamma_0^{\frac{1}{2}})$ and
$\lambda \in [\lambda_0,\infty)$
\begin{eqnarray*}
\langle L_g(\lambda_0)h,h\rangle_{L_0^2} &=& \|\nabla
h\|^2_{L_0^2} +\langle
2Rh,h\rangle_{L_0^2}+(2n-\lambda(n-\lambda))\|h\|_{L_0^2}\\
&\geq& \|\nabla
h\|^2_{L_0^2}+(2n+\lambda(\lambda-n)-2C_1)\|h\|_{L_0^2} \geq 0
\end{eqnarray*}
and equality holds if and only if $h=0$. Hence $L_g(\lambda)$ is
positive and thus invertible. If $\lambda \in D_{\varepsilon}$ and $\Im\lambda\neq 0$, then $\Im(2n-\lambda(n-\lambda))\neq 0$. Hence $L_g(\lambda)$ is injective and surjective since its adjoint $L_g(\bar{\lambda})$ is injective. The second
statement is shown by a proof similar to that of Proposition
\ref{prop-property of inverse}.
\end{proof}

\begin{proposition}
If $\varepsilon>0$ and $\delta>0$ are small enough, then for all
$\lambda\in D_{\varepsilon}$ and $a\in
(\tfrac{n}{2}-\delta,\tfrac{n}{2}+\delta)$, the resolvent
$R_g(\lambda)=(L_g(\lambda))^{-1}\in \Psi_0^{-2,E_T(\lambda),E_B(\lambda)}(M;S^2({}^0T^*M)\otimes\Gamma_0^{\frac{1}{2}})$ and as a map,
$$R_g(\lambda): x^aH_0^2(M;S^2({}^0T^*M)\otimes\Gamma_0^{\frac{1}{2}})
\longrightarrow
x^aL_0^{2}(M;S^2({}^0T^*M)\otimes\Gamma_0^{\frac{1}{2}})$$
is meromorphic in $D_{\varepsilon}$
with poles of finite rank at each point of $\{\lambda\in
D_{\varepsilon}:\lambda(n-\lambda)\in \sigma_{pp}(L_g)\}$.
Moreover, $\{\lambda\in D_{\varepsilon}:\lambda(n-\lambda)\in
\sigma_{pp}(L_g)\}$ is discrete and finite. If $\lambda_1\in D_{\varepsilon}$ and
$\lambda_1(n-\lambda_1)\in \sigma_{pp}(L_g)$, then the residue of
$R_g(\lambda)$ at $\lambda_1$ is
$$\Pi(\lambda_1)=\frac{1}{2\pi i}\int_{\gamma_0}R_g(\lambda)ds$$
with kernel
$\kappa(\Pi(\lambda_1))=\sum_{j=1}^ku_i(\lambda_1)\otimes\overline{u_j(\lambda_1)}$
where $u_j(\lambda_1)\in \sum_{i=0}^3\mathscr{A}^{e_i(\lambda_1)}$
are the $L^2_0$-engenvectors of $L_g$ and $\gamma_0$ is a small circle
around $\lambda_1$.
\end{proposition}
\begin{proof}
Let $\lambda_0$ be the same as in the last Lemma. A similar proof as in
Proposition \ref{prop-property of inverse} shows that for
$\lambda\in D_{\varepsilon}$, $\Re\lambda\geq \lambda_0$ or
$\lambda\in D_{\varepsilon}$, $\Im \lambda\neq 0$,
$$R_g(\lambda)\in
\Psi_0^{-2,E_T(\lambda),E_B(\lambda)}(M;S^2({}^0T^*M)\otimes\Gamma_0^{\frac{1}{2}})$$
is holomorphic. By Analytic Fredholm Theorem,
$R_g(\lambda)=(L_g(\lambda))^{-1}$ is meoromorphic on
$D_{\varepsilon}$. Hence there are only finite poles in
$D_{\varepsilon}$. By choosing a limit from upper half plane, we have
$$R_g(\lambda)\in
\Psi_0^{-2,E_T(\lambda),E_B(\lambda)}(M;S^2({}^0T*M)\otimes\Gamma_0^{\frac{1}{2}})$$
for all $\lambda\in D_{\varepsilon}$, $\lambda_1(n-\lambda_1)\notin \sigma_{pp}(L_g)$. The finite rank property of poles follows directly from the Fredholm property of $L_g(\lambda)$.
\end{proof}

\begin{corallary} \label{cor-formalmap}
If $\lambda\in D_{\varepsilon}$ with $\varepsilon>0$, small and $N\in \mathbb{N}$ satisfies $N\geq 2n$, denoting by $e_i(\lambda)$ the $i$-th column
of $E_T(\lambda)$, then
$$R_g(\lambda):x^NC^{\infty}(\overline{M};S^2({}^0T^*M)\otimes
\Gamma_0^{\frac{1}{2}}(M))\longrightarrow
\sum_{i=0}^3\mathscr{A}^{e_i(\lambda)}(\overline{M};S^2({}^0T^*M)\otimes
\Gamma_0^{\frac{1}{2}}(M))$$ is bounded and meromorphic with poles
of finite rank at $\{\lambda\in
D_{\varepsilon}:\lambda(n-\lambda)\in \sigma_{pp}(L_g)\}$.
\end{corallary}
\begin{proof}
For $N\geq 2n$, $x^NC^{\infty}(\overline{M};S^2({}^0T^*M)\otimes
\Gamma_0^{\frac{1}{2}}(M))\subset x^a H_0^2(M;S^2({}^0T^*M)\otimes
\Gamma_0^{\frac{1}{2}}(M))$ where $a\in (\tfrac{n}{2}-\delta,\tfrac{n}{2}+\delta)$ with $\delta$ small. Hence if $\{\lambda\notin
D_{\varepsilon}:\lambda(n-\lambda)\in \sigma_{pp}(L_g)\}$, the map
$$R_g(\lambda):x^NC^{\infty}(\overline{M};S^2({}^0T^*M)\otimes
\Gamma_0^{\frac{1}{2}}(M))\longrightarrow x^aL_0^2(M;S^2({}^0T^*M)\otimes
\Gamma_0^{\frac{1}{2}}(M))$$
is well defined and continuous. Moreover, for any $f\in C^{\infty}(\overline{M};S^2({}^0T^*M)\otimes
\Gamma_0^{\frac{1}{2}}(M))$, view $R_g(\lambda)$ as its Schwartz kernel, then
$$R_g(\lambda)f=\pi_L{}_{*}(R_g(\lambda)(\pi_R^{*}f))$$
where $\pi_L$ and $\pi_R$ are the projection map from $\overline{M}\times_0\overline{M}$ to the first and second copy of $\overline{M}$. It follows immediately that the image is contained in $\sum_{i=0}^3\mathscr{A}^{e_i(\lambda)}(\overline{M};S^2({}^0T^*M)\otimes
\Gamma_0^{\frac{1}{2}}(M))$.
\end{proof}

\subsection[]{Poisson Operator}
Let $N\geq 2n$ be a fixed integer in this subsection. Recall $D_{\varepsilon}=\{\lambda\in \mathbb{C}:\Re\lambda>n-\varepsilon, |\Im\lambda|<\varepsilon\}$. Fix $\varepsilon>0$ and small enough such that all the lemmas and propositions in Subsection 5.1-5.3 hold. Then for $\lambda\in D_{\varepsilon}$, define maps
$$\Phi_i(\lambda):C^{\infty}(\partial \overline{M};\mathscr{V}^i
\otimes \Gamma^{\frac{1}{2}}(\partial \overline{M})) \longrightarrow
x^{s_i(\lambda)}\mathscr{A}^0(\overline{M};S^2({}^0T^*M)\otimes
\Gamma_0^{\frac{1}{2}}(M))$$
for $i=0,1,2,3$ by formal series construction,  as we
do for finding formal solutions for the ODE in Subsection 5.2, such that
$$L_g(\lambda)\Phi_i(\lambda):
C^{\infty}(\partial \overline{M};\mathscr{V}^i\otimes \Gamma^{\frac{1}{2}}(\partial \overline{M}))
\longrightarrow
x^NC^{\infty}(\overline{M};S^2({}^0T^*M)\otimes \Gamma_0^{\frac{1}{2}}(M))$$
is continuous. Namely, for any $h_0\in C^{\infty}(\partial \overline{M};\mathscr{V}^i\otimes\Gamma^{\frac{1}{2}}(\partial \overline{M}))$,
$$\Phi_i(\lambda)h_0 = x^{s_i(\lambda)}(h_0+\sum_{k=1}^{N}\sum_{l=0}^3 u_{kl}x^k(\log
x)^l)\left|\frac{dx}{x^{n+1}}\right|^{\frac{1}{2}},$$
where $u_{kl}\in C^{\infty}(\partial \overline{M};S^2({}^0T^*M)|_{\partial\overline{M}}
\otimes \Gamma^{\frac{1}{2}}(\partial \overline{M}))$ for all $k,l$. In fact we can construct this formal series up to arbitrary large order. However, we only need finitely many terms here due to Corallary \ref{cor-formalmap}. Now, for any $\lambda\in D_{\varepsilon}$, we can define the \textit{Poisson operator} following Graham and Zworski in [GZ], i.e. for $i=0,1,2,3$
$$P_i(\lambda)=\Phi_i(\lambda)-R_g(\lambda)L_g(\lambda)\Phi_i(\lambda).$$
Hence $L_g(\lambda)P_i(\lambda)=0$. The mapping property of $P_i(\lambda)$ follows immediately.

\begin{lemma}For
$i=0,1,2,3$ and $\lambda\in D_{\varepsilon}$, denoting by $e_i(\lambda)$ the $i$-th column of $E_T(\lambda)$, then
\begin{equation*}
\begin{split}P_i(\lambda):C^{\infty}(\partial \overline{M});\mathscr{V}^i \otimes
\Gamma^{\frac{1}{2}}(\partial \overline{M})) \longrightarrow
x^{s_i(\lambda)}\mathscr{A}^0(\overline{M};S^2({}^0T^*M)\otimes
\Gamma_0^{\frac{1}{2}}(M))\quad\quad\quad\quad\\+\sum_{i=0}^3
\mathscr{A}^{e_i(\lambda)}(\overline{M};S^2({}^0T^*M) \otimes \Gamma_0^{\frac{1}{2}}(M))
\end{split}
\end{equation*}
is continuous and meromorphic with poles of finite rank at $\{\lambda\in D_{\varepsilon}: \lambda(n-\lambda) \in \sigma_{pp}(L_g)\}$. Namely, for any $h_0\in C^{\infty}(\partial \overline{M}); \mathscr{V}^i \otimes \Gamma^{\frac{1}{2}}(\partial \overline{M})))$,
$$P_i(\lambda)h_0=x^{s_i(\lambda)}F+\sum_{j=0}^{3}x^{s^j(\lambda)}G_j,$$
where $F,G_0,G_1,G_2,G_3\in \mathscr{A}^{0}(\overline{M}; S^2({}^0T^*M)) \otimes \Gamma_0^{\frac{1}{2}}(M))$
satisfying $F|\tfrac{dx}{x^{n+1}}|^{-\frac{1}{2}}|_{\partial \overline{M}}=h_0$.
\end{lemma}

Notice that for $\lambda\in D_{\varepsilon}$, $\min\{\Re s^0(\lambda), \Re s^1(\lambda),\Re s^2(\lambda), \Re s^3(\lambda)\} = \Re s^1(\lambda) = \Re\lambda$. Moreover, if $\lambda \notin n + \frac{\mathbb{N}_0}{2}$, we have
$$\{|s_i(\lambda)-s_j(\lambda)|,|s^i(\lambda)-s_j(\lambda)|: i,j=0,1,2,3\} \cap \mathbb{N}=\emptyset,$$
which implies that no logarithmic terms have involved in the construction of formal solutions to the ODE and hence we can choose $K=0$ in the definition of (\ref{defn-Aphg}), i.e, we can replace all the $\mathscr{A}_{phg}^{0}(\overline{M})$ by $C^{\infty}(\overline{M})$ in the previous part of this section. Namely,

\begin{lemma}\label{lem-Rgsmooth}
For $i=0,1,2,3$, $\lambda\in D_{\varepsilon}$ and $\lambda \notin n+\frac{\mathbb{N}_0}{2}$, the two maps
$$R_g(\lambda):x^NC^{\infty}(\overline{M};S^2({}^0T^*M)\otimes\Gamma_0^{\frac{1}{2}}(M))
\longrightarrow
\sum_{i=0}^3x^{s^i(\lambda)}C^{\infty}(\overline{M};S^2({}^0T^*M)\otimes
\Gamma_0^{\frac{1}{2}}(M))$$
\begin{equation*}\begin{split}
P_i(\lambda):
C^{\infty}(\partial \overline{M};\mathscr{V}^i \otimes \Gamma^{\frac{1}{2}}(\partial M))
\longrightarrow
x^{s_i(\lambda)}C^{\infty}(\overline{M};S^2({}^0T^*M)\otimes \Gamma_0^{\frac{1}{2}}(M))
\quad\quad\quad\quad\\
+\sum_{j=0}^3x^{s^j(\lambda)}C^{\infty}(\overline{M};S^2({}^0T^*M)\otimes \Gamma_0^{\frac{1}{2}}(M))
\end{split}\end{equation*}
are continuous and meromorphic with poles of finite rank at $\{\lambda \in D_{\varepsilon}:\lambda\notin n+\frac{\mathbb{N}_0}{2},\ \lambda(n-\lambda)\in \sigma_{pp}(L_g)\}$. Hence for any $h_0\in
C^{\infty}(\partial \overline{M};\mathscr{V}^1\otimes \Gamma^{\frac{1}{2}}(\partial \overline{M}))$, there exists a constant $0<c(\lambda)<1$ such that $$P_1(\lambda)h_0=x^{n-\lambda}F+x^{\lambda}G+O(x^{\lambda+c(\lambda)})$$
where $F,G\in C^{\infty}(\overline{M};S^2({}^0T^*M)\otimes\Gamma_0^{\frac{1}{2}}(M))$ with
$$F|\tfrac{dx}{x^{n+1}}|^{-\frac{1}{2}}|_{\partial \overline{M}}=h_0,\quad
G|\tfrac{dx}{x^{n+1}}|^{-\frac{1}{2}}|_{\partial \overline{M}}\in C^{\infty}(\partial \overline{M}; \mathscr{V}^1 \otimes \Gamma^{\frac{1}{2}}(\partial \overline{M})).$$
\end{lemma}

\begin{proposition}\label{prop-poisson kernel}
For $\lambda\in D_{\varepsilon}$, $\lambda \notin n+\frac{\mathbb{N}_0}{2}$ and $\lambda(n-\lambda)\notin
\sigma_{pp}(L_g)$, the map
\begin{equation*}
\begin{split}
P_1(\lambda):C^{\infty}(\partial \overline{M};\mathscr{V}^1 \otimes \Gamma^{\frac{1}{2}}(\partial \overline{M})) \longrightarrow x^{n-\lambda}C^{\infty}(\overline{M};S^2({}^0T^*M)\otimes
\Gamma_0^{\frac{1}{2}}(M))\quad\quad\quad\quad\\
+\sum_{j=0}^3x^{s^j(\lambda)}C^{\infty}(\overline{M};S^2({}^0T^*M)\otimes \Gamma_0^{\frac{1}{2}}(M))
\end{split}
\end{equation*}
has kernel $\kappa(P_1(\lambda))=R_g(\lambda)x'^{-\lambda+\frac{n+1}{2}}|dx'|^{-\frac{1}{2}}|_{x'=0}$. As a distribution on $\overline{M}\times_0\partial\overline{M}$ with $$\mathrm{singsupp}(\kappa(P_1(\lambda)))\subset\partial\overline{M}\times_0\partial\overline{M}.$$
\end{proposition}
\begin{proof}
For any $h_0\in C^{\infty}(\partial \overline{M};\mathscr{V}^i\otimes \Gamma^{\frac{1}{2}}(\partial \overline{M}))$, $f\in x^NC^{\infty}(\overline{M};S^2({}^0T^*M\otimes\Gamma_0^{\frac{1}{2}}(M))$,
by last lemma,  
$$u_1=P_1(\lambda)(h_0)=x^{n-\lambda}F_1+x^{\lambda}G_1+O(x^{\lambda+c(\lambda)}),
\quad u_2=R_g(\lambda)(f)=x^{\lambda}G_2+O(x^{\lambda+c(\lambda)})$$
for some $0<c(\lambda)<1$ and $F_1,G_1,G_2\in C^{\infty}(\overline{M};S^2({}^0T^*M)\otimes\Gamma_0^{\frac{1}{2}}(M))$. Then
for $\lambda \in \mathbb{R}\cap D_{\varepsilon}$, $\lambda\notin n+\frac{\mathbb{N}_0}{2}$ and $\lambda(n-\lambda)\notin \sigma_{pp}(L_g)$,
\begin{eqnarray*}
\int_M P_i(\lambda)h_0\bar{f}&=& \int_M
u_1\overline{\triangle_gu_2}-\int_M\triangle_gu_1\bar{u}_2\\
&=& \lim_{\epsilon\rightarrow
0}\int_{x=\epsilon}u_1(\overline{\nabla u_2\cdot\vec{n}})
-\int_{x=\epsilon}(\nabla u_1\cdot \vec{n})\bar{u}_2\\
&=&(2\lambda-n)\int_{\partial
M}F_1\overline{G_2}x^{n+1}|dx|^{-1}|_{\partial M}\\
&=& (2\lambda-n)\int_{\partial M} h_0
\overline{x^{-\lambda+\frac{n+1}{2}}R_g(\lambda)f|dx|^{-\frac{1}{2}}|_{\partial
M}}
\\
&=& (2\lambda-n)\int_M
(R_g(\lambda)x^{-\lambda+\frac{n+1}{2}}|dx|^{-\frac{1}{2}})h_0\otimes
\delta_0(x)\bar{f}
\end{eqnarray*}
The result extend to whole domain $\{\lambda\in D_{\varepsilon}: \lambda \notin n + \frac{\mathbb{N}_0}{2},\ \lambda(n-\lambda) \notin \sigma_{pp}(L_g)\}$ by analytic continuation. Since $R_g(\lambda)$ only has singularities on $diag_0(M)$ and cornormal singularities on the boundary surfaces of $\overline{M} \times_0 \overline{M}$, $\kappa(P_1(\lambda))$ as a restriction on one boundary surfaces must be smooth when $x>0$.
\end{proof}

If $\lambda_0=n+\frac{l}{2}\in D_{\varepsilon}$, for some $l\in \mathbb{N}_0$, then
for any $h_0\in C^{\infty}(\partial \overline{M};\mathscr{V}^i\otimes
\Gamma^{\frac{1}{2}}(\partial \overline{M}))$,
$$P_1(h_0)=x^{n-\lambda_0}F+x^{\lambda_0}(\log x)H_1+x^{\lambda_0+1}(\log x)^2H_2 +O(x^{\lambda_0})$$
where $F,H_1,H_2\in C^{\infty}(\overline{M};S^2({}^0T^*M)\otimes\Gamma_0^{\frac{1}{2}}(M))$,
$H_2=0$ if $l\neq 0$ and $$F|\tfrac{dx}{x^{n+1}}|^{-\frac{1}{2}}|_{\partial \overline{M}}=h_0,\quad
H_1|\tfrac{dx}{x^{n+1}}|^{-\frac{1}{2}}|_{\partial \overline{M}}\in C^{\infty}(\partial \overline{M};\mathscr{V}^1 \otimes \Gamma^{\frac{1}{2}}(\partial \overline{M})),$$
$$H_2|\tfrac{dx}{x^{n+1}}|^{-\frac{1}{2}}|_{\partial \overline{M}}\in C^{\infty}(\partial \overline{M};\mathscr{V}^3 \otimes \Gamma^{\frac{1}{2}}(\partial \overline{M})).$$
$H_1|\tfrac{dx}{x^{n+1}}|^{-\frac{1}{2}}|_{\partial \overline{M}}$ and
$H_2|\tfrac{dx}{x^{n+1}}|^{-\frac{1}{2}}|_{\partial \overline{M}}$ are
determined by the formal construction of $\Phi_1$ up to order $N$, i.e. there exist some differential operators
$$A_{n+l}\in \textrm{Diff}^{n+l}(\partial \overline{M}; End(\mathscr{V}^1) \otimes \Gamma^{\frac{1}{2}}(\partial \overline{M})),\quad A_{n+l}\in \textrm{Diff}^{n+l+1}(\partial \overline{M};Hom(\mathscr{V}^1,\mathscr{V}^3) \otimes \Gamma^{\frac{1}{2}}(\partial \overline{M}))$$
such that $H_1|\tfrac{dx}{x^{n+1}}|^{-\frac{1}{2}}|_{\partial \overline{M}}=A_{n+l}h_0$, $H_2|\tfrac{dx}{x^{n+1}}|^{-\frac{1}{2}}|_{\partial \overline{M}}=A_{n+l+1}h_0$.

If $\lambda_0\in D_{\varepsilon}$, $\lambda_0(n-\lambda_0)\in
\sigma_{pp}(L_g)$, then the Poisson operator has residue of a finite rank operator at $\lambda_0$
$$Res_{\lambda=\lambda_0}(P_1(\lambda))=(2\lambda_0-n)\Pi(\lambda_0)x'^{-\lambda}|_{x'=0}.$$

\subsection{Scattering Matrix}
For $\lambda\in D_{\varepsilon}$, $\lambda \notin n+\frac{\mathbb{N}_0}{2}$ and $\lambda(n-\lambda)\notin \sigma_{pp}(L_g)$, we can define the \textit{Scattering matrix} 
$$S_1(\lambda):C^{\infty}(\partial \overline{M};\mathscr{V}^1 \otimes \Gamma^{\frac{1}{2}}(\partial \overline{M}))\longrightarrow C^{\infty}(\partial \overline{M};\mathscr{V}^1 \otimes \Gamma^{\frac{1}{2}}(\partial \overline{M}))$$
by
$$S_1(\lambda)h_0=Gx^{\frac{n+1}{2}}|dx|^{-\frac{1}{2}}|_{\partial \overline{M}}.$$ 
Since the choice of smooth boundary defining function for $\partial \overline{M}$ remains free, it is more natural to view $S_1(\lambda)$ as a map
$$S_1(\lambda):C^{\infty}(\partial \overline{M};\mathscr{V}^1\otimes
\Gamma^{\frac{1}{2}}(\partial \overline{M})\otimes |N^*\partial
\overline{M}|^{n-\lambda})\longrightarrow
C^{\infty}(\partial \overline{M};\mathscr{V}^1\otimes \Gamma^{\frac{1}{2}}(\partial \overline{M})\otimes
|N^*\partial \overline{M}|^{\lambda})$$

\begin{proposition}\label{prop-scattering pole}
For $\varepsilon >0$, small, the scattering matrix $S_1(\lambda)$ is meromorphic in $D_{\varepsilon}$ with poles at $n+\frac{\mathbb{N}_0}{2}\cup\{\lambda\in D_{\varepsilon}: \lambda(n-\lambda)\notin \sigma_{pp}(L_g)\}$. At any pole $\lambda_0$, we have
\begin{equation*}
Res_{\lambda=\lambda_0}S_1(\lambda)=\left\{\begin{array}{ll}
\pi(\lambda_0), & \mathrm{for}\ s_0\neq n+\frac{\mathbb{N}_0}{2},\\
\pi(\lambda_0)+A_{n+l}, & \mathrm{for}\ s_0=n+\frac{l}{2},
\end{array}\right.
\end{equation*}
where $\kappa(\pi(\lambda_0))=(2\lambda_0-n)x^{-\lambda_0}\Pi(\lambda_0)x'^{-\lambda_0}
|\frac{dxdx'}{x^{n+1}x'^{n+1}}|^{\frac{1}{2}}|_{\partial \overline{M}\times\partial \overline{M}}$.
\end{proposition}

To find the kernel of $S_1(\lambda)$, for $\lambda\in D_{\varepsilon}$, $\lambda \notin n+\frac{\mathbb{N}_)}{2}$ and $\lambda(n-\lambda) \notin \sigma_{pp}(L_g)$, we study the
asymptotic behavior of $\kappa(P_1)$ first. Recall the blow down
map $$\bar{\beta}:\overline{M}\times_0\partial \overline{M}\longrightarrow \overline{M}\times\partial \overline{M}, \quad \hat{\beta}:\partial \overline{M}\times_0\partial \overline{M}\longrightarrow \partial \overline{M}\times\partial \overline{M}. $$
The front face for $\bar{\beta}$ is $F\cap B$ and the front face of $\hat{\beta}$ is $F\cap B\cap T$. For any $(q,q)\in diag(\partial \overline{M})$ and a neighborhood $U\subset\overline{M} \times \partial \overline{M}$ around it, we can choose local coordinates in $\bar{\beta}^{-1}(U) \subset \overline{M} \times_0 \partial \overline{M}$ as
$$R=|x^2+|y-y'|^{\frac{1}{2}}|,\ \rho=\frac{x}{R}, \ Y=y-y', \ \omega=\frac{y-y'}{R},\ \rho_{F}=(1+|\omega|^2)^{-\frac{1}{2}}.$$ 
By Proposition \ref{prop-poisson kernel}, for any $i,j=0,1,2,3$, 
$\bar{\beta}^*(\kappa(P_1(\lambda)))=R^{-\lambda}F\nu\otimes\nu_0$, where
$$F=\left[\begin{array}{cccc}
0&C_{00}(\lambda)\rho^{s^0(\lambda)}F_{00}+C_{01}(\lambda)\rho^{s^1(\lambda)+2}F_{01}
+C_{02}(\lambda)\rho^{s^2(\lambda)+2}F_{02}+C_{03}(\lambda)\rho^{s^3(\lambda)+2}F_{03}&0&0\\
0&C_{10}(\lambda)\rho^{s^0(\lambda)+2}F_{10}+C_{11}(\lambda)\rho^{s^1(\lambda)}F_{11}
+C_{12}(\lambda)\rho^{s^2(\lambda)+1}F_{12}+C_{13}(\lambda)\rho^{s^3(\lambda)+1}F_{13}&0&0\\
0&C_{20}(\lambda)\rho^{s^0(\lambda)+2}F_{10}+C_{21}(\lambda)\rho^{s^1(\lambda)+1}F_{21}
+C_{22}(\lambda)\rho^{s^2(\lambda)}F_{22}+C_{23}(\lambda)\rho^{s^3(\lambda)+1}F_{23}&0&0\\
0&C_{30}(\lambda)\rho^{s^0(\lambda)+2}F_{30}+C_{31}(\lambda)\rho^{s^1(\lambda)+1}F_{31}
+C_{32}(\lambda)\rho^{s^2(\lambda)+1}F_{32}+C_{33}(\lambda)\rho^{s^3(\lambda)}F_{33}&0&0
\end{array}\right]$$
with $F_{ij}=F_{ij}(\rho,R,\omega)\in C^{\infty}(\overline{M}\times_0\partial \overline{M})$ and the coefficients $C_{ij}(\lambda)$ are meromorphic functions of $\lambda$. Let $s_{ij}(\lambda) = s^j(\lambda) +1-\delta_{ij}$ and 
$$\tau_{ij}(\lambda)=x^{-s_{ij}(\lambda)}\bar{\beta}_*(R^{-\lambda} F_{ij})\nu\otimes\nu_0 =\bar{\beta}_*(R^{-\lambda-s_{ij}(\lambda)} F_{ij})\nu\otimes\nu_0.$$
Since $s_{ij}> \lambda$ except for $i=j=1$, to study the scattering
matrix $S_1(\lambda)$, we only need to consider the asymptotic behavior of $\tau_{ij}(\lambda)$.

\begin{proposition}
For $\lambda\in D_{\varepsilon}$, $\lambda \notin n+\frac{\mathbb{N}_0}{2}$, $\lambda(n-\lambda) \notin \sigma_{pp}(L_g)$ and for $i,j=0,1,2,3$, $\tau_{ij}(\lambda)$ is a distributional section of $\Gamma^{\frac{1}{2}}(\overline{M}\times \partial \overline{M})$, which has conormal singularity at $diag(\partial \overline{M})$. Moreover, it has asymptotic expansion in $x$ as $x\downarrow 0$ in the sense that for any $f\in C^{\infty}(\partial \overline{M}\times\partial \overline{M})$, 
$$\langle\tau_{ij}(\lambda),f\nu_0\otimes\nu_0\rangle_{\partial \overline{M}\times\partial \overline{M}} =
(H_{\lambda}(x)+x^{n-\lambda-s_{ij}(\lambda)}G_{\lambda}(x)) \left| \frac{dx}{x^{n+1}} \right|^{\frac{1}{2}},$$ 
where $H_{\lambda},G_{\lambda}\in C^{\infty}([0,\epsilon))$ depend holomorphicly on $\lambda$. Moreover if $i=j=1$, 
$$H_{\lambda}(0)=\langle \hat{\beta}_*(R^{-2\lambda}F_{ij}|_{\rho=0})\nu_0\otimes\nu_0, f\nu_0\otimes\nu_0 \rangle_{\partial\overline{M}\times\partial\overline{M}}$$
$$G_{\lambda}(0)=\langle\langle x^{2\lambda-\frac{n-1}{2}}\tau_{ij}|_{F\cap B}, \rho_F^{-1-\frac{n}{2}}(\lambda) \rangle_{F\cap B}\delta_0(Y)\nu_0\otimes\nu_0, f\nu_0\otimes\nu_0\rangle_{\partial \overline{M}\times\partial \overline{M}}$$
\end{proposition}
\begin{proof}
The proof follows step by step from the proof of Proposition 4.1 in [JS] by Joshi and S\'{a} Barreto. We only need to do for such $f$ that $supp(f)$ is contained in a small neighborhood of $(q,q)\in diag(\partial M)$. In the radial coordinates around $(q,q)$,
$$R\partial _R=\bar{\beta}^*(x\partial x+Y\cdot\partial_Y),\quad\partial_x R=\rho, \quad x\partial_x \rho=\rho(1-\rho^2).$$
For any $m\in \mathbb{N}_0$, $$\Pi_{k=0}^{m}(x\partial_x-k)(x\partial_x+Y\partial_Y+\lambda+s_{ij}(\lambda)-k)
\tau_{ij}(\lambda) =x^{m+1}\bar{\beta}_*(R^{-\lambda-s_{ij}(\lambda)}F^m_{ij})$$
where $F^m_{ij}=F^m_{ij}(\rho,R,\omega)\in
C^{\infty}(\overline{M}\times_0\partial \overline{M})$. Let
$$u(x)=\int_{\partial \overline{M}}\int_{\mathbb{R}^n}\bar{\beta}_*(R^{-\lambda-s_{ij}(\lambda)}
F_{ij})(x,y,Y)f(y,Y)dYdy, \quad x>0.$$
Using the identity $\mathrm{div}(Yg(Y))=ng(Y)+Y\cdot\partial_Yg(Y)$, we deduce that
$$\Pi_{k=0}^{m}(x\partial_x-k)(x\partial_x-n+\lambda+s_{ij}(\lambda)-k)
u(x)\leq C|x|^{m+1},\quad x>0.$$
Since $\lambda+s_{ij}(\lambda)$ is not an integer, the ODE theory shows that for any large $m$, there
exists some constant $m'$ independent of $m$ and integers $p_1,p_2$, such that 
$$\left|u(x)-\sum_{k=0}^{p_1}x^kd_k-x^{n-\lambda-s_{ij}(\lambda)} \sum_{l=0}^{p_2}x^ld'_l\right|\leq C|x|^{m+1+m'},$$ 
where $d_k,d_l'$ are constants. Hence $H_{\lambda}(0)=d_0$, $G_{\lambda}(0)=d_0'$.

Furthermore, if $i=j=1$, then $s_{ij}(\lambda)=\lambda$, which can be extended holomorphically over $\mathbb{C}$. Hence $\tau_{11}=\bar{\beta}^*(R^{-2\lambda}F_{11})\nu\otimes\nu_0$ can be
extended meromorphically over $\mathbb{C}$.  As in [JS], the formula of $G_{\lambda}(0)$ and
$H_{\lambda}(0)$ follows directly by analytic continuation.
\end{proof}

\begin{corallary}
For $\lambda\in D_{\varepsilon}$, the kernel of the scattering matrix $S_1(\lambda)$ is a meromorphic family of pseudodifferential operators of order $2\lambda-n$, whose kernel can be expressed as 
$$\kappa(S_1(\lambda))= x^{-\lambda+\frac{n+1}{2}}R_g(\lambda)x'^{-\lambda+\frac{n+1}{2}}
|dxdx'|^{-\frac{1}{2}}|_{x=x'=0}=|Y|^{-2\lambda}\hat{F}(y,Y)$$ 
for some smooth function $\hat{F}\in C^{\infty}(\partial \overline{M}\times\mathbb{R}^n)$, and with poles at $n+\frac{\mathbb{N}_0}{2}\cap\{\lambda\in D_{\varepsilon}:\lambda(n-\lambda)\in \sigma_{pp}(L_g)\}$ of residue as in Proposition \ref{prop-scattering pole}. The principal symbol of $S_1(\lambda)$ is
$$\sigma_{2\lambda-n}(S_1(\lambda))=\mathscr{F}_{Y \rightarrow  \xi} (|Y|^{-2\lambda}\hat{F} (\lambda,y,Y/|Y|,0)),$$
where $\hat{F}(\lambda,y,Y/|Y|,|Y|)=\rho^{-\lambda}R_g(\lambda)\rho'^{-\lambda}|_{\rho=\rho'=0}$
and $\hat{F}(\lambda,y,Y/|Y|,0)=\rho^{-\lambda}R_g(\lambda)\rho'^{-\lambda}|_{R=\rho=\rho'=0}$.
\end{corallary}
\begin{proof}
For $\lambda\in D_{\varepsilon}$, $\lambda\notin n+\frac{\mathbb{N}_0}{2}$ and $\lambda(n-\lambda)\notin \sigma_{pp}(L_g)$,  $\Re s_{ij}(\lambda)>\Re\lambda+c$ for some constant $c>0$ except for $i=j=1$. Hence the scattering matrix $S_1(\lambda)$ only involves $\tau_{11}(\lambda)$. The formula for the kernel follows directly from the preceeding Proposition. Hence $\kappa(S_1(\lambda))$ only has a conormal singularity at $diag(\partial \overline{M})$. Near $diag(\partial \overline{M})$
$$\kappa(S_1(\lambda))=|Y|^{-2\lambda}\hat{F}(\lambda,y,Y/|Y|,|Y|).$$
Therefore $S_1(\lambda)$ is a meromorphic family of pseudodifferntial operators on
$\partial \overline{M}$ of order $2\lambda-n$ for $\lambda\in D_{\varepsilon}$.
\end{proof}

\begin{corallary}\label{cor-elliptic of scattering}
If $\Sigma=\{h_0\in C^{\infty}(\partial\overline{M};\mathscr{V}^1\otimes \Gamma_0^{\frac{1}{2}}): \delta_{g_0}h_0=0\}$, then for $\lambda\in D_{\varepsilon}$, $S_1(\lambda)|_{\Sigma}$ is elliptic and
meromorphic in $D_{\varepsilon}$ with poles at $n+\frac{\mathbb{N}_0}{2} \cup \{\lambda\in D_{\varepsilon}: \lambda(n-\lambda)\in \sigma_{pp}(L_g)\}$. In particular, in this case, $\lambda_0=n$ is not a pole if $L_g$ is invertible.
\end{corallary}
\begin{proof}
Let $\overline{\Sigma}=P_1(\Sigma)$, then the normal operator of $L_g(\lambda)|_{\overline{\Sigma}}$ is diagonal. Hence the Green function for $N_q(L_g(\lambda)|_{\overline{\Sigma}})$ is a function of $(R,\rho,\rho')$ by standard result for $\triangle_g$ acting on functions as in [MM]. By [JS] and the preceding Corallary, for $\lambda \in
D_{\varepsilon}$, $\lambda\notin n+\mathbb{N}_0$ and $\lambda(n-\lambda)\notin \sigma_{pp}(L_g)\}$,
$$\sigma_{2\lambda-n}(S_1(\lambda)|_{\Sigma})=2^{n-2\lambda} \frac{\Gamma(\frac{n}{2}-\lambda)} {\Gamma(\lambda-\frac{n}{2})} \sigma_{2\lambda-n}(\triangle^{\lambda-\frac{n}{2}}_{g_0}).$$
\end{proof}

\section{Proof of Theorem \ref{thm-main}}\label{sec-proof}
\begin{proof}
Under the hypothesis of Theorem \ref{thm-main}, given any $C^1$ family of boundary metrics $g_0(s)\in Met^{\infty}$, for $s\in (-\theta,\theta)$, $\theta >0$ and small, there exist a $C^1$ family of $C^{2,\alpha}$ compact conformally Poincar\'{e}-Einstein metric $g_+(s)$ such that $g_0(s)\in [x^2g_+(s)|_{T\partial \overline{M}}]$ and $g_+(0)=g$ is smooth. By Proposition \ref{prop-smoothness},
there exist a $C^1$ family of $C^{1,\alpha}$ diffeomorphisms $\Phi_s$ over $\overline{M}$, which are $C^{2,\alpha}$ in $M$, such that near the boundary
$$\tilde{g}_+(s)=\Phi_s^*g_+(s)=x^{-2}(dx^2+g_0(s)+x^2g_2(s))+\cdots+x^ng_n(s)+O(x^{n+1}).$$
Moreover, $\Phi_0$ is smooth. Hence $g=\tilde{g}_+(0)$ is smooth. Then
$$\tilde{h}=\frac{d\tilde{g}_+(s)}{dx}\vline_{s=0}=x^{-2}(g_0'+x^2g_2'+\cdots+x^ng_n'+O(x^{n+1}))$$
By Proposition \ref{prop-gauge choice}, there exist a family of $C^{n-1}$ diffeomorphisms $\Psi_s$ such that $\Psi_0=Id$, $\Psi_s|_{\partial \overline{M}}=Id_{\partial \overline{M}}$ and $h=\frac{d}{dx}\vline _{s=0} (\Psi_s\tilde{g}_+(s))$ satisfies $\beta_gh=0$ and $L_gh=0$, where $$h=x^{-2}(h_0+x^2h_2+\cdots+x^{n}h_n+O(x^{n+1}\log x));$$
here $h_0=g_0'-\tfrac{1}{n}tr_{g_0}g_0'$ and $h_n=g_n'-T_ng_0'$ for some differential operators $T_n$ of order $n$ with coefficients depending on the components of $g_0$ and $g_n$ and their derivatives. Moreover, if $tr_{g_0}g_0'=\delta_{g_0}g_0'=0$, then $T_n$ reduces to a differential operator $T_{n-1}$ of order $n-1$.

Recall that $D_{\varepsilon} = \{\lambda \in \mathbb{C}: \Re\lambda > n-\varepsilon, |\Im\lambda| < \varepsilon\}$. The uniqueness assumption implies that $L_g$ is invertible. Hence for $\varepsilon>0$ and small, the resolvent family of $R_g(\lambda)$ and $P_1(\lambda)$ are analytic in a neighborhood of $\lambda=n$ in $D_{\varepsilon}$. According to the decomposition of $S^2(^0T^*M)|_{\partial\overline{M}}$ in Section \ref{sec-linear einstein op}, $x^{-2}h_0\in C^{\infty}(\partial \overline{M};\mathscr{V}^1)$. Hence
$$h\nu=P_1(n)(x^{-2}h_0\nu_0).$$
Recall that $\nu=|dvol_g|^{\frac{1}{2}}$, $\nu_0=|dvol_{g_0}|^{\frac{1}{2}}$. Note that for $\lambda\in\mathbb{C}$, $0<|\lambda-n|<\varepsilon$, where $\varepsilon>0$ is small,
$$P_1(\lambda)(x^{-2}h_0\nu_0)=x^{n-\lambda}F+x^{\lambda}G+O(x^{\lambda+c})$$
for some $c>0$, where $F|\frac{dx}{x^{n+1}}|^{-\frac{1}{2}}|_{\partial \overline{M}} = x^{-2}h_0\nu_0$, $G|\frac{dx}{x^{n+1}}|^{-\frac{1}{2}}|_{\partial \overline{M}}\in C^{\infty}(\partial \overline{M}; \mathscr{V}^1 \otimes \Gamma^{\frac{1}{2}}(\partial \overline{M}))$. The fixed half density sections
also satisfy $\nu |\frac{dx}{x^{n+1}}|^{-\frac{1}{2}}|_{\partial \overline{M}}=\nu_0$. Hence we can omit them and only write
$$h=P_1(\lambda)(x^{-2}h_0)=x^{n-\lambda}\widetilde{F}+x^{\lambda}\widetilde{G}+O(x^{\lambda+c})$$
where $\widetilde{F},\widetilde{G}\in C^{\infty}(\partial \overline{M};S^2({}^0T^*M))$ and $\widetilde{F}|_{\partial \overline{M}}=x^{-2}h_0$,
$\widetilde{G}|_{\partial \overline{M}}=S_1(\lambda)(x^{-2}h_0)$. By Proposition \ref{prop-scattering pole}, $Res_{\lambda=n}S_1(\lambda)$ can only be some differential operator which corresponds to the $x^{n}\log x$ term in $h$. However, by Proposition \ref{prop-gauge choice}, after fixing the gauge, the largest logarithmic term is not bigger than $x^{n+1}\log x$, i.e. $S_1(\lambda)(h_0)$ must have no
pole at $\lambda=n$ for all $h_0\in x^2C^{\infty}(\partial\overline{M};\mathscr{V}^1)$. Hence $S_1(n):h_0\mapsto h_n$ is a pseudodifferential operator of order $n$. Therefore,
$$g_n'=h_n+T_gg_0'=S_1(n)(g_0'-\tfrac{1}{n}tr_{g_0}g_0')+T_ng_0'=\widetilde{T}_ng_0'$$
where $\widetilde{T}_n$ is a pseudodifferntial operator of order n, which is exactly the linearization of the Dirichlet-to-Neumann map $DN'$.

Furthermore, let $\Sigma=\{\gamma\in C^{\infty}(\partial\overline{M};T^*\partial\overline{M}): tr_{g_0}\gamma=\delta_{g_0}\gamma=0\}$. If $g_0'\in \Sigma$, $g_n'=S_1(n)g_0'+T_{n-1}g_0'$ for some
differential operator $T_{n-1}$ of order $n-1$. By Corollary \ref{cor-elliptic of scattering}, $S_1(n)|_{\Sigma}$ is elliptic. Hence $DN'|_{\Sigma}$ is elliptic with the same principal symbol
as $S_1(n)|_{\Sigma}$.

If $g_0'$ is a general symmetric 2-tensor on $\partial \overline{M}$, we first modify $g_0'$ to be an element of $\Sigma$ by a linearization of a $C^1$ family of diffeomorphisms and conformal changes, i.e., we want to find $h_0\in \Sigma$, $u\in C^{\infty}(\partial \overline{M})$ and $\omega\in C^{\infty}(\partial \overline{M};T^*\partial \overline{M})$ such that
\begin{eqnarray}\label{eq trace free and divergence free}
\left\{
\begin{array}{rcl}
h_0 &=& g_0'+ug_0+\delta_{g_0}^*\omega\\
tr_{g_0}h_0 &=& 0\\
\delta_{g_0}h_0 &=& 0
\end{array}
\right.
\end{eqnarray}
From the first two equations, we get $u=\frac{1}{n}(tr_{g_0}g_0'-tr_{g_0}\delta_{g_0}^*\omega)$. Replacing $u$ by $\frac{1}{n}(tr_{g_0}g_0'-tr_{g_0}\delta_{g_0}^*\omega)$ in the first equation and using the third one reduces the problem to finding a solution to the equation
\begin{equation}\label{eq-last}
\delta_{g_0}\mathrm{tf}\delta_{g_0}^*\omega=\delta_{g_0} \mathrm{tf} g_0';
\end{equation}
here '$\mathrm{tf}$' means the trace-free part of a symmetric 2-tensor.
Let $A=\delta_{g_0}\mathrm{tf}\delta_{g_0}^*\omega$. Then
$$A: H^2(\partial \overline{M}; T^*\partial \overline{M})\longrightarrow L^2(\partial \overline{M}; T^*\partial \overline{M}).$$
is continuous. It's easy to check that $A$ is an elliptic operator of second order. For any $q\in \partial \overline{M}$, choose normal coordinate neighborhood $(U;y^1,...,y^n)$ around $q$. Then the principal symbol of $A$ at $q$ is
$$\sigma_2(A)(\xi)|_q=\frac{1}{2}|\xi|^2Id_{n\times n}+(\frac{1}{2}-\frac{1}{n})
\left[\begin{array}{rlc}\xi_1\xi_1 & \cdots & \xi_1\xi_n\\
\vdots & &\vdots \\ \xi_1\xi_n & \cdots & \xi_n\xi_n\end{array}\right].$$
For $\xi_0=(1,0,...,0)$, $\sigma(A)(\xi_0)|_q$ is invertible. Hence $\sigma(A)(\xi)|_q$ is invertible for all $\xi\neq 0$ since it's invariant under the action by $O(n)$. In fact
$$\sigma_2(A)^{-1}(\xi)=|\xi|^{-4}\left(
2|\xi|^2Id_{n\times n}-\frac{n-2}{n-1}
\left[\begin{array}{lcr}
\xi_1\xi_1 &\cdots& \xi_1\xi^n\\
\vdots & & \vdots\\
\xi_n\xi_1 &\cdots&\ \xi_n\xi_n
\end{array}
\right]
\right)$$
Furthermore, $A$ is self-adjoint. Hence $Ran(A)=Null(A)^{\perp}$.  Moreover, for any $\omega\in H^2(\partial \overline{M}, T^*\partial \overline{M})$,
$$\langle A\omega,\omega \rangle_{L^2}= \parallel \mathrm{tf}\delta_{g_0}^*\omega\parallel^2_{L^2}\geq 0.$$
This implies that for any $\omega\in Null(A)$, $ \mathrm{tf}\delta_{g_0}^*\omega=0$, so
$$\langle \delta_{g_0} \mathrm{tf} g_0', \omega \rangle_{L^2} =\langle  \mathrm{tf} g_0
,  \mathrm{tf} \delta_{g_0}^*\omega\rangle_{L^2}=0.$$
Therefore, $\delta_{g_0} \mathrm{tf} g_0'\in Ran(A)$. Furthermore, by standard elliptic theory for compact manifolds, we can construc a generalized inverse $B$, which is an elliptic pseudodifferential operator of order $-2$ with principal symbol $\sigma_{-2}(B)=\sigma_2(A)^{-1}$ such that
$$AB=Id-\Pi,$$
where $\Pi$ is the projection map onto $Null(A)$, which is obviously of finite rank due to the Fredholm property. Let $\omega=B\delta_{g_0}tf g_0'$, then it is a solution to (\ref{eq-last}). Of course, it may not be unique. However, we only need one. Using this solution
$$h_0=g_0'-\frac{g_0}{n}(tr_{g_0} g_0'-tr_{g_0}\delta^*_{g_0}B\delta_{g_0}tf g_0')-\delta_{g_0}^*B\delta_{g_0}tf g_0'=tfg_0'-tf\delta_{g_0}^*B\delta_{g_0}tf g_0':=\Theta_{g_0}g_0'$$
where $\Theta_{g_0}$ defined above is a pseudodifferential operator of order $0$. The linearizaion of properties (\ref{diff inv}) and (\ref{conformal inv}) for the Dirichlet-to-Neumann map shows that
$$d\mathcal{N}(g_0')=d\mathcal{N}(h_0)+\tfrac{2-n}{2}u g_n +\delta_{g_n}^*\omega$$
Therefore, $\sigma_n(d\mathcal{N})=
\frac{\Gamma(-\frac{n}{2})}{2^n\Gamma(\frac{n}{2})}
\sigma_n(\triangle^{\frac{n}{2}}_{g_0}\Theta_{g_0})$. Theorem \ref{thm-main} is proved.
\end{proof}

Finally, we check that the principal symbol formula above is exactly a generalization of Graham's result. For $n\geq 5$, odd, let
$$\mathcal{W}:Met^{\infty}(\partial \overline{M})\longrightarrow C^{\infty}(\partial \overline{M};\otimes^4T^*\partial \overline{M})$$
denote the Weyl tensor and
$$W:C^{\infty}(\partial \overline{M};S^2(T^*\partial \overline{M}))\longrightarrow C^{\infty}(\partial \overline{M};\otimes^4(T^*\partial \overline{M}))$$
its linearization at $g_0$. By direct computation
$$\sigma_4(W^*W)=\tfrac{n-3}{n-2}\sigma_4(\triangle^2_{g_0}\Theta_{g_0}).$$
For $n=3$, let $$\mathcal{C}: Met^{\infty}(\partial \overline{M})\longrightarrow C^{\infty}(\partial\overline{M};S^2(T^*\partial\overline{M}))$$
be the Cotton-York tensor and
$$C: C^{\infty}(\partial\overline{M};S^2(T^*\partial\overline{M}))\longrightarrow C^{\infty}(\partial\overline{M};S^2(T^*\partial\overline{M}))$$
be its linearization at $g_0$. Then $C$ is self adjoint and
$$\sigma_3(|C|)=\sigma_3(\triangle_{g_0}^{\frac{3}{2}}\Theta_{g_0}),$$
where $|C|=(C^*C)^{\frac{1}{2}}$. Recall the result by Graham in [Gr1]
\begin{theorem}
Let $M=B^{n+1}$ and $\partial\overline{M}=S^{n}$. Then the linearization of the Dirichlet-to-Neumann map at the hyperbolic metric $g_0=d\theta^2$ is
\begin{equation}
d\mathcal{N}=
\left\{
\begin{array}{ll}
aW^*W(\triangle_{g_0}+c_1)\cdots(\triangle_{g_0}+c_m)\sqrt{\triangle_{g_0}+c_{m+1}}, & \textrm{if $n\geq 5$, odd;}\\
\frac{1}{3}|C|, & \textrm{if $n=3$,}
\end{array}
\right.
\end{equation}
where $m=\frac{n-5}{2}$, $a\neq 0$ and $c_i>0$ for $1\leq i\leq m+1$.
\end{theorem}
It is now clear that our result is compatible with Graham's. Moreover, we can compute the constant $a$ in his formula, namely,
$$a=\frac{(n-2)\Gamma(-\frac{n}{2})}{2^n(n-3)\Gamma(\frac{n}{2})}.$$

\section[]{N-Point Problem}\label{npoin}
The Dirichlet-to-Neumann map is of high nonlinearity provided that it is well defined. In the previous sections, we investigate its linearization as one way to describe properties of the map. In this section, we continue to the $N$-point problem, to give a better description.

To set up the problem, suppose the Dirichlet-to-Neumann map is well defined in a neighborhood of  $g_0\in Met^{\infty}(\partial \overline{M})$. Let $O$ be a small neighborhood of $0$ in the parameter space $\mathbb{R}^N$ and $g_+(s)\in C^{\infty}(O;PE^{2,\alpha})$ a smooth family of $C^{2,\alpha}$ conformally compact Poincar\'{e}-Einstein metrice. $g=g_+(0)$ is smooth. As in Section \ref{sec-geom setting}, via a smooth family of diffeomorphisms, we can assume
$$g_+(s)=x^{-2}(dx^2+g_0(s)+x^2g_2(s)+\cdots+x^ng_n(s)+O(x^{n+1})).$$
where $g_i(s)\in C^{\infty}(O;C^{\infty}(\partial\overline{M};S^2(T^*\partial\overline{M})))$. For simplicity as before, let
$$g=g_+(0)=x^{-2}(dx^2+g_0+\cdots+x^ng_n+O(x^{n+1}))$$
be a fixed smooth metric. Let $\beta=(\beta_1,...,\beta_N)$ with $\beta_i\in \mathbb{N}_0$ and denote
$$g^{(\beta)}(s)=\partial_s^{\beta}g_+(s), \quad g_0^{(\beta)}(s)=\partial_s^{\beta}g_0(s),\quad g_n^{(\beta)}(s)=\partial_s^{\beta}g_n(s).$$
By assumption, $g_n(s)=DN(g_0(s))$ is well defined. Then clearly, $g^{(\beta)}_n(0)$ is determined by $g_0^{(\gamma)}(0)$ for all $0\leq\gamma\leq \beta$. From previous sections, $g^{(\beta)}_n(0)$ linearly depends on $g^{(\beta)}_0(0)$ via a pseudodifferential operator same as the first order linearization. In this section, the dependence on $g_0^{(\gamma)}(0)$ for $0<\gamma<\beta$ is studied. For simplicity and as a start of induction, we only do the two point problem, i.e. $|\beta|=2$.

In this case, we can assume $N=2$ and WLOG, let $\beta=(1,1)$ in the two point problem. Consider the second order linearization of the Einstein equation,
\begin{equation}\label{eq-seclinear}
L_gg^{(1,1)}(0)-2\delta_g^*\beta_gg^{(1,1)}(0)+E(g^{(1,0)}(0),g^{(0,1)}(0))=0
\end{equation}
where $E(g^{(1,0)},g^{(0,1)})$ is a symmetric quadratic term, which is a differential operator of order $2$ with respect to each variable. For any $p\in M$, choose normal coordinates $(W;y^{\mu})$ around $p$, then for any $h,k\in C^{\infty}(M;S^{2}(^0T^*M))$,
\begin{eqnarray*}
E_{\mu\nu}(h,k) &=& -h^{\sigma\theta}(\partial_{\sigma}\Gamma_{\mu\nu\theta}(k)-\partial_{\nu}\Gamma_{\mu\sigma\theta}(k))
-k^{\sigma\theta}(\partial_{\sigma}\Gamma_{\mu\nu\theta}(h)-\partial_{\nu}\Gamma_{\mu\sigma\theta}(h))
\\&&
-g^{\sigma\rho}g^{\gamma\theta}(\partial_{\sigma}h_{\rho\gamma}\Gamma_{\mu\nu\theta}(k)+
\partial_{\sigma}k_{\rho\gamma}\Gamma_{\mu\nu\theta}(h)
-\partial_{\nu}h_{\rho\gamma}\Gamma_{\mu\sigma\theta}(k)
-\partial_{\nu}k_{\rho\gamma}\Gamma_{\mu\sigma\theta}(h))
\\&&
+\frac{1}{2}g^{\sigma\rho}g^{\gamma\theta}(\Gamma_{\mu\nu\rho}(h)\Gamma_{\sigma\theta\gamma}(k)
+\Gamma_{\mu\nu\rho}(k)\Gamma_{\sigma\theta\gamma}(h)
\\&&
-\Gamma_{\mu\sigma\gamma}(h)\Gamma_{\nu\theta\rho}(k)
-\Gamma_{\mu\sigma\gamma}(k)\Gamma_{\nu\theta\rho}(h))
\end{eqnarray*}
at $p$. Hence $E(g^{(1,0)},g^{(0,1)})\in C^{\infty}(M;S^{2}(^0T^*M))$.

Say a quadratic form $Q(h,k)$ satisfies the \textit{asymptotic property}, if for some $\epsilon>0$, writing
$$h=(h_0+xh_1+\cdots+x^nh_n+O(x^{n+\epsilon})),$$
$$k=(k_0+xk_1+\cdots+x^nk_n+O(x^{n+\epsilon})),$$
where $k_i,h_i\in C^{\infty}(\partial\overline{M};S^2(^0T^*\overline{M}))$ for $0\leq i\leq n$, then
$$Q(h,k)=(Q_0+xQ_1+\cdots+x^nQ_n+O(x^{n+\epsilon}))$$
with $Q_i\in C^{\infty}(\partial\overline{M};S^{2}(^0T^*M))$ for each $0\leq i\leq n$, which is a sum of quadratic terms in $(h_j,k_l)$ as a differenial operator or pseudodifferential operator of order $k$ with respect to each variable. Denote $\mathrm{ord}(Q)=k$. Obiously, $E$ satisfies the asymptotic property with $\mathrm{ord}(E)=2$.

Equation (\ref{eq-seclinear}) is not elliptic with respect to $g^{(1,1)}(0)$. To make it elliptic, we have to  make a suitable choice of the gauge. Following the ?? in Section \ref{sec-gauge fix}, there exist a smooth family of vector field $Y(s_1)$, with dual $1$-form $\varpi(s_1)$ for $(s_1,0)\in O$, which generate one parameter families of diffeomorphisms $\Psi^{s_1}_{s_2}$ for each $s_1$, such that if letting $\tilde{g}(s_1,s_2)=\Psi^{s_1}_{s_2}g_+(s_1,s_2)$, then $\tilde{g}^{(0,1)}(s_1,0)=g^{(0,1)}(s_1,0)+\delta_{g(s_1,0)}^*\varpi(s_1)$
satisfies
\begin{equation}\label{familygauge1}
\beta_{g(s_1,0)}\tilde{g}^{(0,1)}(s_1,0)=0.
\end{equation}
Note that $\tilde{g}(s_1,0)=g(s_1,0)$ for all $(s_1,0)\in O$ since $\Psi^{s_1}_0=Id$. Hence $\tilde{g}^{(1,0)}(0)=g^{(1,0)}(0)$.
Similarly, for each $s_2$, $(0,s_2)\in O$, we can choose a smooth family of vector fields $X(s_2)$, with dual $1$-form $\omega(s_2)$, which generate one parameter familes of diffeomorphisms $\Phi^{s_2}_{s_1}$ for each $s_2$ small, such that if letting $G(s_1,s_2)=\Phi^{s_2}_{s_1}\tilde{g}(s_1,s_2)$, then
$G^{(1,0)}(0,s_2)=\tilde{g}^{(1,0)}(0,s_2)+L_{X(s_2)}\tilde{g}(0,s_2)$ satisfies \
\begin{equation}\label{eq-familygauge}
\beta_{\tilde{g}(0,s_2)}G^{(1,0)}(0,s_2)=0.
\end{equation}
Note that $\Phi^{s_2}_0=Id$, so $G(0,s_2)=\tilde{g}(0,s_2)$ for all $(0,s_2)\in O$. In fact, it is only necessary that $\varpi(s_1)$ and $\omega(s_2)$ satisfy
\begin{eqnarray}
(\triangle_{g(s_1,0}+n)\varpi(s_1)+\beta_{g(s_1,0)}g^{(0,1)}(s_1,0)&=&0,\\
(\triangle_{\tilde{g}(0,s_2)}+n)\omega(s_2)+\beta_{\tilde{g}(0,s_2)}\tilde{g}^{(1,0)}(0,s_2)&=&0.
\end{eqnarray}
By direct computation at $s=0$, 
\begin{eqnarray*}
G^{(0,1)} &=& g^{(0,1)}-\delta_g^*(\triangle_g+n)^{-1}\beta_gg^{(0,1)},\\
G^{(1,0)} &=& g^{(1,0)}-\delta_g^*(\triangle_g+n)^{-1}\beta_gg^{(1,0)},\\
G^{(1,1)} &=& g^{(1,1)}-\delta_g^*(\triangle_g+n)^{-1}\beta_gg^{(1,1)}+Q(g^{(0,1)},g^{(1,0)}),\\
\beta_g(G^{(0,1)}) &=& \beta_g(G^{(1,0)})=0,\\
\beta_g(G^{(1,1)}) &=& R(G^{(0,1)},G^{(1,0)}),
\end{eqnarray*}
where $Q$ and $R$ are quadratic terms satisfying the asymptotic property with $\mathrm{ord}(Q)=2$, $\mathrm{ord}(R)=1$. Moreover,
$$Q(g^{(0,1)},g^{(1,0)})=(Q_0+xQ_1+\cdots+x^nQ_n+O(x^{n+\epsilon}))$$
with $Q_i=x^{-2}\sum_{j+k\leq i}C_{jk}(y)\hat{\nabla}^kg_0^{(0,1)}\hat{\nabla}^jg_0^{(1,0)}\in C^{\infty}(\partial\overline{M},S^2(^0T^*\overline{M}))$ for each $0\leq i\leq n-1$. In particular,
$$x^2Q_0=-\frac{1}{n}[(\mathrm{tr}_{g_0}g_0^{(0,1)})g_0^{(1,0)}
+(\mathrm{tr}_{g_0}g_0^{(1,0)})g_0^{(0,1)}]+\frac{1}{n}\mathrm{tr}_{g_0^{(1,0)}}g_0^{(0,1)}
+\frac{1}{n^2}(\mathrm{tr}_{g_0}g_0^{(0,1)}\mathrm{tr}_{g_0}g_0^{(1,0)})g_0,$$
$$x^2Q_n=\sum_{j+k\leq n-1}C_{jk}\hat{\nabla}^kg_0^{(0,1)}\hat{\nabla}^jg_0^{(1,0)}+(T_0(g_0,g_n,g_0^{(0,1)}))g_n^{(1,0)}
+(\widetilde{T}_0(g_0,g_0^{(1,0)}))g_n^{(0,1)}$$
where $T_0$ and $\widetilde{T}_0$ are differential operators of order $0$ with coefficients depending on the components of $g_0,g_n, g_0^{(0,1)}$ (resp. $g_0,g_n, g_0^{(1,0)})$ and their derivatives, which are linear in $g_0^{(0,1)}$ (resp. $g_0^{(1,0)}$).

Since $G(s_1,s_2)$ is again a smooth family of Poincar\'{e}-Einstein metrics, $G^{(1,1)}, G^{(1,0)}, G^{(0,1)}$ also satisfy the equation (\ref{eq-seclinear}). Hence
$$L_g(G^{(1,1)})-\tilde{E}(G^{(1,0)},G^{(0,1)})=0$$
where $\tilde{E}(G^{(1,0)},G^{(0,1)})=-E(G^{(1,0)},G^{(0,1)})+2\delta^*_gR(G^{(1,0)},G^{(0,1)})$ also satisfies the asymptotic property with $\mathrm{ord}(\tilde{E})=2$. Moreover, with respect to the decomposition of $S^2(^0T^*M)|_{\partial\overline{M}}$ in Section \ref{sec-linear einstein op}, $\Pi_1\widetilde{E}_0=0$. Writing
$H=G^{(1,1)}$, $h=G^{(1,0)}$ and $k=G^{(0,1)}$, 
\begin{eqnarray*}
H &=& H_0+xH_1+\cdots+x^nH_n +O(x^{n+\epsilon})\quad \textrm{with}\\
&& H_0=x^{-2}(\mathrm{tf}_{g_0}g_0^{(1,1)})+Q_0,\quad H_n=x^{-2}(g_n^{(1,1)}-T_ng_0^{1,1})+Q_n,\\
h &=& h_0+xh_1+\cdots+x^nh_n +O(x^{n+\epsilon})\quad \textrm{with}\\ && h_0= x^{-2}\mathrm{tf}_{g_0}g_0^{(1,0)},\quad h_n=x^{-2}(g_n^{(1,0)}-T_ng_0^{1,0}),\\
k &=& k_0+xk_1+\cdots+x^nk_n +O(x^{n+\epsilon})\quad \textrm{with}\\ && k_0 = x^{-2}\mathrm{tf}_{g_0}g_0^{(0,1)},\quad k_n=x^{-2}(g_n^{(0,1)}-T_ng_0^{0,1}),
\end{eqnarray*}
where $T_n$ is a differential operator of order $n$ with coefficients depending on the components of $g_0$ and $g_n$ and their derivatives. By Section \ref{sec-gauge fix}, $h=P_1(h_0)$ and $k=P_1(k_0)$, where $P_1$ is the Poisson operator defined in Section \ref{sec-linear einstein op}. Hence $P_1$ is defined by its kernel $$p(m;y')\in C^{-\infty}(M\times \partial\overline{M};Hom(\mathscr{V}^1,S^2(^0T^*M))).$$
Note that we omit involving the $0$-half density here since the backgound metric $g$ and $g_0$ trivialize  $\Gamma_0^{\frac{1}{2}}(M)$ and $\Gamma^{\frac{1}{2}}(\partial\overline{M})$ respectivley. Denote $m=(x,y)$ in a neighborhood of the boundary.
\begin{lemma} The kernel of $P_1$ is smooth for $x>0$ and has a cornormal singularty at $x=0$. Moreover,
$$p(x,y;y')=\sum_{i=0}^{n}x^{i}u_i(y,y')+O(x^{n+\epsilon})$$
for some $\epsilon>0$, where $u_i\in C^{-\infty}(\partial\overline{M}\times\partial\overline{M};Hom(\mathscr{V}^1,S^2(^0T^*M)))$ with $\mathrm{singsupp}(u_i)\in diag(\partial\overline{M})$ are distributions of finite order for all $i$. The asymptotic expansion of $p$ holds in the sense that for any $f\in C^{\infty}(\partial\overline{M}\times\partial\overline{M};S^2(^0T^*\partial\overline{M}))$, $$P_1(f)=\sum_{i=1}^{n}u_i(f)+O(x^{n+\epsilon}),$$
with $u_i(f)\in C^{\infty}(\partial\overline{M};S^2(^0T^*M))$.
\end{lemma}

Since $\widetilde{E}(h,k)=\widetilde{E}(P_1(h_0),P_2(k_0))$ is a linear operator on $h_0\otimes k_0$, we can denote its Schwartz kernel by
$$e(m;\bar{y},\tilde{y})\in C^{-2}(M\times\partial\overline{M}\times\partial\overline{M};
Hom(\mathscr{V}^1\otimes\mathscr{V}^1,S^2(^0T^*M))).$$
\begin{lemma} \label{lemma-asympofe}
The kernel of $\widetilde{E}$ is smooth in $x>0$ with only cornormal singularity at $x=0$. Moreover, near the boundary
$$e(x,y;\bar{y},\tilde{y})=\sum_{i=1}^{n}x^i e_i(y;\bar{y},\tilde{y})+O(x^{n+\epsilon})$$
for some $\epsilon>0$, where $e_i\in C^{-\infty}((\partial\overline{M})^3; Hom(\mathscr{V}^1\otimes \mathscr{V}^1,S^2(^0T^*\overline{M})))$ with 
$$\mathrm{singsupp}(e_i)\in \{y=\bar{y}\}\cup \{y=\tilde{y}\} \cup \{\bar{y}=\tilde{y}\}$$ 
for all $0\leq i\leq n$. The asymptotic expansion of $e$ holds in the sense that for any $f_1,f_2\in C^{\infty}(\partial\overline{M};\mathscr{V}^1)$,
$$\widetilde{E}(f_1\otimes f_2)=\sum_{i=0}^n x^i e_i(f_1\otimes f_2)+O(x^{n+\epsilon})$$
with $e_i(f_1\otimes f_2)\in C^{\infty}(\partial\overline{M};S^2(^0T^*\partial\overline{M}))$.
\end{lemma}
\begin{proof}
First, for any $\epsilon>0$ and $x>\epsilon$, can choose a cut off function $\chi(x)\in C^{\infty}((0,\infty);[0,1])$ such that $\chi(x)=1$ for $x>\tfrac{1}{2}\epsilon$ and $\chi(x)=0$ for $x\leq \tfrac{1}{4}$. Then $e(m;\bar{y},\tilde{y})=\widetilde{E}(\chi(x)p(m;\bar{y}),\chi(x)p(m;\tilde{y}))$ is smooth since both $\chi(x)p(m;\bar{y})$ and $\chi(x)p(m;\tilde{y})$ are smooth, vanishing near the boundary, and $\widetilde{E}$ is a differential oeprator with respect to each variable. The asymptotic expansion of $e(x,y;\bar{y},\tilde{y})$ can be compute directly. Since for any $f_1,f_2\in C^{\infty}(\partial\overline{M};\mathscr{V}^1)$,
$$P_1(f_j)=\sum_{l=0}^{n}x^lu_l(f_j)+O(x^{n+\epsilon})$$ with $u_l(f_j)\in C^{\infty}(\partial\overline{M};S^2(^0T^*\partial\overline{M}))$. The asymptotic property of $\widetilde{E}$ shows that
$$\widetilde{E}(P_1f_1,P_1f_2)=\widetilde{E}_0+x\widetilde{E}_1+
\cdots+x^n\widetilde{E}_n+O(x^{n+\epsilon})$$
where $E_i\in C^{\infty}(\partial\overline{M};S^2(^0T^*\overline{M}))$ is a finite sum of quadratic forms on $u_l(f_1), u_{l'}(f_2)$ for $l,l'\leq i$ and is a pseudodifferential operator with respect to each variable. So the kernel of $\widetilde{E}_i$ is a sum of distributions of the form
$$\int_{\partial\overline{M}}w(y,y')u_l(y',\bar{y})u_{l'}(y',\tilde{y})dy'$$
with $\mathrm{singsupp}(w)\in \{y=y'\}$, $\mathrm{singsupp}(u_l)\in\{y'=\bar{y}\}$ and $\mathrm{singsupp}(u_{l'})\in \{y'=\tilde{y}\}$. If $y\neq \bar{y}$, $y\neq\tilde{y}$ and $\bar{y}\neq \tilde{y}$, there are small neighborhoods $U$ of $y$, $\bar{U}$ of $\bar{y}$ and $\tilde{U}$ of $\tilde{y}$, such that $U\cap\bar{U}=\emptyset$, $\bar{U}\cap\tilde{U}=\emptyset$ and $\tilde{U}\cap\bar{U}=\emptyset$.
Then for $y'$ in each of these small neighborhood, at most one of $w, u_l, u_{l'}$ can be singular. Hence the integral over such neighborhood will be well defined as a pairing of distribution with smooth function in the variable $y'$, smoothly depending on the parameter $(y,\bar{y},\tilde{y})\in U\cup\bar{U}\cup\tilde{U}$. On $\partial\overline{M}-U-\bar{U}-\tilde{U}$, every term will be smooth and hence 
$$\int_{\partial\overline{M}-U-\bar{U}-\tilde{U}}w(y,y')u_l(y',\bar{y})u_{l'}(y',\tilde{y})dy'$$
is smooth in $(y,\bar{y},\tilde{y})\in U\cup\bar{U}\cup\tilde{U}$, too. 
\end{proof}

\begin{lemma}\label{lemma-asym}
Suppose $\lambda_1(\lambda)$ and $\lambda_2(\lambda)$ are two nonconstant analytic functions on $\mathbb{C}$, satisfying $\lambda_1+\lambda_2>n$, and $P\in \Phi_0^{-2,\lambda_1,\lambda_2}(M)$ is a meromorphic family of pseudodifferential operators with kernel $\rho^{\lambda_1}\rho'^{\lambda_2}F$, where $F\in C^{\infty}(\overline{M}\times_0\overline{M})$. Then for any $k\in\mathbb{R}$ such that $\Re\lambda_2-n+k>0$ and $\lambda_1-k\notin \mathbb{N}_0$, and any $u\in C^{\infty}(\overline{M})$,
$$P(x^ku)=x^kG_1+x^{\lambda_1}G_2,$$
where $G_1,G_2\in C^{\infty}(\overline{M})$, and
$$G_1|_{\partial\overline{M}}=c(\lambda)u|_{x=0},\quad G_2|_{\partial\overline{M}}=(x^{-\lambda_1}P|_{x=0})(u),$$
where $c(\lambda)$ is some constant depending on $\lambda$.

If $\lambda_2-n+k=0$, then $x^kG_1$ is replaced by $x^k\log xF_1+x^kG_1$ with $F_1,G_1\in C^{\infty}(\overline{M})$ and $F_1=c'u$ for some constant $c'$.

If $\lambda_1-k=l\in \mathbb{N}_0$, then $x^{\lambda_1}G_2$ is replaced by $x^{\lambda_1}\log xF_2+x^{\lambda_1}G_2$ with $F_2,G_2\in C^{\infty}(\overline{M})$ and 
$$F_2=\sum_{j\leq l}T_{l-j} (\partial^{j}_{x}u)(0,y)$$ 
where $T_{l-j}$ are differential operators of order $l-j$ on the boundary for $0\leq j\leq l$.
\end{lemma}
\begin{proof}
Let $(U;y)$ be a small neighborhood in $\partial\overline{M}$ around $p$. Let $\phi (x)$ be a smooth cut off function such that for some $\epsilon>0$, $\phi (x)=1$ if $x\leq \epsilon$ and $\phi (x)=0$ if $x\geq 2 \epsilon$. Let $\psi (y)$ be some smooth cut off funtion around $p$ such that $\psi (y)=1$ if $|y-p|<\epsilon$ and $\psi (y)=0$ if $|y-p|>2\epsilon$. Then 
$$P(x^ku)=P((1-\phi)x^ku)+P(\phi (1-\psi)x^k u)+P(\phi\psi x^k u).$$
First,since $(1-\phi)x^ku\in C^{\infty}_c(\overline{M})$, clearly by Proposition \ref{prop-mapproperty},
$$P((1-\phi)x^ku)(x,p)=O(x^{\lambda_1}).$$ 
Secondly, 
$$P(\phi (1-\psi)x^k u)(x,p)=\int_0^{2\epsilon}\int_{|y'-p|>\epsilon} 
\beta_{*}(\rho^{\lambda_1}\rho'^{\lambda_2}F(\rho,\rho',R,\omega)) \phi(x') (1-\psi(y'))x'^k u(x',y')dvol_g.$$
Since $\beta_{*}(\rho^{\lambda_1}\rho'^{\lambda_2}F(\rho,\rho',R,\omega))\phi(x') (1-\psi(y'))\in x^{\lambda_1}x'^{\lambda_2}C^{\infty}(\overline{M}\times\overline{M})$, clearly, 
$$P(\phi (1-\psi)x^k u)(x,p)=O(x^{\lambda_1}).$$
For the third part $P(\phi\psi x^k u)$, we can choose $f\in C^{\infty}(\partial \overline{M})$, suppoted in $\{y\in U: |y-p|<\epsilon\}$. Recall that $z=y-y'$, $R^2=x'^2+x^2+|z|^2$, $\omega=z/R$, $\rho=x/R$ and $\rho'=x'/R$.
\begin{eqnarray*}
\langle P(\phi \psi x^ku),f(y)dy\rangle &=& \langle \pi_{L*}(\rho^{\lambda_1}\rho'^{\lambda_2}F\pi_R^*(\phi (x')\psi (y') x'^ku(y'))), f(y)dy\rangle\\
&=& x^{\lambda_1}\int\int\int R^{-\lambda_1-n+k}\rho'^{\lambda_2-n+k}F
\phi(x')\psi(y') u(x',y')f(z+y') \frac{d\rho'}{\rho'}dzdy'.
\end{eqnarray*}
Let $v(x',y')=\phi(x')\psi(y') u(x',y')$. Note that for any $m\in\mathbb{N}$,
\begin{eqnarray*}
&&\Pi_{i=0}^{m}(x\partial_x-i)(R\partial_R+\lambda_1+n-k-i)[R^{-\lambda_1-n+k}
\rho'^{\lambda_2-n+k}F(R,\rho,\rho',\omega,y')v(\rho'R,y')]\\
&=& x^{m+1} R^{-\lambda_1-n+k}\rho'^{\lambda_2-n+k}(F_0+F_1\rho' \partial_xv (\rho'R,y')+\cdots+F_m\rho'^m\partial_x^mv(\rho'R,y')),
\end{eqnarray*}
where $F_i\in C^{\infty}(\overline{M}\times_0\overline{M})$ for $0\leq i\leq m$. Hence
\begin{eqnarray*}
&&\Pi_{i=0}^{m}(x\partial_x-i)(x\partial_x+\lambda_1-k-i)\langle x^{-\lambda_1}P(\phi \psi x^ku),f(y)dy\rangle\\
&=& x^{m+1} \int \int \int R^{-\lambda_1-n+k} \rho'^{\lambda_2-n+k} (\sum_{i=0}^m F_i\rho'^i\partial^i_xv (\rho'R,y')) f(z+y') \frac{d\rho'}{\rho'}dzdy', 
\end{eqnarray*}
for any $m\in\mathbb{N}$. This implies that for some $m'$ independent of $m$, 
$$\langle P(\phi \psi x^ku),f(y)dy\rangle = x^k\sum_{i=0}^m x^{i}\widetilde{G}_{i1}+x^{\lambda_1}\sum_{i=0}^mx^{i}\widetilde{G}_{i2}+O(x^{m+1+m'}),$$
where $\widetilde{G}_{i1},\widetilde{G}_{i2}\in C^{\infty}(\overline{M})$ for $0\leq i\leq m$. 
If $k<\lambda_1$, let $\theta=z/x$, $t=x'/x$ and $\rho_F=(1+t^2+|\theta|^2)^{-\frac{1}{2}}$.
\begin{eqnarray*}
\widetilde{G}_{10}&=&\lim_{x\rightarrow 0} \int_{\partial\overline{M}}  \int_{\mathbb{R}}\int_{\mathbb{R}^n} \rho_F^{\lambda_1+\lambda_2}t^{\lambda_2-n+k}F(x\rho_F^{-1},\rho_F,t\rho_F,\rho_F\theta,y') v(tx,y')f(x\theta+y')\frac{dt}{t}d\theta dy'\\
&=& ( \int_{\mathbb{R}}\int_{\mathbb{R}^n}  \rho_F^{\lambda_1+\lambda_2} t^{\lambda_2-n+k} F(0,\rho_F,t\rho_F,\rho_F\theta,y') \frac{dt}{t}d\theta )\langle v(0,y), f(y)dy\rangle.
\end{eqnarray*}
Note that $\rho_F$ is just the boundary defining function on the front face $F_p$ with respect to its boundary surface $B\cap F_p$. So above integral is actually an integration over $F_p$ which is well defined since $\lambda_1+\lambda_2>n$. If $\lambda_1>k$, then letting $R_0=R|_{x=0}$,
\begin{eqnarray*}
\widetilde{G}_{20} &=& \int\int\int R_0^{-\lambda_1+k}\rho'^{\lambda_2-n+k}F(R_0,0,\rho',\omega,y')
v(\rho' R_0,y')f(R_0\omega+y') \frac{d\rho'}{\rho'}\frac{dR_0}{R_0}d\omega dy'\\
&=& \langle (x^{-\lambda_1}P|_{x=0})(x^kv)(y), f(y)dy \rangle
\end{eqnarray*}
The two formulae can be extended meromorphically to all $\{\lambda\in\mathbb{C}: \lambda_2(\lambda)-n+k>0, \lambda_1(\lambda)+\lambda_2(\lambda)>n\}$ by analytic continuation. Adding the three parts together, then for $m$ arbitrarily large, 
$$P(x^ku) = x^k\sum_{i=0}^m x^{i}G_{i1}+x^{\lambda_1}\sum_{i=0}^mx^{i}G_{i2}+O(x^{m+1+m'})$$
where 
$$G_{10}=c(\lambda) u(0,y),\quad G_{20}=(x^{-\lambda_1}P|_{x=0})(x^ku)(y);$$
$$c(\lambda)=\int_{\mathbb{R}}\int_{\mathbb{R}^n}  \rho_F^{\lambda_1+\lambda_2} t^{\lambda_2-n+k} F(0,\rho_F,t\rho_F,\rho_F\theta,y') \frac{dt}{t}d\theta.$$
The distribution $R_0^{-\lambda_1+k-1}$ on 1-dimentional space is a homogeneous distribution defined meromorphically over $\lambda_1\in \mathbb{C}$ with poles of order $1$ at $\lambda_1=k+\mathbb{N}_0$. For $l\in\mathbb{N}_0$, the residue is
\begin{eqnarray*}
Res_{\lambda_1=k+l}(G_{20})
&=& \int\int \rho'^{\lambda_2-n+k} \partial^l_{R_0} (F(R_0,0,\rho',\omega,y) u(\rho' R_0,y-R_0\omega))|_{R_0=0} \frac{d\rho'}{\rho'}d\omega \\
&=& \int\int \rho'^{\lambda_2-n+k} (\sum_{l_1+l_2+|\alpha|=l} \rho'^{l_2}\omega^{\alpha} 
\partial_{R_0}^{l_1}F(0,0,\rho',\omega,y')\partial_x^{l_2}\partial_y^{\alpha}u(0,y)) \frac{d\rho'}{\rho'}d\omega\\
&=&  \sum_{l_2\leq l}T_{l-l_2} (\partial^{l_2}_{x}u)(0,y)
\end{eqnarray*}
where $T_{l-l_2}$ are differential operators of order $l-l_2$ on the boundary for $0\leq l_2\leq l$. Similarly $t^{\lambda_2-n+k-1}$ can be extended meromorphically to all $\lambda_2\in \mathbb{C}$ with poles of order $1$ at $\lambda_2=n-k-l$ for $l\in \mathbb{N}_0$. At $\lambda_2=n-k$, the residue is
$$Res_{\lambda_2=n-k}(G_10)=\int_{\mathbb{R}^n}  \rho_F^{\lambda_1+\lambda_2}  F(0,\rho_F,0,\rho_F\theta,y') d\theta. $$
Here $\rho_F=(1+|\theta|^2)^{-\frac{1}{2}}$.
\end{proof}
Let $\Pi_1 : C^{\infty}(\partial\overline{M}; S^2(^0T^*M)) \longrightarrow C^{\infty}(\partial\overline{M};\mathscr{V}^1)$ be the projection map.

\begin{proposition} Given $H_0, h_0, k_0$ as above. Then
$H=P_1(\Pi_1H_0)+Q_g(h_0\otimes k_0)$ where $Q_g$ is a quadratic term with kernel $q(m;\bar{y},\tilde{y}) \in C^{-\infty}(M\times\partial\overline{M}\times\partial\overline{M};Hom(\mathscr{V}^1\otimes \mathscr{V}^1,S^2(^0T^*M)))$, defined by
$$q(m;\bar{y},\tilde{y})= \int_M R_g(m;m')e(m';\bar{y},\tilde{y})dvol_g(m').$$
Moreover, in a neiborhood of the boundary,
$$q(x,y;\bar{y},\tilde{y})=\sum_{i=0}^{n}x^iq_i(y;\bar{y},\tilde{y})+O(x^{n+c}),$$
for some $c>0$, where $q_i(y;\bar{y},\tilde{y})\in C^{-\infty}((\partial\overline{M})^3;Hom(\mathscr{V}^1\otimes \mathscr{V}^1,S^2(^0T^*M))))$ with $$\mathrm{singsupp}(q_i)\subset \{y=\bar{y}\}\cup\{y=\tilde{y}\}\cup\{\bar{y}=\tilde{y}\}.$$
\end{proposition}
\begin{proof} 
For any $\gamma \in \dot{C}^{\infty}(M;S^2(^0T^*M))$, let $H'=R_g(\gamma)$. By Section \ref{sec-linear einstein op}, $H'=x^nH'_n+O(x^{n+c})$ for some $c>0$, where $H_n'\in C^{\infty}(\partial\overline{M};\mathscr{V}^1)$. Then
\begin{eqnarray*}
\langle H,\gamma\rangle_{M}&=&\int_M(HL_gH'-L_gHH')+\widetilde{E}(h,k)H'dvol_g\\
&=& n\int_{\partial\overline{M}}H_0H_n'dvol{g_0}+\int_M\widetilde{E}(h,k)R_g(\gamma) dvol_g\\
&=& \int_M(P_1(\Pi_1H_0)+R_g\widetilde{E}(h,k))\gamma dvol_g.
\end{eqnarray*}
Hence $Q_g(h_0\otimes k_0)=R_g(\widetilde{E}(P_1h_0,P_1k_0))$, whose kernel is
$$
q(m;\bar{y},\tilde{y}) = \int_M R_g(m;m')e(m';\bar{y},\tilde{y})dvol_g(m').
$$
Instead of dealing with the above integral directly, we introduce a parameter $\lambda\in\mathbb{C}$ as was done in Section \ref{sec-linear einstein op} and deal with
$$
q(\lambda,m;\bar{y},\tilde{y}) = \int_M R_g(\lambda,m;m')e(m';\bar{y},\tilde{y})dvol_g(m').
$$
for $|\lambda-n|<\varepsilon$ with $\varepsilon>0$ and small enough. 

Suppose $W_i$, $U_i$ are small open sets in  $\partial\overline{M}$ such that $W_i\subset\overline{W}_i\subset U_i$ for $i=1,2,3$ and $U_1,U_2,U_3$ are disjoint. Let $U_4=\partial\overline{M}-\cup_{i=1}^3 \overline{W}_i$ and $\{\psi_i\}_{i=1}^4$ is the partition of unity subject to the convering $\{U_i\}_{i=1}^4$ of $\partial\overline{M}$. Let $\phi(x)$ be a smooth cut off function such that for some $\epsilon>0$, $\phi (x)=1$ if $x\leq \epsilon$ and $\phi (x)=0$ if $x\geq 2 \epsilon$. Then for $(y,\bar{y},\tilde{y})\in U_1\times U_2\times U_3$ and $0<|\lambda-n|<\varepsilon$,
$$q(\lambda)=R_g(\lambda)((1-\phi)e)+\sum_{i=1}^4R_g(\lambda)(\phi\psi_i e).$$
First, by Lemma \ref{lemma-asympofe}, $(1-\phi)e\in C_c^{\infty}(M\times\partial\overline{M}\times \partial\overline{M})$. Clearly, 
$$R_g(\lambda)((1-\phi)e)=O(x^{\lambda}),$$
smoothly depend on $(y,\bar{y},\tilde{y})$. Secondly, $\phi\psi_1 e \in C^{\infty}(\overline{M}\times\partial\overline{M}\times \partial\overline{M})$. By a proof similar to that of Lemma \ref{lemma-asym}, we have 
$$R_g(\lambda)(\phi\psi_1 e)=G_{11}+x^{\lambda}G_{12}+O(x^{\lambda+c}),$$
where $G_{11},G_{12}\in C^{\infty}(\overline{M}\times\partial\overline{M}\times \partial\overline{M})$ and $c>0$. Moreover, 
$$G_{12}|_{x=0}=(x^{-\lambda}R_g(\lambda)|_{x=0})(\phi\psi_1 e).$$ 
Thirdly, for $i=2$ or $i=3$ or $i=4$, $R_g(\lambda)(x,y;x',y')\phi(x')\psi_i(y')\in \mathscr{A}^{E_T,E_B}(\overline{M}\times \overline{M})$. Hence 
$$R_g(\lambda)(\phi\psi_i e)=x^{\lambda}G_{i2} +O(x^{\lambda+c})$$
for some $G_{i2}\in C^{\infty}(\overline{M}\times\partial\overline{M}\times \partial\overline{M})$ and $c>0$. Moreover,
$$G_{i2}|_{x=0}=e((x^{-\lambda}R_g(\lambda)|_{x=0})\phi\psi_i).$$ 
Let $G_{11}=\sum_{i=0}^n q_i+O(x^{n+1}),$ and
$$q_{\lambda}=(x^{-\lambda}R_g(\lambda)|_{x=0})((1-\phi)e+\phi\psi_1 e)+ \sum_{i=2}^4e((x^{-\lambda}R_g(\lambda)|_{x=0})\phi\psi_i).$$
Then 
\begin{equation}\label{asympofQ}
q(\lambda)=\sum_{i=0}^n q_i+q_{\lambda}+O(x^{\lambda+c})
\end{equation} 
for some $c>0$. By the arbitrariness of $U_1, U_2, U_3$, 
\begin{equation}\label{singsupp}
\mathrm{singsupp}(q_i)\cup \mathrm{singsupp}(q_{\lambda}) 
\subset \{y=\bar{y}\}\cup\{y=\tilde{y}\}\cup\{\bar{y}=\tilde{y}\}.
\end{equation}
According to Lemma \ref{lemma-asym}, the results above can be extended meromorphically to $\{\lambda\in \mathbb{C}:|\lambda-n|<\varepsilon\}$ with a pole at $\lambda=n$. The residue at $\lambda=n$ corresponds to the $x^n\log x$ terms in $q(n)$. However, this can not happen. By paring with $h_0(\bar{y})\otimes k_0(\tilde{y})$ for $h_0,k_0\in C^{\infty}(\partial\overline{M};\mathscr{V}^1)$, 
$$Q_g(h_0\otimes k_0)=H-P_1(\Pi_1H_0).$$
Recall that the asymptotic expansion of neither $H$ nor $P_1(\Pi_1H_0)$ contains any $x^n\log x$ term, so the asymptotic expansion of $Q_g(h_0\otimes k_0)$ can not contain any $x^n\log x$ term, either. Therefore (\ref{asympofQ}) and (\ref{singsupp}) hold for all $|\lambda-n|<\varepsilon$.
\end{proof}

\begin{corallary}
$g_n^{(1,1)}=T_ng_0^{(1,1)}+R_n(g_0^{(0,1)},g_0^{(1,0)})$ where $T_n$ is a pseudodifferential operator of order $n$ and $R_n$ is quadratic form defined by kernel $$r_n(y;\bar{y},\tilde{y})\in C^{-\infty}((\partial\overline{M})^3;Hom(S^2(T^*\partial\overline{M})\otimes S^2(T^*\partial\overline{M}), S^2(T^*\partial\overline{M})))$$ with $\mathrm{singsupp}(r_n)\subset \{y=\bar{y}\}\cup\{y=\tilde{y}\}\cup\{\bar{y}=\tilde{y}\}.$
\end{corallary}

Finally, for the $N$ point problem with $N\geq 3$ and $\beta=(1,...,1)\in \mathbb{N}^N$, we can proceed in the same way as was for the $2$-point problem, i.e.
\begin{itemize}
\item[(1)] Fix a suitable gauge such that $\beta_g(\partial^{\alpha}g)=F(\partial^{\gamma}g;\gamma<\alpha)$ for all $\alpha\leq \beta$
such that the linearized equation is linear and elliptic with repect to $\partial^{\beta}g$;
\item[(2)] Use boundary paring formula and the resovent of $L_g$ to write out the kernel of nonliear Poisson operators and Scattering matrix.
\end{itemize}
All the information about singularities of the Dirichlet-to-Neumann map we can get is from the analysis of the kernel in the second step. However, the first step is a key to make the problem solvable.


\begin{thebibliography}{MM}
\bibitem[An1]{An1} M.T. Anderson, \textit{Boundary regularity, uniqueness,
and non-uniqueness for AH Einstein metrics on 4-manifolds}, Adv.
Math. \textbf{179}(2003), 205-249.
\bibitem[An2]{An2} M.T. Anderson, \textit{Einstein metrics with prescribed conformal
infinity on 4-manifolds}, math.DG/0105243,2001/2004.
\bibitem[An3]{An3} M.T. Anderson, \textit{Geometric aspects of the AdS/CFT
Correspondence}, in: Ads-CFT Correspondence: Einstein Metrics and
their Conformal Boundaries, IRMA Lectures in Mathematics and
Theoretical Physics \textbf{8}, European Mathematical Society
(2005), 59-71.
\bibitem[An4]{An4} M.T. Anderson, \textit{Some results on the structure of conformally compact Einstein
metrics}, arXiv: math. DG/0402198.
\bibitem[Ag]{Ag} S. Agmon, \textit{On the spectral theory of the Laplacian on non-compact
hyperbolic manifolds}, Journ\'{e}es \'{E}quations aux
d\'{e}riv\'{e}es partielles, (1987), 1-16.
\bibitem[CDLS]{CDLS} P.T. Chru\'{s}ciel, E. Delay, J.M. Lee, D.N.
Skinner, \textit{Boundary regularity of conformally compact
Einstein metrics}, J. Diff. Geom. \textbf{69}(2005), 111-136.
\bibitem[CL]{CL} E. Coddington, N. Levinson, \textit{Theory of ordinary differential
equations}, McGraw Hill, New York, (1955).
\bibitem[FG]{FG} C. Fefferman and C. R. Graham, \textit{The ambient metric}, 2007. arXiv:0710.0919.
\bibitem[Gr1]{Gr1} C.R. Graham, \textit{Dirichlet-to-Neumann map for Poincar¨¦-Einstein
metrics}, Oberwolfach Reports 2 (2005), 2200-2203
\bibitem[Gr2]{Gr2} C.R. Graham, \textit{Volume and area renormalizations for conformally compact Einstein
metrics}, Rend. Circ. Mat. Palermo, Ser.II, Suppl.
\textbf{63}(2000), 31-42.
\bibitem[GL]{GL} C.R. Graham and J.M. Lee, \textit{Einstein metrics with prescribed conformal infinity on the ball},
Advances in Math. \textbf{87}, (1991), 186-225.
\bibitem[GZ]{GZ} C. Graham, M. Zworski, \textit{Scattering matrix in conformal
geometry}, Invent. Math. \textbf{152} (2003), 89-118.
\bibitem[He]{He} D.W. Helliwell, \textit{Boundary regularity for conformally compact einstein metrics in
even dimemsions}, Ph.D. thesis, University of Washington, (2005).
\bibitem[JS]{JS} M. S. Joshi and a. S\'{a} Barreto, \textit{Inverse scattering on
asymptotically hyperbolic manifolds}, Acta Math. \textbf{184}
(2000),41-86.
\bibitem[Le]{Le} J. M. Lee, \textit{the spectrum of an asymptotically hyperbolic Einstein
manifold}, Comm. Anal. Geom. \textbf{3} (1995), no. 1778, 171-185.
\bibitem[Ma]{Ma} R. Mazzeo, \textit{the Hodge cohomology of a conformally
compact metric}, J. Diff. Geom, \textbf{28} (1988), 309-339.
\bibitem[MM]{MM} R. Mazzeo, R. Merose, \textit{Meromorphic exension
of the resolvent on complete spaces with asymptotically constant
negative curvature}, J. Funct. Anal. \textbf{75} (1987), 260-310.
\end{thebibliography}
\end{document}